  \documentclass[11pt]{article}
\usepackage[T1]{fontenc}
\usepackage[margin=1in, a4paper]{geometry}
\usepackage{amsthm, amsfonts,amsmath,amssymb,bbm, bbold}
\usepackage{mleftright}
\usepackage{verbatim}
\usepackage{textgreek}
\usepackage[]{authblk}
\usepackage{color}
\usepackage{hyperref}
\usepackage{etoolbox}
\makeatletter
\let\ams@starttoc\@starttoc
\makeatother
\usepackage[parfill]{parskip}
\makeatletter
\let\@starttoc\ams@starttoc
\patchcmd{\@starttoc}{\makeatletter}{\makeatletter\parskip\z@}{}{}
\makeatother

\usepackage{enumerate}
\usepackage{pdfsync}
\usepackage{color}

\numberwithin{equation}{section}

  \newtheorem{theorem}{Theorem}[section]
  \newtheorem{proposition}[theorem]{Proposition}
  \newtheorem{lemma}[theorem]{Lemma}
   \newtheorem{corollary}[theorem]{Corollary}

\theoremstyle{definition}
\newtheorem{definition}[theorem]{Definition}

\theoremstyle{remark}
\newtheorem{remark}[theorem]{Remark}

\let\d\delta

\newcommand{\lb}{\mleft(}
\newcommand{\rb}{\mright)}
\newcommand{\lbb}{\mleft [}
\newcommand{\rbb}{\mright ]}
\newcommand{\labs}{\mleft |}
\newcommand{\rabs}{\mright |}
\newcommand{\lbrb}[1]{\lb #1 \rb}
\newcommand{\lbbrbb}[1]{\lbb#1\rbb}
\newcommand{\labsrabs}[1]{\labs#1\rabs}
\newcommand{\lbbrb}[1]{\lbb#1\rb}
\newcommand{\lbrbb}[1]{\lb#1\rbb}

\newcommand{\lbcurly}{\mleft\{}
\newcommand{\rbcurly}{\mright\}}
\newcommand{\lbcurlyrbcurly}[1]{\lbcurly#1\rbcurly}
\newcommand{\lrfloor}[1]{\lfloor#1\rfloor}

\newcommand{\simi}{\stackrel{\infty}{\sim}}
\newcommand{\sima}{\stackrel{a}{\sim}}
\newcommand{\simo}{\stackrel{0}{\sim}}

\newcommand{\intervalOI}{\lbrb{0,\infty}}
\newcommand{\abs}[1]{\labs#1\rabs}

\newcommand{\curly}[1]{\lbcurly#1\rbcurly}

\newcommand{\bo}[1]{\mathrm{O}\lbrb{#1}}
\newcommand{\so}[1]{\mathrm{o}\lbrb{#1}}

\newcommand{\Pbb}[1]{\Pb\lb #1\rb}
\newcommand{\Ebb}[1]{\Eb\lbb #1\rbb}

\newcommand{\LL}{L\'{e}vy }
\newcommand{\BG}{Bernstein\allowbreak-gamma }
\newcommand{\LLP}{L\'{e}vy process }
\newcommand{\LLPs}{L\'{e}vy processes }
\newcommand{\LLK}{\LL\!\!-Khintchine }

\newcommand{\PiPlus}{\overline{\Pi}_+}
\newcommand{\PiMinus}{\overline{\Pi}_-}
\newcommand{\PiH}{\overline{\Pi}_\Hc}
\newcommand{\AH}{A_\Hc}

\newcommand{\Rez}{\Re\lbrb{z}}
\newcommand{\Imz}{\Im\lbrb{z}}

\newcommand{\limi}[1]{\lim_{#1\to \infty}}
\newcommand{\limsupi}[1]{\limsup_{#1\to \infty}}
\newcommand{\liminfi}[1]{\liminf_{#1\to \infty}}
\newcommand{\limo}[1]{\lim_{#1\to 0}}

\newcommand{\limsupo}[1]{\limsup_{#1\to 0}}

\newcommand{\Cb}{\mathbb{C}}
\newcommand{\C}{\mathbb{C}}
\newcommand{\Eb}{\mathbb{E}}

\newcommand{\N}{\mathbb{N}}
\newcommand{\Rb}{\mathbb{R}}
\newcommand{\R}{\mathbb{R}}

\newcommand{\Pb}{\mathbb{P}}

\newcommand{\Zb}{\mathbb{Z}}
\renewcommand{\P}{\mathbb{P}}

\newcommand{\Bc}{\mathcal{B}}

\newcommand{\Hc}{\mathcal{H}}

\newcommand{\Mcc}{\mathcal{M}}

\newcommand{\Wc}{\mathcal{W}}

\newcommand{\ind}[1]{\mathbbm{1}_{\lbcurlyrbcurly{#1}}}
\newcommand{\indset}[1]{\mathbbm{1}_{#1}}

\newcommand{\IntOI}{\int_{0}^{\infty}}
\newcommand{\IntOIc}{\int_{[0, \infty)}}
\newcommand{\IntOIo}{\int_{(0, \infty)}}

\newcommand{\Oiclosed}{\R_+}

\newcommand{\Wp}{W_\phi}
\newcommand{\Wkap}{W_\kappa}
\newcommand{\WkapP}{W_{\kappa_+}}
\newcommand{\WkapN}{W_{\kappa_-}}

\newcommand{\CbOI}{\Cb_{\intervalOI}}

\newcommand{\IPsiq}[1]{I_{\Psi_{#1}}}

\newcommand{\MPsiq}{\Mcc_{I_{\Psi_q}}}

\usepackage{authblk}

\newcommand{\D}{\mathrm{d}}
\newcommand{\phiplus}[1]{\phi_+({#1})}
\newcommand{\phiplusderiv}[2]{\phi_+^{{(#2)}}\lbrb{{#1}}}
\newcommand{\MBG}[2]{M_\Psi\lbrb{{#1}, {#2}}}

\newcommand{\partialder}[3]{\frac{\partial^{#1} }{\partial {#2}^{#1}}{#3}}
\makeatletter
\newcommand{\myitem}[1]{%
\item[#1]\protected@edef\@currentlabel{#1}%
}
\makeatother

\newcommand{\jump}[2]{#1^{(#2)}}
\usepackage{tikz}
\usetikzlibrary{decorations.markings}
\usetikzlibrary{patterns}
\DeclareMathSymbol{\mrq}{\mathord}{operators}{`'}
\usepackage{wrapfig}
\newcommand{\fakesubsection}[1]{%
  \par\refstepcounter{subsection}
  \subsectionmark{#1}
  \addcontentsline{toc}{subsection}{\protect\numberline{\thesubsection}#1}
}
\allowbreak

\newcommand{\minusone}{(-1)}

\renewcommand{\Im}{\mathtt{Im}}
\renewcommand{\Re}{\mathtt{Re}}

\let\f\frac

\begin{document}
	\date{}
\author[,1]{M. Minchev \thanks{
Email: mjminchev@fmi.uni-sofia.bg.}}
\author[,1,2]{M. Savov \thanks{ 
Email: msavov@fmi.uni-sofia.bg, mladensavov@math.bas.bg.}
}
	
\affil[1]{Faculty of Mathematics and Informatics, Sofia University "St. Kliment Ohridski", 5,
James Bourchier blvd., 1164 Sofia, Bulgaria}

\affil[2]{Institute of Mathematics and Informatics,  Bulgarian Academy of Sciences, Akad.  Georgi Bonchev str., 
	Block 8, Sofia 1113, Bulgaria}

\title{
Bivariate
\BG functions, potential measures, and asymptotics of exponential functionals of \LLPs}

	

\maketitle

\begin{abstract}
Let $\xi$ be a L\'{e}vy process and $I_\xi(t):=\int_{0}^te^{-\xi_s}\D s$, $t\geq 0,$ be the exponential functional of  L\'{e}vy processes  on deterministic horizon. Given that $\limi{t}\xi_t=-\infty$, we evaluate for general functions $F$ an upper bound on the rate of decay of $\Ebb{F(I_\xi(t))}$ based on an explicit integral criterion. When $\Ebb{\xi_1}\in\lbrb{-\infty,0}$ and  $\Pbb{\xi_1>t}$ is regularly varying of index $\alpha>1$ at infinity, we show that the law of $I_\xi(t)$, suitably normed and rescaled,  converges weakly to a probability measure stemming from a new generalisation of the product factorisation of classical exponential functionals. These results substantially improve upon the existing literature and are obtained via a novel combination between Mellin inversion of the Laplace transform of $\Ebb{I^{-a}_{\xi}(t)\ind{I_{\xi}(t)\leq x}}$, $a\in (0,1)$, $x\in(0,\infty],$ and Tauberian theory augmented for integer-valued $\alpha$ by a suitable application of the one-large jump principle in the context of the de Haan theory.  The methodology rests upon the representation of the aforementioned Mellin transform in terms of the recently introduced bivariate \BG functions for which we develop the following new results of independent interest (for  general $\xi$): we link these functions to the $q$-potentials of $\xi$; we show that their derivatives at zero are finite upon the finiteness of the aforementioned integral criterion; we offer neat estimates of those derivatives along complex lines. These results are useful in various applications of the exponential functionals themselves and in different contexts where properties of bivariate Bernstein-gamma functions are needed. $\xi$ need not be non-lattice.
\end{abstract}


  \tableofcontents

\tableofcontents
\section{Introduction and motivation}\label{sec:intro}
Let $\xi=\lbrb{\xi_s}_{s\geq0}$ be a one-dimensional \LL process, and for any $t\in\lbrb{0,\infty}$, set
\begin{equation}\label{def:exp}
    I_{\xi}(t):=\int_0^te^{-\xi_s} \D s
\end{equation}
for the exponential functional of \LLPs on deterministic horizon.
When $t=\infty$ or $t=\mathbf{e}_q$ for some $q>0$ and $\mathbf{e}_q\sim Exp(q)$ independent of $\xi$, we have the classical exponential functional of \LL processes.
Due to their role in different areas of probability theory, the latter objects have been extensively studied in the last thirty years, see \cite{ArRiv23,BerYor05,KuzPar12,Maulik-Zwart-06,MinSav23,PardoPatieSavov2012,Pat_2012,PatieSavov2012, PatieSavov2013, PatieSavov2018,Urban95,Yor01}
for a list of references covering the main results in the area and the different methods by which they have been obtained.
Compared to the breadth and depth of the results concerning  $I_\xi(\infty)$ and $I_\xi(\mathbf{e}_q)$, the knowledge about the distributional properties of $I_{\xi}(t)$ for fixed $t$ is quite meagre despite their widespread appearance in probability theory as well.
Its law has a closed but unwieldy form when $\xi$ is a Brownian motion with drift, see \cite{AlMatSh01,Yor92} and the survey \cite{MatYor05}.
The moments of $I_{\xi}(t)$, including complex ones,
are only available as a series for a class of subordinators, i.e. $\xi$ being a non-decreasing \LL process,
see \cite{BarkerSavov2021, PalSarSav23,SalVos18}.
In fact, \cite{SalVos18} considers stochastic processes with independent increments only,
but the explicit formulae are for the positive integer moments of $I_\xi(t)$ when $\xi$ is a subordinator.
Paper \cite{PalSarSav23} also offers some new convolutional identities for the moments of  $I_\xi(t)$  when $\xi$ is a general \LLP which yield that $\Ebb{I^{-1/2}_\xi(t)}$ is independent of $\xi$ when the latter is a symmetric \LL process, which is reminiscent of the invariance of the law of the first ladder time of symmetric random walks.
The most advanced knowledge for exponential functionals of \LLPs on  deterministic horizon is available for 
the asymptotic behaviour of distributional quantities of $I_\xi(t)$,
as $t\to\infty$, provided $I_\xi(\infty)=\infty$ almost surely.
Besides being of theoretical interest, these results have applications, especially in the area of branching processes and diffusions in \LL random environments, see e.g. \cite{BanParSm21},\cite[Section 2]{PalauPardoSmadi2016}
and \cite[Section 5]{Xu2021}.
Whilst postponing the detailed discussion of these results to Section
\ref{sec:previous results}, we highlight that our main contributions
on exponential functionals of \LLPs on deterministic 
horizon are in the case when $\xi$ drifts to minus infinity and therefore $I_\xi(\infty)=\infty$.
Under the additional assumptions that  $\Ebb{\xi_1}\in\lbrb{-\infty,0}$ and $\Pbb{\xi_1>t}$ is regularly varying at infinity with index $\alpha>1$,
we show that the measures 
\begin{equation}\label{eq:res}
    \frac{y^{-a}\Pbb{I_{\xi}(t)\in \D y}}{\Pbb{\xi_1>t}},\quad 
    a\in\lbrb{0,1},
\end{equation}
converge weakly, as $t\to\infty$,
to a finite positive measure supported by $\lbrb{0,\infty}$. Via Mellin inversion, we give a semi-explicit analytic description of the cumulative distribution function of the limit measure, 
which also yields the asymptotic behaviour of $\Ebb{I^{-a}_\xi(t)}$, see Theorem \ref{thm:main} and the ensuing corollaries. The latter allows us to derive a neat probabilistic representation of the normed limit measure which turns out to be the law of a curious generalisation of the main product factorisation for classical exponential functionals, see Theorem \ref{thm:mainPr}.
As a result, we improve substantially upon the recent work of W. Xu \cite{Xu2021},
which obtains results in the very same context, but under several additional assumptions, which do not seem to allow for the deduction of the aforementioned weak convergence, see 
Section \ref{sec:previous results} for a detailed discussion. Moreover, via the probabilistic factorisation, we obtain novel analytic properties for the limit measure such as existence and smoothness of its density. Also, dropping the requirement for regular variation of the upper tail and finiteness of the expectation, we obtain general upper bounds for the decay of quantities of the type $\Ebb{F(I_{\xi}(t))}$, as $t\to\infty$, which frequently suffice for applications such as some of those in \cite{BanParSm21,
PalauPardoSmadi2016,Xu2021}. The rate of decay depends on the finiteness of an explicit integral criterion, see Theorem \ref{thm:upperB}.  We also offer a very different methodology to the one in \cite{Xu2021}. It is based on a mixture of analytic and probabilistic techniques. We discuss it in more detail in Section \ref{sec: methodology} and touch upon it further below in this part.

 One of the main ingredients in the proof of the aforementioned results is the new information, which is of independent interest, that we obtain on the functional-analytic
 properties of the recently introduced in \cite{BarkerSavov2021} class of special functions called bivariate \BG functions.
 It is well-known that the bivariate \BG functions generated
 by the two Wiener-Hopf factors (Bernstein functions) of the underlying \LLP $\xi$
 represent the Mellin transform of the classical exponential functional, and therefore in turn the Laplace transform in $t$ of the exponential functional on deterministic horizon, see 
 \cite{BarkerSavov2021,PalSarSav23,PatieSavov2018}.
 For this reason, bivariate \BG functions with their functional-analytical properties have proved a key tool for
 the study of both classical exponential functionals and 
 exponential functionals on deterministic horizon, see e.g. \cite{BarkerSavov2021, PalSarSav23,PatieSavov2018}.
 In fact, being  a natural extension to the classical Gamma function in the context of right-continuous self-similar Markov processes,  these functions, predominantly through their role in the understanding of the exponential functionals of \LL processes,  are very useful in various other areas of probability,
 see e.g. for applications in the spectral theory of non-self-adjoint Markov semigroups \cite{PatSav17,PatieSavov2021,PatSav19}
 and \cite{LoePatSav19} for the study of extinction times in non-Markovian setting.
 In this work we show that bivariate \BG functions related to the Wiener-Hopf factors of transient \LLPs are $n$ times
 differentiable at zero in the first variable if and only if
an explicit integral criterion is finite, see Theorem \ref{thm:derW}. Moreover, we link the derivatives of these functions to the convolutions and the derivatives of the $q$-potential
 measures of the underlying \LL process, see Lemma \ref{lem:repW}, thereby 
 deriving some new results for as basic  
 objects as these potentials.
 This connection is not unexpected as classical exponential functionals, whose Mellin transform is represented in terms of \BG functions, are related to the potential measure in 
 the recent work \cite{ArRiv23}, which offers an integral equation for the density of the exponential functional in which the potential measure appears. When the potential measure has a bounded density, we obtain upper bounds for the densities of the convolutions of the potential measure, see Theorem \ref{thm:convo}. Using these results we get universal and seemingly quite handy estimates for the derivatives of the Mellin transforms of the classical exponential functionals and their decay along complex lines, see Corollary \ref{cor:der_Mq}.

The results on the derivatives of bivariate \BG functions play a crucial role in a convoluted Tauberian approach we adopt in order to derive the weak 
convergence of the measures defined in \eqref{eq:res}. The method  requires a particularly intricate analysis  when the regular 
variation index $\alpha$ of $\Pbb{\xi_1>t}$ is an integer. Then we have to appeal to the de Haan theory and in order to utilize
it successfully we have to resort to probabilistic considerations which complement the analytical approach.
The latter in turn reveals that one may have a more direct way to study the moments of the measure \eqref{eq:res} but at the cost of a loss of explicitness and precision.
Similar probabilistic approach has been adopted in the predecessor of our work, see \cite{Xu2021}. We discuss further our methodology in Section \ref{sec: methodology} 

We summarise the main outcomes of the manuscript. From probabilistic perspective, in the regular variation setting, we derive a scaling limit for the law of $I_\xi(t)$, together with fine asymptotic estimates in much broader generality. The scaling limit is identified a via novel factorisation. From analytical and special functions viewpoint, we link the derivatives of bivariate Bernstein–gamma functions to convolutions of potential measures, leading to new results and sharp estimates for these derivatives along complex lines. From potential theory perspective, we get bounds on densities of convolutions of potential measures. Methodologically, we combine Tauberian theorems, including de Haan’s theory, with the one-large-jump principle to overcome challenges, e.g. lack of monotone density theorem. Overall, our work affirms the central role of Bernstein–gamma functions, which encompass many classical special functions, in understanding exponential functionals which are in turn related to positive self-similar Markov processes. For this reason, these results may serve as a starting point for studying more general self-similar Markov processes by investigating matrix-valued extensions of bivariate Bernstein–gamma functions; see \cite{Doring-Trottner-Watson-2024, Kyprianou-Rivero-Sengul-Yang-2020} for related context.

The paper is organised as follows: Section \ref{sec:prelim} introduces the necessary notation, main
objects of interest and
previous results, outlines our
methodology, and offers a discussion on an open problem. The results on
finiteness of the
derivatives of bivariate \BG functions
are presented in
Section \ref{sec:W} and proved
in Section \ref{sec:proofs BG}.
In Section \ref{sec: proof thm main}
we prove our main result, Theorem
\ref{thm:main},
the factorisation
Theorem \ref{thm:mainPr}, and the general estimates of Theorem \ref{thm:upperB}. Section \ref{sec: auxiliary}  
provides the proofs of auxiliary results.
The \hyperref[sec:append]{Appendix} contains
classic facts about slowly varying functions and Tauberian theorems,
as well as it presents the de Haan class of
slowly varying functions.

\section{Preliminaries,  methodology, background, literature review, and an open problem}\label{sec:prelim}

\subsection{Notation}
We use $\Cb$ to denote the complex plane with $z$ standing for a complex variable with a real part $\Re(z)$ 
and an imaginary part $\Imz$, and
$\R_+$ for the set
of non-negative real numbers, $[0, \infty)$. For a set $A \subset \R$, we employ $ \Cb_A$ for the subset of the complex plane
$ \Cb_A := \lbcurlyrbcurly{z\in\Cb:\Re(z)\in A}$. Further, unless stated otherwise, $n, k, l$ are positive integers. For any non-negative random variable $X$, we denote by 
$\Mcc_X(z):=\Ebb{X^{z-1}}$ its Mellin transform, which is defined at least on the complex line 
$\Cb_{\lbcurlyrbcurly{1}} =\lbcurlyrbcurly{z\in\Cb:\Re(z)=1}$. If
$\mu$ is a measure, we define for
better readability $\mu(a,b) := \mu((a,b))$. Next, for a function $f$, we formally define its 
Laplace transform, for $q \geq 0$, as $\widehat{f}(q) := \int_0^\infty e^{-qt} f(t) \D t$.
For asymptotic behaviour, we use 
the classic
Landau notation and
also $f\sima g$ to denote $\lim_{x\to a}f(x)/g(x)=1$.
We recall that a function $\phi$ is a Bernstein function if and only if it can be represented as
\begin{equation}\label{eq:Bern}
    \phi(z)=\phi(0)+dz+\IntOI\lbrb{1-e^{-zy}}\mu( \D y),\quad  z\in \Cb_{\lbbrb{0,\infty}},
\end{equation}
where $\phi(0)\ge0$, $d\ge0$, and $\mu$ is positive and $\sigma$-finite, with $\int_0^\infty\min\{1,y\}\,\mu(\mathrm{d}y)<\infty$. Note that each Bernstein function is the Laplace exponent of a potentially killed subordinator with killing rate $\phi(0)$, linear drift $d$ and jump measure $\mu$. In the same fashion, a bivariate Bernstein function $\kappa$, given by, for
$ (\zeta,z)\in \Cb_{\lbbrb{0,\infty}}\times\Cb_{\lbbrb{0,\infty}}$,
\begin{equation}\label{eq:Bern2}
    \kappa(\zeta,z)=\kappa(0,0)+d_1\zeta+d_2z+\IntOI
    \IntOI
    \lbrb{1-e^{-\zeta y_1-zy_2}}\mu(\D y_1,\D y_2)
\end{equation}
is the Laplace exponent of a potentially killed bivariate subordinator, where $\kappa(0,0)$,
$d_1$ and $d_2$ are non-negative and have the same meaning as above, and $\mu$ is the bivariate jump measure, which satisfies $\IntOI\IntOI
\min\{1,  \sqrt{y_1^2 + y_2^2} \}\mu( \D y_1, \D y_2)<\infty$.

\subsection{Bivariate Bernstein-Gamma functions and Mellin transform of exponential functionals}
From \cite[Definition 2.5, Theorem 2.8]{BarkerSavov2021}, we know that for any given bivariate Laplace exponent $\kappa\not\equiv 0$,  the equation
\begin{equation}\label{eq:rec}
    \Wkap(\zeta,z+1)=\kappa(\zeta,z)\Wkap(\zeta,z),\quad
    \quad
 \Wkap(\zeta,1)=1, \zeta\in \Cb_{\lbbrb{0,\infty}}, z\in\CbOI,
\end{equation}
has a unique bivariate holomorphic solution in
$\CbOI\times\CbOI
$. Moreover, for any $q\geq 0$, the function 
$\Wkap(q,\cdot)$
is analytic on $\CbOI$, 
and it is a Mellin transform
of a positive random variable.  Note that due to \eqref{eq:rec} the bivariate \BG functions are 
a natural extension to  both the classical Gamma function with multiplier a general bivariate Bernstein function and the already well-known (univariate) Bernstein-gamma functions.
As mentioned in the introduction, these functions have arisen from the studies of non-continuous self-similar Markov processes.

From now on, we work with the case where the
bivariate \BG functions are based on the Wiener-Hopf factors of \LL processes.
Recall that a \LLP $\xi$ with a \LLK 
exponent $\Psi$ is 
killed at rate $q\geq 0$ if, for some independent of $\xi$ exponential random variable $\mathbf{e}_q\sim Exp(q)$, we set $\xi_t=\infty$ for $t\geq \mathbf{e}_q$ with $\mathbf{e}_0=\infty$ a.s. The \LLK exponent of the killed \LL process  is given by, for $z\in i\Rb$,
\begin{equation}\label{eq:LLK}
    \Psi_q(z)=\log_0\Ebb{e^{z\xi_1}}=\Psi(z)-q=\gamma z+\frac{\sigma^2}2z^2+\int_{-\infty}^{\infty}\lbrb{e^{zy}-1-zy\ind{|y|\leq 1}}\!\Pi( \D y)-q,
\end{equation}
where $\log_0$ is the principal branch of
the complex logarithm, $\gamma\in\Rb$
and $\sigma^2\geq 0$ are 
respectively the linear term and
the Brownian component of $\xi$,  $\Pi$ is the \LL measure of the process and satisfies
$\int_{-\infty}^{\infty}\min\{1, y^2\}\Pi(\D y)<\infty$,  and $q$ is the killing rate. For any potentially killed \LL process, we have the Wiener-Hopf factorisation
\begin{equation}\label{eq:WH}
	\begin{split}
		&\Psi_q(z)=\Psi(z)-q=-\phi_{+}(q,-z)\phi_{-}(q,z)=-h(q)\kappa_{+}\lbrb{q,-z}\kappa_{-}\lbrb{q,z},  \quad
  z \in i\R,
	\end{split}
\end{equation}
where, for $q \geq 0$,
\begin{equation}\label{eq:h}
	\begin{split}
		&h(q) := \exp
  \lbrb{-\int_{0}^{\infty}
  \lbrb{\frac{e^{-t} -
  e^{-qt}}{t}}\Pbb{\xi_t=0}\D t},
	\end{split}
\end{equation}
which in the case of a transient \LLP $\xi$
is non-zero by \cite[Theorem 12]{Bertoin96},  and, for $q\geq 0$ and $\Re(z)\geq 0$,
\begin{equation}\label{eq:kap}
	\begin{split}
		&\kappa_{\pm}(q,z)=
  c_{\pm}
\exp
\lbrb
{\IntOI\int_{\lbbrb{0,\infty}}\lbrb{
\frac{e^{-t}-e^{-qt-zx}}{t}}\Pbb{\pm\xi_t\in \D x}\D t}	\end{split}
\end{equation}
are the bivariate Laplace exponents 
of the ascending/descending ladder
times and ladder heights processes which form 
bivariate subordinators, see  
\cite[Corollary VI.10]{Bertoin96} 
and \cite[Section 4]{Doney2007} for more
information. The latter reference also includes expressions for the Lévy measures $\mu_{\pm}$ in the context of \eqref{eq:Bern} of $\kappa_\pm$. Moreover, it is  worth noting that
the factorisation \eqref{eq:WH} is indeed unique up to multiplication with a scalar as proved recently in full generality in \cite{Doring-Savov-Trottner-Watson-24}.

Further, the constants
$c_\pm > 0$ depend
on the choice of the local time which we can choose
appropriately so that
$c_+ c_- = 1$. Furthermore,
we can choose $\phi_{+}(q,z)=\kappa_{+}\lbrb{q,z}/c_+$ and $\phi_{-}(q,z)=h(q)\kappa_{-}\lbrb{q,z}/c_-$,
i.e.,
for $q\geq 0$ and $\Re(z)\geq 0$,
\begin{equation}\label{eq:phi_+}
	\begin{split}
		&\phi_{+}(q,z)=
\exp
\lbrb
{\IntOI\IntOIc\lbrb{
\frac{e^{-t}-e^{-qt-zx}}{t}}\Pbb{\xi_t\in \D x}\D t},\\
&\phi_{-}(q,z)=
\exp
\lbrb
{\IntOI\int_{\lbrb{0,\infty}}\lbrb{
\frac{e^{-t}-e^{-qt-zx}}{t}}\Pbb{\xi_t\in -\D x}\D t}.
	\end{split}
\end{equation}
We note that
$h(q)\equiv 1$ provided that $\xi$
is not a compound Poisson process (CPP),
and that in this case $h(q)\equiv 1$ and $\phi_\pm\equiv \kappa_\pm/c_\pm$ are 
bivariate Laplace exponents as well. 
When $h(q)\not\equiv 1$, for any 
fixed $q\geq 0$, $\phi_{q,\pm}
(z):=\phi_\pm(q,z)$ are Laplace 
exponents of univariate subordinators. 

Let us denote, for $q \geq 0$, $\IPsiq{q}:=I_{\xi}(\mathbf{e}_q)=\int_0^{e_q} e^{-\xi_s}\D s$. From \cite[Theorem 2.4]{PatieSavov2018} we know that,
at least for $z \in \Cb_{(0,1)}$,
\begin{equation}
\label{eq:M_I_psi}
   \MPsiq(z)=\phi_{-}(q, 0)\frac{\Gamma(z)}{W_{\phi_{q,+}}(z)}W_{\phi_{q,-}}(1-z)
    =: \phi_{-}(q, 0)\MBG{q}{z},
\end{equation}
where  for the fixed $q$,  the functions $W_{\phi_{q,\pm}}$ solve
\eqref{eq:rec}, and therefore they are  Bernstein-gamma functions.
To provide a link between the Bernstein-gamma functions
associated to the Wiener-Hopf factors $\phi_\pm$
and the probabilistic $\kappa_\pm$, we use that
by \cite[Theorem 4.1 (5)]{PatieSavov2018},
for any $c>0$, $z \in \Cb_{(0, \infty)}$, and a Bernstein function $\phi$, $W_{c\phi}(z)=c^{z-1}\Wp(z)$, 
so
we obtain from the choice of $\phi_{q,\pm}$ made above that
\begin{equation}\label{eq:Wk=Wp}
    W_{\phi_{q,+}}(z) = c_+^{1-z}\WkapP(q,z),
    \quad
    \text{and}
    \quad
    W_{\phi_{q,-}}(z)=c_-^{1-z} h^{z-1}(q)\WkapN(q,z).
\end{equation}
Substituting in \eqref{eq:M_I_psi}, we get 
\begin{equation}\label{eq:M_I_psi1}
    \MPsiq(z)=h^z(q)\kappa_-(q,0)\frac{\Gamma(z)}{W_{\kappa_+}(q,z)}W_{\kappa_{-}}(q, 1-z).
\end{equation}

\subsection{Methodology}
\label{sec: methodology} 
Recall that $I_\xi(t) := \int_0^t e^{-\xi_s} \D s$. From \eqref{eq:M_I_psi}, for $a\in\lbrb{0,1}$, we have that the Laplace transform for the moments in $t$ is given by
\begin{equation}\label{eq:LT}
    \begin{string}
     \IntOI e^{-qt}\Ebb{I_\xi^{-a}(t)}\D t=\frac{\MPsiq(1-a)}{q}=\frac{\MBG{q}{1-a}}{\phi_+(q,0)},
    \end{string}
\end{equation}
where we have utilised that from \eqref{eq:WH} it holds that $q=\phi_+(q,0)\phi_-(q,0)$, and we have employed \eqref{eq:M_I_psi}. Relation \eqref{eq:LT}  is the starting point of our proof as it opens the way to apply Tauberian theorems to determine the asymptotic behaviour of $\Ebb{I_\xi^{-a}(t)}$.
However, in the setting $\Ebb{\xi_1}<0$ from Theorem \ref{thm:upperB}, it is valid that
\[\IntOI\Ebb{I_\xi^{-a}(t)}\D t<\infty,\]
 and a direct application of Tauberian theorems is impossible. For this purpose, depending on the index of regular variation $\alpha>1$ of $\Pbb{\xi_1>t}$, we resort to investigating several repeated tails,
e.g. when $\alpha\in\lbrbb{n,n+1}$, $n\geq 1$, we consider  \[V(t)=\int_t^\infty \int_{s_1}^\infty \dots \int_{s_{n-1}}^ \infty 
  \Ebb{I_\xi^{-a}(s_n)} \D s_n \dots \D s_1 \]
  and via Tauberian methods we deduct the asymptotic behaviour of $V(t),$ as $t\to\infty$. When $\alpha\neq n+1$  by the monotone density theorem we infer the asymptotics of $\Ebb{I_\xi^{-a}(t)}$.
  Dealing with Laplace transforms of repeated tails requires the consideration of the derivatives in $q$ of $\MBG{q}{1-a}$, see \eqref{eq:LT}, including at $q=0$, which through \eqref{eq:M_I_psi} depend on the derivatives in $q$ of the involved bivariate \BG functions.
  Via an integral representation of the bivariate \BG functions which links these functions to the convolutions and the derivatives of the $q$-potential measures of $\xi$, we derive results of independent interest for both the aforementioned derivatives and  the potential measures. 

The outstanding case is when $\alpha=n+1$ is an integer. Then the results concerning the repeated tails can be deduced as above, but the monotone density theorem fails. Indeed, $\int_0^t V(s)\D s$ is slowly varying at infinity, see beneath \eqref{eq:def_g_an}, and Theorem \ref{thm:Monotone density, 1.7.2} does not hold. In this case, we are forced to resort to the more complicated de Haan theory, which requires a preliminary estimate of the type $\Ebb{I_\xi^{-a}(t)}=\bo{\Pbb{\xi_1>t}}$, and which we derive by purely probabilistic means such as for example the one-large jump principle. This allows us to use a generalised monotone density theorem, see Theorem \ref{thm:de Haan monotone density, 3.6.8}. The rest is similar to the non-integer  $\alpha$ case.

Having established the asymptotics of $\Ebb{I_\xi^{-a}(t)}$, we proceed to study the behaviour of \[\frac{\Ebb{I_\xi^{-a}(t)\ind{I_\xi(t)\leq x}}}{\Pbb{\xi_1>t}}\quad 
\text{for any fixed } x\in\lbrb{0,\infty}.\]
We do so via a Mellin inversion representation of the Laplace transform
\begin{equation*}
    \begin{split}
       x\mapsto \frac{1}{q}\Ebb{I_\xi^{-a}(\mathbf{e}_q)\ind{I_\xi(\mathbf{e}_q)\leq x}}
       &=\IntOI e^{-qt}\Ebb{I_\xi^{-a}(t)\ind{I_\xi(t)\leq x}} \D t\\
       &=-\frac{1}{2\pi i \phiplus{q,0}}\int_{\Re(z) = b}\frac{x^{-z}}{z}\MBG{q}{ z+1-a} \D z,
    \end{split}
\end{equation*}
and a subsequent investigation of $n$th repeated tails of $\Ebb{I_\xi^{-a}(t)\ind{I_\xi(t)\leq x}}$ as in the case above, which corresponds to the limit case $x=\infty$. Note that here the differentiation of the Laplace transform involves differentiation under the sign of the integral. This requires the additional understanding of the growth rate of the derivatives of bivariate \BG functions along complex contours which we develop in Theorem \ref{thm:derW} and Corollary \ref{cor:der_Mq}. The case when $\xi$ is a compound Poisson process with positive drift requires further refinement of the aforementioned estimates, see item (\ref{it:cor_der_Mq_iii}) of Corollary \ref{cor:der_Mq}. Then we are able to apply the same Tauberian, including the de Haan theory, arguments as in the case of the moments above. We emphasise that, for $\alpha\in\lbrbb{n,n+1}$, the main carrier of the asymptotics is the behaviour of $\lbrb{\phiplus{q,0}}^{(n)}$ at zero. As a result, we prove that 
\[\frac{\Ebb{I_\xi^{-a}(t)\ind{I_\xi(t)\leq x}}}{\Pbb{\xi_1>t}}\]
converges to the cumulative distribution function of a finite measure $\nu_a$ and compute its Mellin transform, see Theorems \ref{thm:main} and \ref{thm:mainPr}. The latter allows us to deduce analytic properties for the density of the limiting measure and  to obtain a product factorisation of the random variable behind the normed $\nu_a$, that is $\tilde{\nu}_a$, in terms of well-known independent random variables, namely the classical exponential functional of the subordinator pertaining to the Wiener-Hopf factor $\phi_{0,+}$ and the remainder term pertaining to $\phi_{0,-}$, see Theorem \ref{thm:mainPr}.

The bounds on the rate of decay of $\Ebb{\abs{F(I_{\xi}(t))}}$ in Theorem \ref{thm:upperB} follow straightforwardly from \eqref{eq:LT}, once the finiteness of the derivatives at $q=0$ on the right-hand side of \eqref{eq:LT} is established. This, in turn, ensures the finiteness of the derivatives of the integral on the left-hand side.

\subsection{Background and previous results}\label{subsec:back}
\label{sec:previous results}
We discuss the existing results on the asymptotic behaviour of exponential functionals of Lévy processes over deterministic horizons.

When $\Ebb{\xi_1}=0$, the \LLP oscillates and $I_\xi\lbrb{\infty}=\infty$. In this case, under the assumptions that $\Ebb{\xi^2_1}<\infty$ and the existence of some positive exponential moments for $\xi$, the authors of \cite{PalauPardoSmadi2016} show that $\limi{t}\sqrt{t}\Ebb{I^{-a}_{\xi}(t)}=C\in\lbrb{0,\infty}$. Furthermore, under the additional assumption for the existence of some negative exponential moments, they have obtained a result of the type $\limi{t}\sqrt{t}\Ebb{F\lbrb{I_{\xi}(t)}}=C_F$ for a limited class of functions $F$. Similar results in the same setting are obtained in \cite{LiXu18}. The requirements for the existence of exponential moments, for the finiteness of the variance, and for the limited form of $F$  have been completely removed in \cite[Theorem 2.20]{PatieSavov2018}. Therein even oscillating \LLPs $\xi$ with no first moment and with Spitzer's condition are studied in full generality. 

When $\xi$ is symmetric with infinite activity, it is shown in \cite[Corollary 2.3.]{PalSarSav23} that $\Ebb{I^{-1/2}_\xi(t)}=1/\sqrt{t}$. This is one of the few cases where a precise evaluation in a closed form of a fractional moment of $I_\xi(t)$ is available. It fits the asymptotics mentioned in the previous paragraph.

The case where $\Ebb{\xi_1}<0$ and $\xi$ drifts to minus infinity is considered in \cite{PalauPardoSmadi2016}. 
Under an assumption for the existence of negative exponential moments of some order, they show that the quantities considered in the paragraph above have explicit rate of decay. The case $\Ebb{\xi_1}<0$ and $\Pbb{\xi_1>t}$ being regularly varying at infinity of index $\alpha>1$, which is precisely our setting, is considered in \cite{Xu2021}. 
However, there the author assumes additionally that $\Pbb{\xi_1\in\lbbrb{t,t+\delta}}$ is regularly varying at infinity for any $\delta>0$, that is $\Pbb{\xi_1>t}$ is additionally \textit{locally} regularly varying. Then upon further assuming that $F$ is bounded, positive, non-increasing, locally-Lipschitz, and $\sup_{x>0}x^aF(x)<\infty$ for some $a\in\lbrb{0,1}$, see \cite[Conditions 4.1 and 4.2]{Xu2021}, then \eqref{eq:weak} holds true. 
We note that in our Corollary \ref{thm:main} we only assume $\sup_{x>0}x^aF(x)<\infty$ for some $a\in\lbrb{0,1}$ and $F$ being continuous. Since the result in \cite{Xu2021} is proved in a probabilistic way, the constant has a probabilistic representation, see \cite[(4.3)]{Xu2021}. We offer a characterisation of the limiting measure and the related constants in terms of special functions, which in many cases are easier to evaluate. Moreover, the Mellin inversion representation of the limiting measure yields analytical properties of this measure that cannot be captured by probabilistic means and offers in turn even a neat new probabilistic representation of the normed limiting measure as the law of the product of two independent random variables, which are well-known in the literature
and are easier to study. We do not discard lattice \LL processes either which would be necessitated by \textit{local} regular variation. We wish to highlight that \cite{Xu_23} obtains results similar to those of \cite{Xu2021} for exponential functionals of random walks. The respective conditions in \cite{Xu_23} match those in \cite{Xu2021}.

In all the aforementioned works that deal with asymptotic behaviour of the expectation $\Ebb{F(I_{\xi}(t))}$ there are results that explicitly or implicitly yield some bounds on the rate of decay of $\Ebb{F(I_{\xi}(t))}$. However, they require as a minimum the existence of exponential moments or fine structure of the tail $\Pbb{\xi_1>t}$ as in \cite{Xu2021}. Here, we offer such estimates on the rate of decay based on the finiteness of an explicit general integral criterion.

We are not aware of any existing results linking the derivatives of bivariate Bernstein–gamma functions to the convolutions and derivatives of the $q$-potential measures of the underlying Lévy process. We expect that this connection may prove useful in other contexts.

\subsection{Open problems}
A close inspection of our proof, see Section \ref{sec: proof thm main}, shows that $\Pbb{\xi_1>t}$ being regularly varying at infinity with index $\alpha>1$ implies that $\Pbb{\xi_t>0}$ is regularly varying with index $\alpha-1$, see \eqref{eq:asymp_xi_t_0}, which in turn yields that the derivative of the \LLK exponent of the ascending ladder time is regularly varying at zero, that is $\phi^{(n)}_+(q,0)$ is regularly varying of order $\curly{\alpha}-1$, see \eqref{eq:ln_kappa_case1}. A natural question is whether we can start only with the assumption of the behaviour of the latter derivative. In fact,  the only point where we need the regular variation of $\Pbb{\xi_1>t}$ is when we apply the de Haan theory via the probabilistic estimate, see the proof of Lemma \ref{lemma:bound_E}, which relies on the one large jump principle. Should the latter be avoided,
then one might prove results like ours only upon requirements on some derivative of $\phi_+(q,0)$. Due to the link to bivariate \BG functions, we believe this is a worthy line for future research which could further expand our knowledge not only on exponential functionals but also on bivariate \BG functions themselves too. 

The generalisation of the product factorisation of exponential functionals (see Theorem \ref{thm:mainPr}) raises the question of whether such a factorisation holds more universally, beyond the regular tail assumption. Classical factorisations are indeed universal, but our approach does not provide a path to answer this.

\section{Main results on exponential functionals}
\label{sec:main results}
We state the main result on $I_\xi(t)$, which is proven in Section \hyperref[proof thm 3.1]{\ref*{sec: proof thm main}}.
\begin{theorem} \label{thm:main}
	        Let $\xi$ be a \LL process with a finite negative mean and \begin{equation}\label{eq:reg_var_asymp}
\P\lbrb{\xi_1 > t} \simi  \frac{\ell(t)}{t^{\alpha}},
\quad
\text{or equivalently}
\quad
\Pi(t,\infty) \simi \frac{\ell(t)}{t^{\alpha}}
         \end{equation}
for $\alpha > 1$ and $\ell$ a slowly varying at infinity function.
	Then, for any $a \in (0,1)$,
\begin{equation} \label{eq:result_thm}
\frac{\mathbb{P}(I_\xi(t) \in \mathrm{d}y) t^{\alpha}}{y^a \ell(t)} \xrightarrow[t \to \infty]{w}\nu_a(\mathrm{d}y),
\end{equation}
where $\nu_a$ is a finite measure,
supported on $(0, \infty)$, with a distribution function given by
\begin{equation}
    \label{eq: nu_a}
\nu_a\lbrb{(0,x]} = 
\frac{1}{C}\int_{\Rez = b}\frac{x^{-z}}z \MBG{0}{ z+1-a} \D z,
\,\, \text{with}
\,\, C =-2 \pi i \phi_+(0,0)\lbrb{-\Ebb{\xi_1}}^{\alpha}\end{equation}
     and for all $b \in \Cb_{(a-1,0)}$.
	    \end{theorem}
     \begin{remark}\label{rem:main}
         We note that the proof of \cite[Theorem 2.20]{PatieSavov2018} which deals with the oscillating case has a minor gap regarding the tightness of the measures in \eqref{eq:result_thm}. It is amended precisely as the proof of this theorem.
     \end{remark}
    Relations of the statement of Theorem \ref{thm:main} to previous results can be found in subsection \ref{subsec:back}.
A direct consequence of the
last theorem are the following results.
\begin{corollary} \label{cor:limit functions}
     Under the conditions of Theorem
     \ref{thm:main}, for every  function $F : (0, \infty) \to \R$ such that,
     for some $a\in (0,1)$,
     $x \mapsto x^a F(x)$ is
     bounded and continuous,
     \begin{equation}\label{eq:weak}
         \frac{t^\alpha}{\ell(t)}
     \Ebb{F\lbrb{I_\xi(t)}}
     \xrightarrow[t \to \infty]{} \IntOIo y^a F(y) \nu_a(\D y)
     < \infty.
     \end{equation}
\end{corollary}

\begin{corollary} \label{cor:limit moments}
     Under the conditions of Theorem
     \ref{thm:main}, 
     for any $a\in (0,1)$,\begin{equation}\label{eq:moments}
         \frac{t^\alpha}{\ell(t)}
     \Ebb{I_\xi^{-a}(t)}
     \xrightarrow[t \to \infty]{} \frac{
     \MBG{0}{1-a}}{
     \phi_+(0,0)
     \lbrb{-\Ebb{\xi_1}}^{\alpha}}
     < \infty.
     \end{equation}
\end{corollary}
 Relations of the statements of Theorem \ref{thm:main} and Corollaries \ref{cor:limit functions}-\ref{cor:limit moments}  to previous results can be found in subsection \ref{subsec:back}.
 
Theorem \ref{thm:main} and Corollary \ref{cor:limit moments} lead to a neat probabilistic reformulation that extends the classical factorisation of exponential functionals; see Remark \ref{rem:mainPr} for further details. Following \cite[Theorem 2.22]{PatieSavov2018}, for a positive random variable $X$ and $a \in \mathbb{R}$ such that $\mathbb{E}[X^a] < \infty$, we define the size-biased transform by
\begin{equation}\label{eq:sizeBiased}
    \Ebb{f(\Bc_a X)}=\frac{\Ebb{X^af(X)}}{\Ebb{X^a}}
\end{equation}
for any bounded, continuous function $f$. Then we have the following factorisation, proven in
Section \hyperref[proof 3.5]{\ref*{sec: proof thm main}}.
\begin{theorem}\label{thm:mainPr}
   Under the conditions of Theorem \ref{thm:main}, for any $a\in\lbrb{0,1}$, 
   \begin{equation}\label{eq:mainPr}
       \Bc_{-a}I_{\xi}(t)\stackrel{d}{\xrightarrow[t \to \infty]{}}\Bc_{-a}I_{\phi_+}\times \Bc_{1-a}Y^{-1}_{\phi_-},
   \end{equation}
 where $\times$ stands for product of independent random variables, $I_{\phi_+}$ is the classical exponential functional pertaining to the killed subordinator with Bernstein function $\phi_+:=\phi_{0,+}$, and $Y_{\phi_-}$ is the positive random variable whose Mellin transform is $W_{\phi_-}$ with Bernstein function $\phi_-:=\phi_{0,-}$, see \eqref{eq:Wk=Wp}. 
 
 Moreover, the law of $\Bc_{-a}I_{\phi_+}\times \Bc_{1-a}Y^{-1}_{\phi_-}$, and therefore $\nu_a$, has infinitely differentiable bounded densities as long as $\xi$ is not a compound Poisson process with positive drift.
\end{theorem}
\begin{remark}\label{rem: factor}
  We prove Theorem \ref{thm:mainPr} by calculating explicitly the Mellin transforms of the respective quantities:
let $W$ be a random variable whose law is the normalised measure $\nu_a$, that is $F_W(x) := \nu_a((0,x])/\nu_a(0,\infty)$. In
  \eqref{eq:M_W} we establish that
\[
        \Mcc_{F_W}(z) =-\frac{1}{z}\frac{\Ebb{I^{z}_{\phi_+}I^{-a}_{\phi_+}}}{\Ebb{I^{-a}_{\phi_+}}}\frac{\Ebb{(Y^{-1}_{\phi_-})^{z}(Y^{-1}_{\phi_-})^{1-a}}}{\Ebb{(Y^{-1}_{\phi_-})^{1-a}}},
\]
and use it later to identify
$W\stackrel{d}{=}\Bc_{-a}I_{\phi_+}\times \Bc_{1-a}Y^{-1}_{\phi_-}$.
\end{remark}  
\begin{remark}\label{rem:mainPr}
   $Y_{\phi_-}$ is known as the remainder random variable and $\Bc_1 Y^{-1}_{\phi_-}$ is a factor in the general factorisation of exponential functionals, see \cite[Theorem 2.22]{PatieSavov2018}. This corresponds to $a=0$, which is not applicable here as the process $\xi$ drifts to $-\infty$. Nonetheless, if $\xi$ drifts to $\infty$, then $I_\Psi=I_{\xi}(\infty)=I_{\phi_+}\times  \Bc_{1}Y^{-1}_{\phi_-}<\infty$, see \cite[Theorem 2.22]{PatieSavov2018}, and clearly then \eqref{eq:mainPr} holds for all $a\in\lbbrb{0,1}$. In this sense, the result above is an extension of the factorisation for exponential functionals, which in this case is valid for the limit and only for $a\in (0,1)$. This raises the question whether a result as in \eqref{eq:mainPr} may hold beyond the assumptions of this theorem. Although, typically, for such factorisations, we cannot offer path-wise interpretation, to the best of knowledge we present some partial interpretation of $\Bc_{-a}I_{\phi_+}, \Bc_{1-a}Y^{-1}_{\phi_-}$. Assuming that $\psi_-(z):=z\phi_-(z)$ is the \LLK exponent of oscillating spectrally negative \LL process ($\phi_-(0)=0$), conditions for which exist e.g. in \cite[Chapter 9]{Doney2007}, then according to \cite[(1.8)]{Patie-2011}, $Y_{\phi_-}\stackrel{d}{=}J_{\psi_-}=I^{-1}_{\psi_{-,2}}$, where $J_{\psi_-}$ is the entrance law of the positive self-similar Markov processes associated to the \LLP pertaining to $\psi_-$ and $\psi_{-,2}$ is the $\mathcal{T}_1$ transform of $\psi_-$, see \cite[Proposition 2.4]{Patie-2011}. Also, in this case, under further mild conditions, $\Bc_{1-a}Y^{-1}_{\phi_-}=\Bc_{1-a}I_{\psi_{-,2}}\stackrel{d}{=}I_{\mathcal{T}_{1-a}\psi_{-,2}},$ where $\mathcal{T}_{1-a}\psi_{-,2}$ is the \LLK exponent of a \LL process, see \cite[Theorem 2.4 (2.8)]{PatieSavov2012}. Therefore, this component is related to an entrance law and exponential functional which is induced by the size-bias. Similar argument can be made for $\Bc_{-a}I_{\phi_+}$.
   Finally, we note that the Mellin transform of $\nu_a$ in Theorem \ref{thm:main} is proportional to $M_{\Psi}(z-a)$ and from there and the representation \eqref{eq:M_I_psi} combined with the results in \cite{PatieSavov2018} one can proceed to establish a number of functional-analytic properties of $\nu_a$. Due to the already lengthy manuscript, we leave this aside.
\end{remark}

The next result requires no general assumptions on $\xi$ beyond transience. Although it provides only an $\so{\cdot}$-type estimate, this is often sufficient for applications, as illustrated by results in the study of branching processes in Lévy random environments; see, for example, \cite{BanParSm21, PalauPardoSmadi2016, Xu2021}.
First, we introduce the quantity, for $x > 0$
and the linear term $\gamma$, defined in \eqref{eq:LLK},
\begin{equation}\label{eq:A}
    A(x):=\gamma+\Pi\lbrb{1,\infty}-
    \Pi\lbrb{-\infty,-1}+\int_{1}^x\lbrb{\Pi\lbrb{y,\infty}-\Pi\lbrb{-\infty,-y}} \D y.
\end{equation}
\begin{remark}
    \label{rem: DoneyMaller}
Let $T_{-x}$ denote the passage time below $-x<0$. The \textit{truncated mean} $A(x)$, introduced by Doney and Maller \cite{DonMal04}, provides equivalent conditions for the existence of moments of $T_{-x}$. The integral condition from \eqref{condi:upperB} or \eqref{eq:derW} originates from this framework and ensures the existence of $\mathbb{E}(T_{-x}^{n+1})$. Passage times play a crucial role in the estimates of the potential measures; see \eqref{eq:Ub}, \eqref{eq:sigmaInf}, or the compact inequality \eqref{eq:Ub1}, which yields
\[
U^{*n}([-x,0)) = O\left(\max\{1, \mathbb{E}(T_{-x}^n)\}\right).
\]
Naturally, $A(x)$ determining the finiteness of $\mathbb{E}(T_{-x}^{n+1})$, also plays a role in the study of exponential functionals. From a probabilistic perspective, large values of $I_\xi(t)$ occur when $\xi$ remains ``too low'' for a long time, keeping $e^{-\xi_s}$ large. Since $\lim_{t \to \infty} \xi_t = -\infty$, reaching very low levels ensures, roughly speaking, a sustained large value of $e^{-\xi_s}$. Thus, the behaviour of $I_\xi(t)$ is related to the speed of reaching negative levels, governed by $A(x)$. Inspection of the proof (see \eqref{eq:I_2+}) shows that the main contribution to $U^{*n}([-x,0))$ comes from the time $T_{-x}$, partially resembling a Laplace principle, as the whole interval $[0, T_{-x}]$ contributes to the maximum.

\end{remark}

Let us note some basic facts about $A(\cdot)$ for the case we consider:
from \cite[Lemma 13 $(ii)$]{DonMal04}, in the case where 
$\limi{t}\xi_t=-\infty$, if $\Ebb{|\xi_1| }< \infty$, then $\limi{x}A(x)$ $=\Ebb{\xi_1}<0$; and
if $\Ebb{\xi_1}$ is not finite, $\limi{x}A(x)=-\infty$.  Since $\limi{x}$ $\abs{A(x)}$ $=$ $\abs{A(\infty)}>0$, there
exists $x_0$ such that
$A(x)$ is non-zero for $x \geq x_0$.
Formally
it is possible that $A(x) = 0$ for some $x\in (1, x_0]$. However,
 the integral criterion \cite[(1.14)]{DonMal04},  requires integrability only at infinity, 
we may
 redefine $A(x)$ to be a positive constant
 over $(1, x_0]$. 
From
 now on
 we work as if $x_0 = 1$.
 
 We are ready to state our final result for this section.
See Section \hyperref[proof 3.8]{\ref*{sec: proof thm main}} for its proof.
\begin{theorem}\label{thm:upperB}
    Let $\xi$ be a \LLP such that $\limi{t}\xi_t=-\infty$. If
    \begin{equation}\label{condi:upperB}
        \int_{(1, \infty)} \lbrb{\frac{x}
{\abs{A(x)}}}^{n+1}\Pi(\D x)<\infty 
    \end{equation}
    for some $n\geq 1$, then, for every measurable function $F : (0, \infty) \to \R$ such that
     for some $a\in (0,1)$,
     $x \mapsto x^a F(x)$ is
     bounded, we have that at infinity
    \begin{equation}\label{res:upperB}
        \Ebb{\abs{F(I_{\xi}(t))}}=\so{t^{-n}},
    \end{equation}
    and $\IntOI t^n\Ebb{\abs{F(I_{\xi}(t))}} \D t<\infty$.
    Finally, assuming only that $\limi{t}\xi_t=-\infty$  and  $x \mapsto x^a F(x)$ is
     bounded, it holds that $\IntOI  \Ebb{\abs{F(I_{\xi}(t))}}\D t<\infty$.
\end{theorem}

\section{Derivatives of bivariate Bernstein--gamma functions}\label{sec:W}
In this section, we study the finiteness of the derivatives of 
$W_{\kappa_\pm}$ and
related quantities, which play a 
key role in the proofs in Section \ref{sec: proof thm main}. However, since  they are of independent interest, we state them in this separate section. We start 
with an integral representation of $W_{\kappa_\pm}$ recalling that, for $q \geq 0$,
\begin{equation}\label{def:pot}
    U_q(\D x)=\IntOI e^{-qt}\Pbb{\xi_t\in \D x} \D t
\end{equation}
are the $q$-potential measures of $\xi$ and $U^{*n}_{q}$ are their respective formal convolutions. $U:=U_0$ is called the potential measure which is a well-defined Radon measure when $\xi$ is transient.
\begin{lemma}\label{lem:repW}
   Let $\xi$ be a \LLP with Laplace exponent $\Psi$, and $\kappa_\pm$ be the bivariate Laplace exponents associated with $\Psi$, introduced in \eqref{eq:WH}. Then, for any $z\in\Cb_{(-1,\infty)}$ and $q > 0$,
   \begin{equation}\label{eq:repW}
       \begin{split}
           &\log\lbrb{W_{\kappa_\pm}(q,z+1)}=z\ln\kappa_\pm(q,1)+\int_{[0,\infty)}\lbrb{e^{-zy}-1-z(e^{-y}-1)}\frac{\Wc_{\pm}(q,\D y)}{e^y-1},
       \end{split}
   \end{equation}
   where $\log$ is a complex logarithm, and the measures $\Wc_{\pm}(q,\D y)$ do not charge $0$ and are defined on $\R\setminus\{0\}$ as
   \begin{equation}\label{eq:Wc}
       \Wc_{\pm}(q,\D y)=\IntOI \frac{e^{-qt}}{t}\Pbb{\xi_t\in\pm \D y}\D t.
   \end{equation}
Also, for any $z\in\Cb_{(-1,\infty)}$ and  $q > 0$, we have the representation
\begin{equation}\label{eq:repW2_1}
    \begin{split}
        \frac{\partialder{}{q}{W_{\kappa_\pm}(q,1+z)}}{W_{\kappa_\pm}(q,1+z)}
        =
        \IntOIc \frac{1- e^{-zy}}{e^y-1}U_q(\pm \D y).
    \end{split}
\end{equation}
  Finally, for any $q_0>0$, all derivatives of $ \Wc_{\pm}(q,\D y)$ in $q$ exist at $q_0$ and are equal to
   \begin{equation}
   \begin{split}
\label{eq:derWq}
       \Wc^{(n)}_{\pm}(q_0,\D y)&=\minusone^{n} \IntOI t^{n-1}e^{-q_0t}\Pbb{\xi_t\in \pm \D y}
       \D t\\
       &=
       \minusone^{n}(n-1)!U^{*n}_{q_0}(\pm \D y)
       \\
       &=-U^{(n-1)}_{q_0}(\pm \D y),
   \end{split}
   \end{equation}
   where $\lbrb{\partial^{n-1}/\partial q^{n-1}} U_{q_0}=U^{(n-1)}_{q_0}$ are the measure derivatives of $U_{q}$ at $q=q_0$.
\end{lemma}
    See  Section \hyperref[proof lem:repW]{\ref*{sec:proofs BG}} for the proof of the last lemma.
\begin{remark}\label{rem:repW}
Note that, given the integral representation in \cite[Theorem 4.7 (4.15)]{PatieSavov2018}, the new contribution in \eqref{eq:repW} lies in the representation \eqref{eq:Wc} in terms of the harmonic measure of the Lévy process $\xi$, and in the connection established in \eqref{eq:repW2_1} with the $q$-potential measures. Although this refinement is minor, it proves quite useful when working with $W_{\kappa_\pm}$, since the available input data are the Lévy process $\xi$ and its Laplace exponent.
 The differentiation of the
measures is in the sense of weak limits of
the respective difference ratios,
i.e. $\lbrb{\Wc_\pm\lbrb{q_0+h,\D y}-\Wc_\pm\lbrb{q_0, \D y}}/h$ 
as $h\to 0$.
We note that the measures in \eqref{eq:Wc} are
known as the $q$-harmonic measures of $\xi$. A
nice link between the harmonic measures and the
law of exponential functionals of subordinators
has been established in \cite{AliJedRiv14}. 
Here, it is clear that this link extends to
general exponential functionals of \LL 
processes.
\end{remark}
Lemma \ref{lem:repW} allows us to study the derivatives of $W_{\kappa_\pm}$.
We have the following result, which is proved in Section
\hyperref[proof 4.3]{\ref*{sec:proofs BG}}.
\begin{theorem}\label{thm:derW}
Let $\xi$ be a \LLP with \LLK 
exponent $\Psi$ such that
$\limi{t}\xi_t=-\infty$. Then,
for any $n\geq0$, $z \in 
\Cb_{(-1,\infty)}$, and $q\geq 0$,
\begin{equation}\label{eq:derW}
  \begin{split}
    &\int_{(1, \infty)} \lbrb{\frac{x}{\abs{A(x)}}}^{n+1}\Pi(\D x)<\infty \\
    &\qquad\iff \partialder{n+1}{q}{W_{\kappa_\pm}}(q,z+1)\text{ and }\partialder{n+1}{q}{W_{\phi_{q, \pm}}}(z+1)\text{ are finite},
  \end{split}
\end{equation}
where for $q=0$ the derivatives are understood as right derivatives. Upon the finiteness of the integral in \eqref{eq:derW}, the derivatives are right-continuous at $q=0$.
\end{theorem}

\begin{corollary} \label{cor:der_Mq}
Let $\xi$ be a \LLP with \LLK 
exponent $\Psi$ such that 
$\limi{t}\xi_t=-\infty$, and assume that,
\[
\MBG{q}{z} :=\frac{\Gamma(z)}{W_{\phi_{q,+}}(z)}W_{\phi_{q,-}}(1-z),
\]
and assume that, for some $n \geq 0$, 
\begin{equation*}
   \int_{(1,\infty)} \lbrb{\frac{x}{\abs{A(x)}}}^{n+1}\Pi(\D x)<\infty. 
\end{equation*}
Then, for any $0 \leq k \leq n+1$, $z
\in \Cb_{(0,1)}$, and $q \geq 0$:
\begin{enumerate}
    \item \label{it:cor_der_Mq_i}
    the derivatives $\partialder{k}{q}
{ \MBG{q}{z}}$ are finite and right-continuous at zero;
    \item \label{it:cor_der_Mq_ii}
    there exist polynomials $P_{\Re(z), k}$ of degree $k$ such that
    $$\labsrabs{\partialder{k}{q}
{ \MBG{q}{z}}} \leq P_{ \Re(z), k}
    \lbrb{\labsrabs{z}}\labsrabs{ \MBG{q}{z}};$$
    \item \label{it:cor_der_Mq_iii}
    if the potential measure of $\xi$ has a bounded density with respect
    to the Lebesgue measure, there exist polynomials
    $P_{\Re(z), k}$ of degree $k$ such that $$\labsrabs{\partialder{k}{q}
{ \MBG{q}{z}}} \leq P_{ \Re(z), k}
    \lbrb{\labsrabs{\ln |z|}}\labsrabs{ \MBG{q}{z}}.$$ 
\end{enumerate}
Moreover, if in addition $\Ebb{\xi_1}\in\lbrb{-\infty, 0}$, $n\geq 1$,  and for any $0< \beta<n$,
\begin{equation}\label{eq:R}
    \int_{(1,\infty)} x^{\beta+1}\Pi(\D x)<\infty=\int_{(1,\infty)} x^{n+1}\Pi(\D x),
\end{equation}
then 
\begin{enumerate}
\myitem{($1'$)}\label{it:cor_der_Mq_i'}
items ($\ref{it:cor_der_Mq_i}-(\ref{it:cor_der_Mq_iii})$ hold for any $0 \leq k\leq n$ and $q\geq 0$.
For $k>n$, 
the derivatives $\partialder{k}{q}
{ \MBG{q}{z}}$ exist for $q>0$.
\end{enumerate}
Further, for $k=n+1$, $q>0$, and $z
\in \Cb_{(0,1)}$, items (\ref{it:cor_der_Mq_ii})-(\ref{it:cor_der_Mq_iii}) are modified to
\begin{enumerate}
\myitem{($2'$)}\label{it:cor_der_Mq_ii'}
    for any $\delta>0$, there exist  polynomials $P_{\Re(z), n, \delta}$ of
    degree $n+1$ such that,  for all $q>0$,
    $$\labsrabs{\partialder{n+1}{q}
{ \MBG{q}{z}}} \leq \lbrb{U^{*n}_q(\Oiclosed) + q^{-\delta}}P_{ \Re(z), n, \delta}\lbrb{\labsrabs{z}}\labsrabs{ \MBG{q}{z}};$$
    \myitem{($3'$)} \label{it:cor_der_Mq_iii'}
    if the potential measure of $\xi$ has a bounded density with respect
    to the Lebesgue measure,  for any $\delta>0$, there exist  polynomials $P_{\Re(z), n, \delta}$ of degree $n+1$ such that, for all $q>0$,
    $$\labsrabs{\partialder{n+1}{q}
{ \MBG{q}{z}}} \leq \lbrb{U^{*n}_q(\Oiclosed) + q^{-\delta}}P_{ \Re(z), n, \delta}\lbrb{\labsrabs{\ln|z|}}\labsrabs{ \MBG{q}{z}}.$$
    \myitem{($4'$)} \label{it:cor_der_Mq_iv'} it holds true that $\limo{q}U^{*n}_q(\Oiclosed)=\infty$.
\end{enumerate}

\end{corollary}
The last result is proved in Section
\hyperref[proof 5.10]{\ref*{sec:proofs BG}}.
\begin{remark}\label{rem:derW}
We note that according to \cite[Theorem 1]{DonMal04},
\begin{equation}\label{eq:PandInt}
    \int_1^{\infty} t^{n-1}\Pbb{\xi_t\geq 0} \D t<\infty\iff \int_{(1, \infty)} \lbrb{\frac{x}{\abs{A(x)}}}^{n+1}\Pi(\D x)<\infty.
\end{equation}
Furthermore, from \eqref{eq:derWq}, \eqref{eq:U^*}, and \eqref{eq:ln_kappa}, we see that
\begin{equation}\label{eq:PandInt1}
\begin{split}
    &\int_{(1, \infty)} \lbrb{\frac{x}{\abs{A(x)}}}^{n+1}\Pi(\D x)<\infty\\
    &\qquad\iff U^{*n}\lbrb{\Rb_+}<\infty
    \iff \abs{U^{(n-1)}\lbrb{\Rb_+}}
    \iff \abs{\phi^{(n)}_+(0)}<\infty.
\end{split}
\end{equation}
In the case where $-\infty<\Ebb{\xi_1}<0$, as noted, we have that $\limi{x}A(x)=\Ebb{\xi_1}<0$, and the integral criterion from $\eqref{eq:derW}$ simplifies further to
\begin{equation}\label{eq:PandInt2}
\begin{split}
    \int_1^{\infty}t^{n-1}\Pbb{\xi_t\geq 0} \D t<\infty&\iff \int_{(1, \infty)}  x^{n+1}\Pi(\D x)<\infty\iff \abs{\phi^{(n)}_+(0)}<\infty\\
    &\iff \Ebb{\xi^{n+1}_1\ind{\xi_1>0}}<\infty\iff U^{*n}\lbrb{\Rb_+}<\infty,
\end{split}  
\end{equation}
where the equivalence before the last one is due to \cite[p.159]{Sato1999}. 
\end{remark}

The last result is of independent interest and concerns the densities of the potential measures $U_q$ under the condition of Corollary \ref{cor:der_Mq} and in the setting of item $(\ref{it:cor_der_Mq_iii})$.
We note that conditions for 
the existence of such densities can be found in \cite[Chapter I.3]{Bertoin96}, \cite[Chapter 12.5]{Kallenberg-17}, and \cite[Section 2]{Knopova-Schilling-13}.

\begin{theorem}\label{thm:convo}
    Let $\xi$ be a \LLP with \LLK 
    exponent $\Psi$ such that 
$\limi{t}\xi_t=-\infty$. Also, assume that $U$ has a bounded density $u$ on $\Rb$ with $C:=\sup_{x\in\Rb}u(x)$. Then, for any $k \geq 1$ and $q>0$, the potential densities of $U_q^{*k}$ exist and are locally bounded, since they satisfy, for any $x\in\Rb$ and $q>0$,
\begin{equation}\label{eq:convo1}
  u_q^{*k}(x) \leq kCU^{*(k-1)}_q\lbrb{\min\{0, x\}, \infty}
\end{equation}
with the convention $U^{*0}\equiv1$.
Next, if for some $n \geq 1$,
\begin{equation*}
   \int_{(1,\infty)} \lbrb{\frac{x}{\abs{A(x)}}}^{n+1}\Pi(\D x)<\infty, 
\end{equation*}
then, for any $1 \leq k\leq n+1$, we have that the potential densities of $U^{*k}$ exist and are locally bounded since they satisfy \eqref{eq:convo1} for any $x\in\Rb$, and also $U^{*l}(0,\infty)<\infty$ for any $1\leq l\leq n$. 
\end{theorem}
The theorem is proved in Section
\hyperref[proof 4.6]{\ref*{sec:proofs BG}}.
\begin{remark}\label{rmk:convo}To be precise, we note that in the last theorem we work with the canonical $q$-excessive versions of the respective densities, see \cite[Proposition I.12]{Bertoin96}. This allows us to state universal and not almost everywhere results.
The proof of \eqref{eq:convo1} can be easily modified to
improve the bound to
\[
  u_q^{*k}(x)\leq C\sum_{j=0}^{k-1}U^{*j}_q\lbrb{x,\infty}U^{*(k-1-j)}_q(0,\infty),
\]
but we use the former as it is more compact. Further, 
we note that Theorems \ref{thm:derW}
and
\ref{thm:convo}
    can be extended to include the case
    $\limi{t}\xi_t = \infty$:
    by considering the process
    $\widehat{\xi} = -\xi$, which 
    drifts to $-\infty$, and has $\widehat{\kappa}_\pm=
    \kappa_\mp$, $\widehat{A} = -A$,
    we see that the corresponding
    integral criterion would be
    \[
     \int_{(1, \infty)} \lbrb{\frac{x}
{\abs{A(x)}}}^{n+1}\Pi(-\D x)<\infty, 
    \]
    and in \eqref{eq:convo1},
    $U_q^{*j}\lbrb{\min\{0, x\}, \infty}$
    is changed to 
    $U_q^{*j}\lbrb{-\infty, \max\{0, x\}}$. Similarly, Corollary \ref{cor:der_Mq} can be adjusted accordingly for the case where $\xi$ drifts to infinity.
\end{remark}

\section{Proofs on Bernstein-gamma functions} \label{sec:proofs BG}
We proceed with the proofs for Section \ref{sec:W} by stating a preliminary proposition, which
links the derivatives of the $q$-potentials, see \eqref{def:pot}, 
with their convolutions. We have not been able to detect this result in the literature, but it is natural in view of the resolvent
equation and even clearer in the case of \LL processes.
\begin{proposition}\label{prop:conv}
Let $U$ be the potential measure of a \LL process. Then, for every $q\geq 0$ and $n\geq1$,
\begin{equation}\label{eq:conv}
    U_q^{*n} (\D x) =\frac{1}{(n-1)!}\minusone^{n-1}U^{(n-1)}_q(\D x)= \frac{1}{(n-1)!}\IntOI e^{-qt}t^{n-1}
\P(\xi_t \in \D x) \D t,
\end{equation}
with the measures necessarily finite for $q>0$ and possibly infinite for some or all $n$ when $q=0$.
\end{proposition}
We prove Proposition \ref{prop:conv}, as well as the next technical lemma, in Section \ref{sec: auxiliary}. 
\begin{lemma}\label{lem:uz}
For $z \in \Cb_{(-1, \infty)}$, the functions 
\begin{equation}\label{eq:uz}
  u_z(y):=  
  \frac{e^{-zy}-1-z(e^{-y}-1)}{e^y-1}
  \quad\text{and}
  \quad
  v_z(y):=  
  ze^{-y} -
  u_z(y)
  =
  \frac{1- e^{-zy}}{e^y-1}
\end{equation}
are bounded and continuous on $\lbbrb{0,\infty}$. Furthermore, there exist $C_{\Re(z)},$ $ C_{1, \Re(z)},$ $ C_{2, \Re(z)},$ $\epsilon_{1, \Re(z)}$, $\epsilon_{2, \Re(z)}  > 0$ such that, for all $y 
\geq 0$ and $x \geq C_{\Re(z)} $,
    \begin{equation}\label{eq:uz_bound}
        |v_z(y)| \leq C_{1, \Re(z)}|z|e^{-\epsilon_{1, \Re(z)}y}
        \quad
        \text{and}
        \quad
        |v_z(x)| \leq C_{2, \Re(z)}e^{-\epsilon_{2, \Re(z)}x}.
    \end{equation}
\end{lemma}
In the next theorem we also outline some basic facts about
potential measures, which can be found in
\cite{Revuz84}. Its first part is
\cite[p. 101, Corrolary 3.3.6]{Revuz84},
and the second one is a
combination of \cite[p. 169, Theorem 5.3.1]{Revuz84},
\cite[p. 171, Theorem 5.3.4]{Revuz84}, and
 \cite[p. 173, Theorem 5.3.8]{Revuz84}.
\begin{theorem}\label{thm:Revuz}
Let $\xi$ be a \LLP which drifts to $-\infty$.
\begin{enumerate}
    \item 
    The measure $U$ is a Radon measures, i.e., for every compact $K \subset \R$,
    $U(K) < \infty$.
\item (renewal theorem)
If $\xi$ is non-lattice, then  it holds that 
vaguely\\
${\limi{x}U(x-\D y)=0 \D y}$, and $\lim_{x \to -\infty}U(x-\D y)=- \D 
y/\Ebb{\xi_1}$ with the convention $1/\infty = 0$.
In the case of a lattice process, $\D y$ should be replaced by the
counting measure on the respective lattice.
\end{enumerate}
\end{theorem}
We start with the proof of Lemma \ref{lem:repW}.
\begin{proof}[Proof of Lemma \ref{lem:repW}]\label{proof lem:repW}
Recall that $\phi_{q,\pm}(z):=\phi_\pm(q,z)$. From \cite[Theorem 4.7 (4.15)]{PatieSavov2018}, we know that, for
$z \in \Cb_{(-1,\infty)}$,
\begin{equation}\label{eq:logWW}
     \log_0\lbrb{W_{\phi_{q,\pm}}(1+z)}
    =
    z\ln\phi_{q,\pm}(1)+\IntOIc
    \lbrb{e^{-zy}-1-z(e^{-y}-1)}\frac{k_\pm(q,\D y)}{y(e^y-1)},
\end{equation}
where the measures $k_\pm(q,\D y)$ are defined in \cite[Theorem 4.7 (4.15)]{PatieSavov2018} for any $q 
> 0$, and, from \cite[(5.37)]{PatieSavov2018}, 
have Laplace transforms, for $\lambda > 0$,
\begin{equation*}
    \begin{split}
        \IntOIc e^{-\lambda y}k_\pm(q,\D y)=\frac{\partialder{}{\lambda}{\phi_{\pm}(q, \lambda)}}
        {\phi_{\pm}(q, \lambda)}.
    \end{split}
\end{equation*}
On the other hand,
for $q> 0$ and $\lambda >  0$, by construction, $\phi_{+}(q, \lambda) = \kappa_{+}\lbrb{q, \lambda}/c_+$; $\phi_{-}(q,
\lambda)$ $=$ $ h(q) \kappa_{-}\lbrb{q, \lambda}/c_-$, and differentiating with respect to (w.r.t.) $z $ the logarithm of \eqref{eq:kap}, we get
\begin{align*}
\frac{\partialder{}{\lambda}{\phi_{\pm}(q, \lambda)}}
        {\phi_{\pm}(q, \lambda)}
    =
\frac{\partialder{}{\lambda}{\kappa_{\pm}(q, \lambda)}}
        {\kappa_{\pm}(q, \lambda)}
    &=
    \IntOI \IntOIc e^{-qt}ye^{-\lambda y}\frac{\Pbb{\xi_t\in \pm \D y}}{t}\D t
    \\
    &=
    \IntOIc e^{-\lambda y}y\IntOI e^{-qt}\frac{\Pbb{\xi_t\in\pm \D y}}{t}\D t.
\end{align*}
Identifying the two Laplace transforms above, we arrive at
\begin{equation}\label{eq:W+-}
  k_\pm(q,\D y)
  =
  y\IntOI e^{-qt}\frac{\Pbb{\xi_t\in\pm \D y}}{t}\D t
  =:
  y\Wc_\pm\lbrb{q, \D y},  
\end{equation}
as defined in
\eqref{eq:Wc} of the lemma itself. Plugging this into
\eqref{eq:logWW} and using the relations \eqref{eq:Wk=Wp},
we demonstrate the validity of \eqref{eq:repW}.
Next, we prove \eqref{eq:derWq} noting that we understand
the derivatives as measures,
which are the weak limit of the respective difference ratios, i.e. $(\Wc_\pm\lbrb{q_0+h,\D y}$
$-$ $\Wc_\pm\lbrb{q_0,\D y})/h$.
From \eqref{eq:W+-},
it is immediately seen that, for $q>0$ and any
$n\geq 0$, by monotone convergence,
\begin{equation}
    \label{eq: deriv Wc}
\Wc^{(n)}_\pm\lbrb{q,\D y}=(-1)^n\IntOI t^{n-1} e^{-qt}
\Pbb{\xi_t\in\pm \D y}\D t=-\partialder{n-1}{q}{U^{(n-1)}_{q}}(\pm \D y),
\end{equation}
where the last expression holds by the very definition of the $q$-potential measure, see \eqref{def:pot}.
Using \eqref{eq:conv} of Proposition \ref{prop:conv}, we derive all identities in \eqref{eq:derWq}. Finally,
we prove \eqref{eq:repW2_1}: first we note that from
\cite[Theorem 4.1]{PatieSavov2018}, for each $q>0$, $W_{\phi_{q,\pm}}$ are analytic and zero-free on $\Cb_{(0,\infty)}$, and according to \cite[Theorem 2.8]{BarkerSavov2021},  
$W_{\kappa_\pm}(q,z) $ are analytic on $\CbOI\times\CbOI$. Using the defined in
 Lemma \ref{lem:uz}
 bounded and continuous
 function $u_z$,
see \eqref{eq:uz}, by differentiating \eqref{eq:logWW} and using \eqref{eq:W+-}, we get that, for $q >0$ and $z \in \Cb_{(-1,\infty)}$,
\begin{equation} \label{eq: deriv W_phi}
    \begin{split}
        \frac{\partialder{}{q}{W_{\phi_{q,\pm}}(1+z)}}{W_{\phi_{q,\pm}}(1+z)}
        =
        z\frac{\partialder{}{q}{\phi_{q,\pm}}(1)}{\phi_{q,\pm}(1)}+\int_{[0, \infty)} u_z(y)\Wc'_{\pm}(q,\D y).
    \end{split}
\end{equation}
Substituting the relations \eqref{eq:Wk=Wp}, we get that
\begin{equation}\label{eq:repW1}
    \begin{split}
        \frac{\partialder{}{q}{W_{\kappa_\pm}(q,1+z)}}{W_{\kappa_\pm}(q,1+z)}
        =
        z\frac{\partialder{}{q}{\kappa_\pm}(q,1)}{\kappa_\pm(q,1)}+\int_{[0, \infty)} u_z(y)\Wc'_{\pm}(q,\D y).
    \end{split}
\end{equation}
For the first term in the sum above, employing \eqref{eq:kap}, we arrive at 
\begin{equation}\label{eq:logK}
   \frac{\partialder{}{q}{\kappa_\pm}(q,1)}{\kappa_\pm(q,1)}
   =
   \IntOI\int_{[0, \infty)} e^{-qt}e^{-y}\Pbb{\xi_t\in \pm \D y}\D t=
    \int_{[0, \infty)}e^{-y}U_q(\pm \D y).
\end{equation}
The finiteness of the last quantity for $U_0 = U$ is guaranteed by the vague convergence stated in Theorem \ref{thm:Revuz},
and also for $q > 0$, because in the sense
of measures $U_q(\pm \D y)\leq U(\pm \D y)$, see \eqref{def:pot}. Also, from \eqref{eq: deriv Wc}, we see that $\Wc'_{\pm}(q,\D y)=-U_q(\pm \D y)$,
so once again by the renewal Theorem \ref{thm:Revuz} and the exponential decay
of $u_z(y)$ as $y\to \infty$, see Lemma \ref{lem:uz}, we conclude that
\begin{equation*}
    \IntOIc u_z(y)\Wc'_{\pm}(q,\D y)=-\IntOIc u_z(y)U_q(\pm \D y)
\end{equation*}
is well-defined for $q\geq 0$. Plugging the last expression and \eqref{eq:logK} into \eqref{eq:repW1}, we arrive at 
\begin{equation}\label{eq:deriv W_kappa end of proof}
    \begin{split}
        \frac{\partialder{}{q}{W_{\kappa_\pm}(q,1+z)}}{W_{\kappa_\pm}(q,1+z)}
        =
        \IntOIc \lbrb{ze^{-y}-u_z(y)}U_q(\pm \D y)
        = 
        \IntOIc  \frac{1- e^{-zy}}{e^y-1} U_q(\pm \D y),
    \end{split}
\end{equation}
which is \eqref{eq:repW2_1}.
This concludes the proof of the lemma. 
\end{proof}
Now we proceed with the proof of the main result of Section
\ref{sec:W}, Theorem \ref{thm:derW}.
\begin{proof}[Proof of Theorem \ref{thm:derW}]\label{proof 4.3}
We first prove the forward direction
for the derivatives
of
$W_{\kappa_\pm}$.

Let us recall again that from \cite[Theorem 2.8]{BarkerSavov2021},  
$W_{\kappa_\pm}(q,z) $ are analytic on $\CbOI\times\CbOI$,
and thus all derivatives of $W_{\kappa_\pm}(q,1+z)$
exist at any $q>0$ and any $z \in \Cb_{(-1,\infty)}$.
For each such pair $\lbrb{q,z}$, using the defined in
 Lemma \ref{lem:uz}
 bounded and continuous
 function $v_z$,
see \eqref{eq:uz},
\eqref{eq:repW2_1} of Lemma \ref{lem:repW} gives us
\begin{equation}\label{eq:repW2}
    \begin{split}
        \frac{\partialder{}{q}{W_{\kappa_\pm}(q,1+z)}}{W_{\kappa_\pm}(q,1+z)}
        =
        \IntOIc \frac{1- e^{-zy}}{e^y-1}U_q(\pm \D y)
        = : 
        \IntOIc v_z(y) U_q(\pm \D y).
    \end{split}
\end{equation}
From \eqref{eq:uz_bound}, we have that for any fixed $z$ such that $\Re(z)>-1$, there exist
$\epsilon:=\epsilon_{\Re(z)}>0$ and $C:=C_z > 0$ such that 
\begin{equation}\label{eq:vzB}
    \abs{v_z(y)}\leq Ce^{-\epsilon y}. 
    \end{equation}
From \cite[Theorem 4.1 (1)]{PatieSavov2018}, it is known that $W_{\kappa_\pm}(q,z)$ are zero free on $q\geq 0$, $\Re(z)>0$.    Therefore, for $q > 0$, we can use the dominated convergence theorem to get from \eqref{eq:repW2}
    \begin{equation}\label{eq:repW3}
    \begin{split}
       \partialder{n}{q}{} \frac{\partialder{}{q}{W_{\kappa_\pm}(q,1+z)}}{W_{\kappa_\pm}(q,1+z)}&=\frac{\partial^n}{\partial q^n}\IntOIc v_z(y) \IntOI e^{-qt}\Pbb{\xi_t\in \pm \D y}\D t\\
       &=\IntOIc v_z(y)U^{(n)}_q(\pm \D y).
    \end{split}
\end{equation}
Taking limit as $q\to 0$, from \eqref{eq:vzB}
and using $|U_q^{(n)}(\pm \D y)|
\leq |U^{(n)}(\pm \D y)|$ in the sense of measures, we observe the two implications
\begin{equation}\label{eq:implication U_derivative}
\begin{split}
    \IntOIc e^{-\epsilon y}\abs{U^{(n)}(\pm \D y)}<\infty  &\implies 
  \labsrabs{\partialder{n+1}{q}{W_{\kappa_\pm}(0,1+z)}}<\infty\\ 
  &\implies\limo{q}\partialder{n+1}{q}{W_{\kappa_\pm}(q,1+z)}=\partialder{n+1}{q}{W_{\kappa_\pm}\!(0,1+z)}.
\end{split}
\end{equation}
However, from \eqref{eq:derWq}, we know that $n$th derivative of $U$ is a constant multiple of
the $(n+1)$-st convolution of $U$, so to obtain the finiteness of the considered derivatives of $W_{\kappa_\pm}$ at $0$, it
would be enough to prove that 
\begin{equation}\label{eq:repW4}
 \IntOIc e^{-\epsilon y}U^{*k}(\pm \D y)<\infty, \quad
1 \leq k \leq n+1.
\end{equation}
By  \eqref{eq:derWq}, 
\[
U^{*k}(\pm \D y) = 
\IntOI t^{k-1}\P(\xi_t \in
\pm \D y) \D t,
\]
so
if we check
\eqref{eq:repW4} for
$k = n+1$, this would
imply the result for all $k \leq n$. If $n = 0$, 
\eqref{eq:repW4} follows by the same
application of the renewal theorem as for \eqref{eq:logK}, since
Theorem \ref{thm:Revuz} holds. Assume now that $n \geq 1$.
From the assumption \eqref{eq:derW}, 
its equivalent statement \eqref{eq:PandInt},  and
\eqref{eq:derWq}, we see that
\begin{equation}\label{eq:U^*}
\begin{split}
    &\int_{(1, \infty)}\lbrb{\frac{x}{\abs{A(x)}}}^{n+1}\Pi(\D x)<\infty\\
    &\qquad\iff U^{*n}\lbrb{\Rb_+}
    =\frac{1}{(n-1)!}\IntOI t^{n-1}\Pbb{\xi_t\geq 0}\D t<\infty.
\end{split}
\end{equation}
Plugging the expression
\begin{equation}\label{eq:Un+1}
    U^{*(n+1)}(\pm \D y)=\int_{-\infty}^{\infty}U(\pm \D y-x)U^{*n}( \D x)
\end{equation}
into \eqref{eq:repW4} for $k = n+1$, we derive 
\begin{equation}\label{eq:repW5}
\begin{split}
    &\int_{[0,\infty)} e^{-\epsilon y}\int_{-\infty}^{\infty}U(\pm \D y-x)U^{*n}( \D x)
    =
    \int_{-\infty}^{\infty}\IntOIc e^{-\epsilon y}U(\pm \D y-x)U^{*n}(\D x)\\
    &\,=\IntOIc\IntOIc e^{-\epsilon y}U(\pm \D y-x)U^{*n}(\D x)
    +
    \int_{(-\infty,0)}\IntOIc e^{-\epsilon y}U(\pm \D y-x)U^{*n}(\D x)\\
    &\,=: I^\pm_1+I^\pm_2.
    \end{split}
\end{equation}
To establish \eqref{eq:repW4}, we will check that $I^\pm_1$ and $I^\pm_2$ are finite. 
For $I^\pm_1$, we have that
\begin{equation*}
    \begin{split}
        I^\pm_1 &=
        \IntOIc \sum_{m\geq 0} 
        \int_{[m, m+1)}
        e^{-\epsilon y
        }U(\pm \D y-x)U^{*n}(\D x) 
        \\
        &\leq 
        \sum_{m\geq 0}e^{-\epsilon m}
        \IntOIc\int_{[m, m+1)}U(\pm \D y-x)U^{*n}(\D x)
        \\
        &\leq 
        \sum_{m\geq 0}e^{-\epsilon m}
        \IntOIc U{\lbrb{[\pm m -x,\pm(m+1)-x)}}U^{*n}(\D x) \\
        &\leq \sup_{v\in\Rb}U\lbrb{[v,v+1)}U^{*n}\lbrb{\Rb^+}\frac{1}{1-e^{-\epsilon}}.
    \end{split}
\end{equation*}
Therefore, from \eqref{eq:U^*} and 
the finiteness of the last supremum by
Theorem \ref{thm:Revuz}, 
we conclude that $I^\pm_1<\infty$.
We proceed to study $I^\pm_2$ for which we need to introduce some more notation:  
let, for $x>0$,
\[T_{-x}:=\inf\curly{t\geq 0: \xi_t<-x},\]
which is a.s. finite because $\limi{t}\xi_t=-\infty$.
We thus obtain that
\begin{equation}\label{eq:Ub}
\begin{split}
     U^{*n}\lbrb{[-x,0)} 
     &= 
     \IntOI t^{n-1}\Pbb{\xi_t\in\lbbrb{-x,0}}
     \D t\\
     &     \leq \Ebb{T^{n}_{-x}}+\IntOI 
     \Ebb{(t+T_{-x})^{n-1}}\Pbb{\xi_t\geq 0}\D t.
\end{split}
\end{equation}
However, from the finiteness of the integrals in \eqref{eq:U^*} and \cite[Theorem 1 (1.15)]{DonMal04}, which triggers $\Ebb{T^n_{-x}}<\infty$, we arrive via H\"{o}lder's inequality at
\begin{equation}\label{eq:Ub1}
\begin{split}
     U^{*n}\lbrb{
     [-x,0)}&
     \leq \Ebb{T^{n}_{-x}}+ \sum_{j=0}^{n-1}\binom{n-1}{j}\Ebb{T^{n-1-j}_{-x}}\IntOI t^j\Pbb{\xi_t\geq 0}\D t\\
     &\leq C_n\max\lbcurlyrbcurly{1,\Ebb{T^n_{-x}}},
\end{split}
\end{equation}
where $C_n>0$ is some finite constant, so we see also that
$U^{*n}$ is $\sigma$-finite. We note that as $T_{-x}$ goes to infinity as $x$ increases, since by assumption $\limi{t}\xi_t=-\infty$, the last inequality implies that, for some $C' >1$, if $x\geq C'$, then  $U^{*n}\lbrb{
     [-x,0)} \leq C_n\Ebb{T^n_{-x}}. $
     Therefore, by \cite[(1.25)]{DonMal04},
for $x\geq C'$ and some $D_n>0$,
\begin{equation}
    \label{eq:U*bound}
 U^{*n}\lbrb{
     [-x,0)}
     \leq
     C_n\Ebb{T^n_{-x}} \leq D_n
\lbrb{\frac{x}{|A(x)|}}^n,
\end{equation}
where we recall that $A$ may have been suitably redefined to be positive for $x\geq 1$, see below \eqref{eq:A}.
Next, since $\xi$ drifts to $-\infty$, we can apply \cite[Lemma 2.2]{Yamamuro1998},
to get \[U([m + x, m+ x + 1))\leq c_1\Pbb{\rho_{m+x,m+x+1}<\infty}\] for some $c_1 > 0$ and,
for
$a<b$, $\rho_{a,b}:=\inf\curly{t\geq 0: \xi_t\in\lbbrbb{a,b}}$. 
Therefore, we immediately deduct that, for $m \geq 0$,
\begin{equation}\label{eq:Ubound}
  U([m + x, m + x + 1))\leq c_1\Pbb{\Hc_\tau
  \geq
  m+x}\leq
  c_1\Pbb{\Hc_\tau
  \geq 
  x},  
\end{equation}
where $\lbrb{\Hc_t}_{t\geq 0}$ is the conservative ascending ladder 
height process, which,  due to the fact that $\limi{t}\xi_t=-\infty$,
is killed at non-zero rate $\lambda_\kappa:=\kappa_+(0,0)>0$, that is, after independent of $\xi$ exponential time $\tau \sim Exp(\lambda_\kappa)$, which entails $\Hc_{\tau-}=\sup_{t\geq 0}\xi_t$. We note
that by quasi-left-continuity, $\Hc_{\tau-}
= \Hc_{\tau}$ a.s., 
see \cite[Proposition I.7]{Bertoin96},
and we will use the latter for smoother
presentation.
Hence, employing \eqref{eq:U*bound} and \eqref{eq:Ubound} in \eqref{eq:repW5}, we next get
\begin{equation}\label{eq:I_2+}
	\begin{split}
I^+_2&=\int_{(-\infty,0)}\!\IntOIc e^{-\epsilon y}U( \D y-x)U^{*n}(\D x)=\int_{(0,\infty)}\!\IntOIc e^{-\epsilon y}U( \D y+x)U^{*n}(-\D x)\\
&\leq \sum_{m\geq 0}e^{-\epsilon
		m}\IntOIo U([m+ x, m + x + 1)
  ) U^{*n}(-\D x)\\
		&\leq
  c_1\sum_{m\geq 0}e^{-\epsilon
		m}
		\IntOIo \P( \Hc_\tau 
  \geq
  x)U^{*n}(- \D x) \\
		&= \frac{c_1}{1-e^{-\epsilon}}
		\IntOIo
  \int_{[x, \infty)}\P( \Hc_\tau \in \D y)U^{*n}(-\D x)\\
  &= \frac{c_1}{1-e^{-\epsilon}}
		\IntOIo  U^{*n}\lbrb{[-
  x,0)} \P( \Hc_\tau \in \D x) \\
		&\leq
		\frac{c_1D_n}{1-e^{-\epsilon}}
  \lbrb{c_2 + 
		\int_{(1, \infty)} 
  \lbrb{\frac{x}{|A(x)|}}^n
  \P( \Hc_\tau \in \D x)}
		 \end{split}
\end{equation} 
for  $c_2:=U^{*n}\lbrb{[-
  C',0)}<\infty$ with $C'$ defined prior to \eqref{eq:U*bound}. 
If $\Ebb{\xi_1}$ is finite, then $\abs{A(\infty)}=-\Ebb{\xi_1}\in\lbrb{0,\infty}$, so it is enough
to prove that
\[
\IntOIo x^n \P( \Hc_\tau \in \D x) =
\Ebb{ H_\tau^n}=\Ebb{\lbrb{\sup_{t\geq 0}\xi_t}^n} < \infty,
\]
which is true by \cite[Theorem 5, (1.33)] {DonMal04} and assumption \eqref{eq:derW}. 
 Next, let us consider the
case where $\Ebb{\xi_1} = - \infty$,
so $|A(x)|$ is unbounded as well. Let us define 
$\Pi_\Hc$  to be
the \LL measure of $\Hc$ and
$A_\Hc(x) := \int_0^x \PiH(y)\D y$,
for which we
know from \cite[(2.15)]{DonMal04}
and \cite[p.34]{Kluppelberg-Kyprianou-Maller-2004} that
\begin{equation}
    \label{eq:equiv A AH}
|A(x)| \asymp \int_1^x \PiMinus(y)\D y
\asymp A_\Hc(x)
\end{equation}
with ``$\asymp$" denoting that the ratio of the two sides is bounded
away from zero and infinity.
Hence, our assumption
\eqref{eq:derW} entails that
\begin{equation}
    \label{eq:g_finite}
\int_{(1,\infty)} \lbrb{\frac{x}{\AH(x)
}}^{n+1} \Pi(\D x) =:
\int_{(1,\infty)}\lbrb{g(x)}^{n+1}\Pi(\D x)
< \infty,
\end{equation}
where we have used once again the
convention that if
$x/A_\Hc(x)$ is not
well-defined over
$(1,x_0]$ for some
$x_0 > 1$, we define
it there as some
constant.
Moreover, as $g(x) = x/\AH(x) = 1/\int_0^1\PiH(xy)\D y$, the function
$g$ is increasing, tends to
infinity at infinity, and has a finite limit at 0. Also, by
\eqref{eq:equiv A AH},
$\AH(\infty)$ is a positive constant
or infinity, so
\begin{equation}
    \label{eq:g = O(x^n)}
    g^n(x) = \bo{x^n},
\end{equation}
and, also from
\cite[(7.21)]{Kluppelberg-Kyprianou-Maller-2004}, we have the following
connection between $\PiH$ and $\Pi$:
for large $x$, there exist a constant
$c_3 > 0$ such that
\begin{equation} \label{eq:PiH inequality}
\PiH(x) \leq c_3 \int_{(x,\infty)}
g(y) \Pi(\D y).
\end{equation}
We are ready to prove that $I_2^+$ 
is finite in the considered case:
by \eqref{eq:I_2+}, \eqref{eq:equiv A AH}, and the last
discussion, it would be enough
to  prove that
$\int_{(1,\infty)} g^n(x) \P(\Hc_\tau
\in \D x)$
  is finite. To do this, recall that $\Hc_\tau \sim Exp(\lambda_\kappa)$, $g$ is increasing and note that because $g^n(x) \leq \int_{(1,x)}
\D \lbrb{g^n(y)} +g^n(2)$, we have to check that
\begin{equation}\label{eq:dgn}
	\begin{split}
		\int_{(1, \infty)}\lbrb{
 \int_{(1,x)}
  \D \lbrb{g^n(y)} + g^{n}(2) }
    &\P( \Hc_\tau \in \D x)
  \leq g^{n}(2) + \int_{(1,\infty)}
  \P(\Hc_\tau > y) \D \lbrb{g^n(y)}
  \\
  &\!\!\!\!= g^{n}(2)+ \int_{(1,\infty)} \IntOI \lambda_\kappa
  e^{-\lambda_\kappa t} \P(\Hc_t > y) 
  \D t\D \lbrb{g^n(y)}
		 \end{split}
\end{equation} 
is finite. To bound the integral, let us split the process into
two components, according to the size of the jumps:
$\Hc = \Hc^s+ \Hc^l$ for $\Hc^l$ the process of the jumps
larger than some $c>0$ and $\Hc^s$, the component with jumps
at most $c$, and then use the bound
$\P(\Hc_t > y) \leq \P(\Hc^s_t > y/2) + \P(\Hc^l_t > y/2)$ in \eqref{eq:dgn}.
For $\Hc^s$, observe that by Markov's inequality, for $\beta > 0$,
\[
\P(\Hc^s_t > y/2) \leq e^{-\beta y/2}\lbrb{\Ebb{e^{\beta\Hc^s_1}}}^t,
\]
where the expectation is finite, because the \LL
measure has bounded support, see
\cite[Theorem 25.3]{Sato1999}. Choosing $\beta=\beta(c,\lambda_\kappa)$ small enough, we
 have $\Ebb{e^{\beta\Hc^s_1}} \leq e^{\lambda_\kappa/2}$, so
by \eqref{eq:g = O(x^n)},
\begin{equation*}
    \begin{split}
      \IntOI \int_{(1,\infty)} \lambda_\kappa
  e^{-\lambda_\kappa t} &\P(\Hc^s_t > y/2) \D \lbrb{g^n(y)}
  \D t \\
  &\qquad\leq \IntOI \lambda_\kappa e^{-\lambda_\kappa t/2 } 
   \int_{(1,\infty)} e^{-\beta y/2} \D \lbrb{g^n(y)}
  \D t < \infty.  
    \end{split}
\end{equation*}
For the process of large jumps, first observe that
\begin{equation*}
	\begin{split}
\int_{(1,\infty)} \int_{ c_4 \ln y}^\infty \lambda_\kappa
  e^{-\lambda_\kappa t} \P(\Hc^l_t > y/2)
   \D t \D \lbrb{g^n(y)}
   \leq
   \int_{(1,\infty)} y^{-\lambda_\kappa c_4}\D \lbrb{g^n(y)},
		 \end{split}
\end{equation*} 
where we used $\P(\Hc^l_t > y/2) \leq 1$,
so again by \eqref{eq:g = O(x^n)}, choosing sufficiently large $c_4$,
the last integral is finite.  For the
remaining integral, with $N_t := 
\#$ jumps of $\Hc^l$ in $[0,t]
\sim Poi(t\PiH(c))$,
we can choose $c>0$ as large as we wish so that $e^{\PiH(c)}-1\leq \lambda_\kappa/2$. Fix $D>0$ such that for $x\geq1/(8D)>c$ relation \eqref{eq:PiH inequality} holds. Then, we have
\begin{equation}\label{eq:rep}
    \begin{split}
        &\int_{(1,\infty)} \int_{ 0 }^{c_4 \ln y }
 \lambda_\kappa e^{-\lambda_\kappa t} \P(\Hc^l_t > y/2)
   \D t \D \lbrb{g^n(y)}\\
   &\qquad= \int_0^\infty \lambda_\kappa e^{-\lambda_\kappa t}
  \int_{\lbrb{e^{t/c_4}, \infty}}
  \P\lbrb{\Hc^l_t > y/2}\D \lbrb{g^n(y)}
 \D t\\
 &\qquad=\int_0^\infty \lambda_\kappa e^{-\lambda_\kappa t}
  \int_{\lbrb{e^{t/c_4}, \infty}}
  \P\lbrb{\Hc^l_t > y/2;N_t\leq 2Dy}\D \lbrb{g^n(y)}
 \D t\\
 &\qquad\qquad
+\int_0^\infty \lambda_\kappa e^{-\lambda_\kappa t}
  \int_{\lbrb{e^{t/c_4}, \infty}}
  \P\lbrb{\Hc^l_t > y/2;N_t>2Dy}\D \lbrb{g^n(y)}
 \D t=:J_1+J_2.
    \end{split}
\end{equation}
First, we have by the classical Markov inequality and the choice of $c$ 
\begin{equation}\label{eq:rep1}
    \begin{split}
        J_2&\leq \int_0^\infty \lambda_\kappa e^{-\lambda_\kappa t}
  \int_{\lbrb{e^{t/c_4}, \infty}}
  \P\lbrb{N_t>2Dy}\D \lbrb{g^n(y)}\D t\\
  &\leq \int_0^\infty \lambda_\kappa e^{-\lambda_\kappa t}
  \int_{\lbrb{e^{t/c_4}, \infty}}
  e^{-2Dy}e^{t\lbrb{e^{\PiH(c)}-1}}\D \lbrb{g^n(y)}\D t\\
  &\leq \int_0^\infty \lambda_\kappa e^{-\lambda_\kappa t/2}\D t\int_{({1},\infty)} e^{-2Dy}\D \lbrb{g^n(y)}<\infty,
    \end{split}
\end{equation}
because $e^{-2Dy}\leq y^{-\lambda_\kappa c_4}$ for large $y$ and the finiteness discussed above. Also, with the quantity $k(y)=\max\curly{l\geq 1: l\leq 2Dy}$,
\begin{equation}\label{eq:up to c3ln}
	\begin{split}
 J_1&\leq 
 \int_0^\infty \lambda_\kappa e^{-\lambda_\kappa t}
  \int_{\lbrb{e^{t/c_4}, \infty}}
  \sum_{k = 1}^{k(y)}
  \P\lbrb{\Hc^l_t > y/2\middle|N_t = k}
  \P(N_t = k)\D \lbrb{g^n(y)}
 \D t \\
 &\leq
  \int_0^\infty \lambda_\kappa e^{-\lambda_\kappa t}
  \sum_{k = 1}^\infty
  \int_{\lbrb{e^{t/c_4}, \infty}} \ind{k\leq 2Dy}
  \P(N_t = k)
  k \frac{\PiH\lbrb{y/(2k)}}{\PiH\lbrb{c}}\D \lbrb{g^n(y)}
 \D t 
 \\
&\leq \frac{c_2}{\PiH\lbrb{c}}
  \int_0^\infty \lambda_\kappa e^{-\lambda_\kappa t}
  \sum_{k = 1}^\infty
   k\P(N_t = k)\\
      &\hspace{8em}\times
   \int_{\lbrb{e^{t/c_4}, \infty}}\ind{k< 4Dy}
   \int_{(y/(2k),\infty)}g(x) \Pi(\D x)
 \D \lbrb{g^n(y)}
 \D t\\
 &\leq \frac{c_2}{\PiH\lbrb{c}}
  \int_0^\infty \lambda_\kappa e^{-\lambda_\kappa t}
  \sum_{k = 1}^\infty
   k\P(N_t = k)\\
   &\hspace{8em}\times
   \int_{\left(\max\curly{e^{t/c_4}, k/(4D)}, \infty\right)}
   \int_{(y/(2k),\infty)}g(x) \Pi(\D x)
 \D \lbrb{g^n(y)}
 \D t,
 \end{split}
\end{equation} 
where, since $y/2k\geq 1/(4D) > 1/(8D)>c$, by the choice of $D$ we apply \eqref{eq:PiH inequality} for the third
inequality,
and for the second inequality we utilize  the fact that if $\Hc^l_t> y/2$  and
there are exactly $k$ jumps up to time $t$, then at least one of them has to be
larger than $y/(2k)$ with the probability of such jump given by $\PiH\lbrb{y/(2k)}/\PiH\lbrb{c}$.
For clarity, let us separate the following 
term from \eqref{eq:up to c3ln}:
\begin{equation*}
	\begin{split}
   &\int_{\left(\max\curly{e^{t/c_4}, k/(4D)}, \infty\right)}
   \int_{(y/(2k),\infty)}g(x) \Pi(\D x)
 \D \lbrb{g^n(y)}\\
 &\quad\quad=
\int_{\lbrb{\max\curly{e^{t/c_4}/(2k), 1/(8D)}, \infty}} 
 g(x)\int_{\lbrb{\max\curly{e^{t/c_4}, k/(4D)}, 2kx}}
  \D \lbrb{g^n(y)}\Pi(\D x)\\
  & \quad\quad\leq \int_{\lbrb{1/(8D), \infty}}
   g(x)\lbrb{\frac{2kx}{\AH(2kx)
}}^{n} \Pi(\D x)\\
&\qquad
  \leq (2k)^n \int_{(1/(8D), \infty)} g^{n+1}(x)
  \Pi(\D x) = \bo{k^n},
 \end{split}
\end{equation*} 
where we have used that $\AH$ is decreasing, the definition of $g$, and the finiteness in \eqref{eq:g_finite}. Substituting the last
in \eqref{eq:up to c3ln}, because, e.g.
by Hölder's inequality, $\Ebb{ N_t^{n+1}}
=\bo{t^{n+1}}$, we can conclude that
$J_1<\infty$. From this and \eqref{eq:rep1}, we can deduce that the integral in \eqref{eq:rep} is finite which yields that
$I_2^+$ is finite as well.

Thus, it remains
to demonstrate that $I^-_2<\infty$. Note that upon change
of variables $x\to -x$ and
$-y+x\to -z$ we get that
\begin{equation}\label{eq:I2-}
\begin{split}
    I^-_2&=\int_{(-\infty,
    0)}\IntOIc e^{-\epsilon y}U(- \D y-x)U^{*n}(\D x)
   \\
   &=
    \IntOIo e^{-\epsilon x}
    \int_{[-x,
    \infty)} e^{-\epsilon z}U(- \D z) U^{*n}(-\D x)\\
    &=
    \IntOIo e^{-\epsilon x}
    \IntOIc e^{-\epsilon z}U(- \D z)
    U^{*n}(-\D x)\\&\hspace{10em}+\IntOIo
    e^{-\epsilon x}\int_{[-x,0)} 
    e^{-\epsilon z}U(- \D z) U^{*n}(-\D x)=:J_1+J_2.
\end{split}
\end{equation}
Clearly, from \eqref{eq:U*bound},
\begin{equation}\label{eq:J1}
    \begin{split}
       &J_1= \IntOIc
       e^{-\epsilon z}U(- \D z)\IntOIo
       e^{-\epsilon x} U^{*n}
       (-\D x)\\
       &=
       \epsilon\IntOIc
       e^{-\epsilon z}U(- \D z)\IntOI  
       U^{*n}\lbrb{-y,0} e^{-\epsilon y}
       \D y\\
       &\leq \epsilon\IntOIo e^{-\epsilon z}U(- \D z)\IntOI  \lbrb{U^{*n}\lbrb{-C',0} +D_n
\lbrb{\frac{y}{|A(y)|}}^n \ind{y> 1}}e^{-\epsilon y}\D y<\infty,
    \end{split}
\end{equation}
since the integral with respect to $U(- \D z)$ is finite thanks to Theorem \ref{thm:Revuz}, and the other one, from \eqref{eq:equiv A AH}, $g(x)\asymp x/|A(x)|$ and the properties of $g$, see \eqref{eq:g = O(x^n)}. For $J_2$,  by splitting the inner integration at $x/2$, we get
\begin{equation*}
    \begin{split}
        &J_2=\IntOIo e^{-\epsilon x}\int_{(0,x]} e^{\epsilon z}U( \D z) U^{*n}(-\D x)\\
        &\,\,\leq \IntOIo e^{-\epsilon x/2} U\lbrb{\lbrbb{0,x/2}}U^{*n}(-\D x)
        +
        \IntOIo e^{-\epsilon x}\int_{(x/2,x)
        }e^{\epsilon z}U( \D z) U^{*n}(-\D x)\\
        &\,\,\leq
        c_5\IntOIo \max\curly{x,1}e^{-\epsilon x/2}U^{*n}(-\D x)
        +
        \IntOIo e^{-\epsilon x}\int_{(x/2,x)
        }e^{\epsilon z}U( \D z) U^{*n}(-\D x),
    \end{split}
\end{equation*}
where $U\lbrb{\lbrbb{0,x/2}}\leq c_5\max\curly{x,1}$ for $x>0$ and some $c_5>0$ follows easily from Theorem \ref{thm:Revuz}.
The first integral in the
last expression is finite since $xe^{-\epsilon x/2}$
decays faster to $0$ than $e^{-\epsilon x/3}$, and we can then proceed as in \eqref{eq:J1}. Next, we estimate the second integral, using 
$\lfloor\cdot\rfloor$ for the floor function, and employing 
\eqref{eq:Ubound} in the second inequality, as follows
\begingroup
\allowdisplaybreaks
\begin{align*}
     \IntOIo e^{-\epsilon x}\int_{(x/2,x)}
        &e^{\epsilon z}U( \D z) U^{*n}(-\D x) 
        \\&\leq 
     \IntOIo e^{-\epsilon x} \sum_{m = \lrfloor{x/2}}^{
\lrfloor{x}}e^{\epsilon (m+1)}U([m, m+1))U^{*n} (-\D x)\\
&\leq c_1e^{\epsilon}
     \IntOIo e^{-\epsilon x}\sum_{m = \lrfloor{x/2}}^{
\lrfloor{x}}e^{\epsilon m} \P\lbrb{\Hc_\tau \geq m}U^{*n} (-\D x)\\
&\leq c_1e^{\epsilon}
     \IntOIo e^{-\epsilon x}
     \sum_{m =0}^{
\lrfloor{x}}e^{\epsilon m}
\P\lbrb{\Hc_\tau \geq  \lrfloor{x/2}} U^{*n} (-\D x)\\
&\leq 
c_6
\IntOIo\P\lbrb{\Hc_\tau \geq  \lrfloor{x/2}} U^{*n} (-\D x)\\
&\leq 
c_6
\lbrb{\int_{(0,6]} U^{*n} (-\D x)
+
\IntOIo \P\lbrb{\Hc_\tau >  x/3}U^{*n} (-\D x)}
\\
&=c_6\lbrb{U^{*n}\lbrb{[-6,0)}+ \IntOIo U^{*n} (-3x,0) \P\lbrb{\Hc_\tau \in \D x}},
\end{align*}
\endgroup
where we have defined $c_6:=c_1e^{2\epsilon}/(e^\epsilon-1)$. The last is expression is finite because $U^{*n}\lbrb{[-6,0)}$ is finite by
\eqref{eq:Ub1}, and for the remaining, we can proceed exactly as in \eqref{eq:I_2+}: if $A(\infty) < \infty$, we have a direct
connection with $\Ebb{H_\tau^{n}}<\infty$, and in the other case, we can work with $A_\Hc$, which is
monotone and therefore, by \eqref{eq:U*bound}, for large $x$,
\[
U^{*n} (-3x,0) \leq D_n\lbrb{\frac{3x}{|\AH(3x)|}}^n \leq D_n3^n \lbrb{\frac{x}{|\AH(x)|}}^n,
\]
and we reduce to \eqref{eq:I_2+}.

Therefore, $J_2<\infty$,
and from \eqref{eq:J1} and \eqref{eq:I2-},
we get that $I^-_2<\infty$.
From the previous arguments, all $I^{\pm}_1,$ $I^{\pm}_{2}$
are finite, and from \eqref{eq:repW5}, we obtain that \eqref{eq:repW4} holds true.
Thus far we have proved that
the considered
derivatives of 
$W_{\kappa_\pm}$ are finite.
To conclude that this leads to the
same conclusion for
the derivatives of the same order of $W_{\phi_{q,\pm}}$,
 we recall \eqref{eq:Wk=Wp}:
\begin{equation*}
    W_{\phi_{q,+}}(z) = c_+^{1-z}\WkapP(q,z),
    \quad
    \text{and}
    \quad
    W_{\phi_{q,-}}(z)=c_-^{1-z} h^{z-1}(q)\WkapN(q,z),
\end{equation*}
with $h$ defined in \eqref{eq:h} as
\begin{equation*}
	\begin{split}
		&h(q) := \exp
  \lbrb{-\int_{0}^{\infty}
  \lbrb{\frac{e^{-t} -
  e^{-qt}}{t}}\Pbb{\xi_t=0}\D t}.
	\end{split}
\end{equation*}
If $\xi$ is not a CPP, $h \equiv 1$  and our claim is obvious. Let us assume now that
$\P(\xi_t = 0) \neq 0$ on some Lebesgue non-negligible set. Since $\xi$ is transient, the integral above is finite for every $q \geq 0$ by \cite[Theorem VI.12]{Bertoin96}, and therefore $h$ is strictly positive for each $q
\geq 0$. For higher derivatives of $h$, note that by the definition of $h$ and \eqref{eq:conv},
\begin{equation}
\label{eq:ln(h(q))}
\begin{split}
(\ln h(q))^{(k)}
&= (-1)^{k+1}\IntOIc
e^{-qt}t^{k-1}\P(\xi_t = 0)
\D t \\
&= U^{(k-1)}_q
\lbrb{\curly{0}}=
(-1)^{k-1}(k-1)!U_q^{*k}(\{0\}).
\end{split}
\end{equation}

The last is valid
for $1 \leq k \leq n+1$, as we have already proven \eqref{eq:repW4} which implies  for any $\epsilon>0$ that
\[U_q^{*k}(\{0\})\leq \IntOIc e^{-\epsilon y}U^{*k}(\pm \D y)<\infty,\] and therefore, for these $k$, $(\ln h(q))^{(k)}$ are
finite and
non-zero for each $q \geq 0$. Therefore $h^{(k)}(q)$ are finite as well, which 
concludes the proof of
the forward part
of the theorem.

For the backward part of the theorem, see \eqref{eq:derW}, by \eqref{eq:repW3}
and \eqref{eq:Wk=Wp},
it would suffice to show that if the integral in \eqref{eq:derW} is infinite,
then $U^{*(n+1)}$, and, respectively, $U^{(n)}$, do not define a 
$\sigma$-finite measure.
Under the assumption that 
\[\int_1^\infty \lbrb{\frac{x}{\abs{A(x)}}}^{n+1}\Pi(\D x)=\infty,\]
\cite[Theorem 4]{DonMal04} yields that, with $T_{a}:=\inf\curly{t\geq 0: \xi_t>a}$ for $a > 0$, it holds that
\begin{equation}\label{eq:rho}
    \Ebb{T^n_{a}\ind{T_{a}<\infty}}=\infty.
\end{equation}
Thus, picking $a<b$, with $\xi'_t=\xi_{t+T_{a}}-\xi_{T_{a}}$ on $T_{b}<\infty$,
\begin{equation}\label{eq:sigmaInf}
    \begin{split}
        U^{*(n+1)}\lbrb{a,b}&=\IntOI t^{n}\Pbb{\xi_t\in\lbrb{a,b}}\D t =\Ebb{\int_{T_{a}}^{\infty} t^n \ind{\xi_t\in\lbrb{a,b}}\D t}\\
        & = \Ebb{\ind{T_{a}<\infty}\IntOI \lbrb{t+T_{a}}^n \ind{\xi'_t\in\lbrb{a-\xi_{T_{a}},b-\xi_{T_{a}}}}\D t}\\
        &\geq \Ebb{T^n_{a}\ind{T_{a}<\infty}\IntOI  \ind{\xi'_t\in\lbrb{a-\xi_{T_{a}},b-\xi_{T_{a}}}}\D t}.
    \end{split}
\end{equation}
Assume, temporarily, that $\xi$ does not live on a lattice. Then clearly there exists $K>0$ such that $\Pbb{a \leq \xi_{T_{a}}\leq K+a}\geq 1/2$, and therefore from \eqref{eq:sigmaInf}, we get that
\begin{equation}\label{eq:sigmaInf1}
    \begin{split}
      U^{*(n+1)}\lbrb{a,b}&\geq \frac{1}{2}\inf\curly{a\leq y\leq K+a:U\lbrb{a-y,b-y}}\Ebb{T^n_{a}\ind{T_{a}<\infty}}.
    \end{split}
\end{equation}
Since $\inf\curly{a\leq y\leq K+a:U\lbrb{a-y,b-y}}>0$, we get from 
\eqref{eq:rho} that the quantity $U^{*(n+1)}\lbrb{a,b}$ is infinite. If $\xi$ lives on a
lattice, then the result is proved in the case where $\lbrb{a,b}$
contains a point of this lattice. This suffices. Therefore, our theorem is
proved.
\end{proof} 

\begin{proof}[Proof of Corollary \ref{cor:der_Mq}]
The following proof establishes bounds on the \(q\)-derivatives of Bernstein–Gamma functions in the contexts of interest. Although the underlying ideas are similar across cases, each situation requires specific arguments, which is why the full proof occupies several pages. The main approach is:
\begin{itemize}
    \item Express a logarithm $I(q,z)$ of a Bernstein-Gamma $\MBG{q}{z}$ function as a sum of integrals w.r.t. the $q$-potential measures $U_q$;
    \item Differentiate w.r.t. $q$ and
    use that successive derivatives of the $q$-measure are linked with its convolutions (Proposition \ref{prop:conv});
    \item Bound the obtained integrals;
    \item For the subtle cases
    \ref{it:cor_der_Mq_i'}-\ref{it:cor_der_Mq_iii'}, bound the
    convolutions of the potential measure
    via its link with passage times.  
\end{itemize}
Below, we also tried to outline key steps to help guide the reader through the argument. 

\textit{Proof of item 
(\ref{it:cor_der_Mq_i}):}\label{proof 5.10}
First, let us recall that \eqref{eq:M_I_psi} defines
\[
\MBG{q}{z} := 
\frac{\Gamma(z)}{W_{\phi_{q,+}}(z)}W_{\phi_{q,-}}(1-z).
\]
By Theorem \ref{thm:derW}, 
$\partialder{k}{q}{W_{\phi_{q,\pm}}(z)}$ are
right-continuous and finite for $q\geq 0$, $\Re(z)> 0$, 
and $0\leq k \leq n + 1$. Moreover, from 
\cite[Theorem 4.1 (1)]{PatieSavov2018},
$W_{\phi_{q,+}}(z)$ is non-zero for $q \geq 0$ and $\Re(z)> 0$, so differentiating the equality above, we obtain the desired result.

\textit{Proof of item (\ref{it:cor_der_Mq_ii}):}
The claim is trivial for 
$k = 0$. 

{\textbf{{Introduce $I(q,z)$}}: fix $k \in \{1, \dots, n+1\}$, and write, for 
$z \in \Cb_{(0,1)}$, $\MBG{q}{z} = e^{I(q, z)}$
with
\begin{equation}\label{eq:def I(q,z)}
\begin{split}
   I(q, z) &:= 
   \log_0(\Gamma(z)) +
   \log_0\lbrb{W_{\phi_{q,-}}(1-z) }
   -
   \log_0\lbrb{W_{\phi_{q,+}}(z) },  
\end{split}
\end{equation}
which is in fact equal to $\log_0( \MBG{q}{z})$ or
$\log_0( \MBG{q}{z}) + 2\pi i$, which gives the same
value after applying the exponential function. 

\textbf{Link the $q$-derivatives of $I$ and $U_q$}: substituting $\phi_{q,+}(1) = \kappa_+(q,1)/c_+$ and
\eqref{eq:logK} into \eqref{eq: deriv W_phi},
we obtain similar to \eqref{eq:deriv W_kappa end of proof} that
\[
  \frac{\partialder{}{q}{W_{\phi_{q,+}}(1+z)}}{W_{\phi_{q,+}}(1+z)}
        =
        z\frac{\partialder{}{q}{\kappa_+(q, 1)}}{\kappa_+(q, 1)}+\int_{[0, \infty)} u_z(y)\Wc'_{+}(q,\D y)=
        \IntOIc v_z(y) U_q(\D y)
\]
with $v_z$ is defined in
\eqref{eq:uz} and is bounded and continuous by Lemma \ref{lem:uz}. Similarly, because
$\phi_{q,-}(1) = h(q)\kappa_-(q,1)/c_-$, the relations \eqref{eq: deriv W_phi}, \eqref{eq:logK}, and \eqref{eq:ln(h(q))} gives us
\begin{equation*}
    \begin{split}
  \frac{\partialder{}{q}{W_{\phi_{q,-}}(1+z)}}{W_{\phi_{q,-}}(1+z)}
        &=
        z
        \lbrb{\frac{h'(q)}{h(q)}+\frac{\partialder{}{q}{\kappa_-(q, 1)}}{\kappa_-(q, 1)}}+\int_{[0, \infty)} u_z(y)\Wc'_{-}(q,\D y)\\
        &=
        zU_q\lbrb{\lbcurlyrbcurly{0}}+
        \IntOIc v_z(y) U_q(-\D y).     
    \end{split}
\end{equation*}
Using the last two equations in \eqref{eq:def I(q,z)}, we get that
\begin{equation}
\label{eq:der_I}
\partialder{k}{q}{I}
(q,z)
=\IntOIc v_{-z}(y) U^{(k-1)}_q
\lbrb{- \D y}
- \IntOIc v_{z-1}(y)
U^{(k-1)}_q\lbrb{\D y}
-z
U_q^{(k-1)}(\{0\}),
\end{equation}
where $v_z$ is defined in 
\eqref{eq:uz}.

\textbf{Estimate the integrals w.r.t. $U_q$}: we will
bound the quantities
in the last equality to obtain in \eqref{eq:bound_I}
the polynomial bound
\[\labsrabs{ 
  \partialder{k}{q}{I}
(q,z)} 
\leq 
 Q_{k, \Re(z)}\lbrb{\labsrabs{z}}.\]
 First note
that
from Proposition 
\ref{prop:conv},
\[\labsrabs{U^{(k-1)}_{q}(\D v)} \leq \labsrabs{U^{(k-1)}(\D v)},\]
and as the conditions of Theorem \ref{thm:derW}
are fulfilled by assumption, we know
from its proof, see \eqref{eq:repW4},
and
\eqref{eq:conv},
that, for each $\epsilon > 0$ and any $
1 \leq k \leq n+1$,
\begin{equation}
\label{eq:U_k
finite} 
\IntOIc
e^{-\epsilon y} \labsrabs{U^{(k-1)}(\pm \D y)}
=
(k-1)!
\IntOIc e^{-\epsilon y} \labsrabs{U^{*k}(\pm \D y)} < \infty.
\end{equation}
Therefore, 
by \eqref{eq:U_k
finite} and 
with some
$\epsilon > 0$,
we
have that
\begin{equation}\label{eq:hh}
 \labsrabs{
 zU^{(k-1)}\lbrb{\{0\}}
 }
 \leq|z|\IntOIc
e^{-\epsilon y} \labsrabs{U^{(k-1)}(\pm \D y)}
  \leq c_k |z|,
\end{equation}
with $c_k \geq0$. 

\textbf{Bound the integrated functions:}: by Lemma \ref{lem:uz}, there exist strictly positive constants $c_{1, \Re(z)},$ $c_{2, \Re(z)},$ $\epsilon_{1,
\Re(z)},$ and  $\epsilon_{2, \Re(z)}$ such that
for $z \in \Cb_{(0,1)}$,
\begin{equation}
    \label{eq: vz bounds in proof}
\labsrabs{v_{-z}(y)} \leq c_{1, \Re(z)} \labsrabs{z}e^{-\epsilon_{1,
\Re(z)} y}, 
\quad \text{and} \quad \labsrabs{v_{z-1}(y)} \leq c_{2, \Re(z)} \labsrabs{1-z}
e^{-\epsilon_{2, \Re(z)} y}.
\end{equation}

\textbf{Infer for $I$}: substituting $\eqref{eq:hh}$
and
$\eqref{eq: vz bounds in proof}$ in \eqref{eq:der_I},
we get that
\begin{equation}\label{eq:bound_I}
    \begin{split}
    &\labsrabs{ 
  \partialder{k}{q}{I}
(q,z)} 
\leq 
\IntOIc 
\labsrabs{ v_{-z}(y)}
\labsrabs{ U^{(k-1)}_q
\lbrb{- \D y}}
+ \IntOIc \!
\labsrabs{ v_{z-1}(y)}
\labsrabs{ 
U^{(k-1)}_q\lbrb{\D y}}
+ c_k |z|
\\
&\quad\leq
c_{1, \Re(z)} \labsrabs{z}
\IntOIc e^{-\epsilon_{1,
\Re(z)} y}\labsrabs{ U^{(k-1)}
\lbrb{- \D y}} \\
&\hspace{5em}
+
c_{2, \Re(z)} 
\lbrb{1 + \labsrabs{z}}
\IntOIc e^{-\epsilon_{2,
\Re(z)} y}\labsrabs{ U^{(k-1)}
\lbrb{\D y}} 
+ c_k |z|
=: Q_{k, \Re(z)}\lbrb{\labsrabs{z}},
    \end{split} 
\end{equation}
where $Q_{k, \Re(z)}$ is a polynomial of first degree, which is the desired bound for $I$.

\textbf{Transfer to $\MBG{q}{z}$}: to conclude, let us note that
 by Faà di Bruno's formula,
\begin{equation}
\label{eq:FdB exponent}
\lbrb{e^{ f}}^{(k)}=
e^{f}
\sum_{m_1+2m_2+\cdots+km_k=k} \frac{k!}{m_1!\,m_2!\,\,\cdots\,m_k!}
\prod_{j=1}^k\left(\frac{f^{(j)}}{j!}\right)^{m_j},
\end{equation}
and since we have $\MBG{q}{z} = e^{I(q, z)}$, by
using \eqref{eq:bound_I}, we obtain that
\[
\labsrabs{\partialder{k}{q}{\MBG{q}{z}}}
=\labsrabs{\partialder{k}{q}{e^{I(q, z)}}}
\leq \labsrabs{e^{I(q, z)}}
P_{ k, \Re(z)}
    \lbrb{\labsrabs{z}}
    =\labsrabs{ \MBG{q}{z}}P_{ k, \Re(z)}
    \lbrb{\labsrabs{z}}
\]
with $P_{ k, \Re(z)}$ a polynomial of degree $k$, which
is exactly the statement of the lemma.

\textit{Proof of item 
(\ref{it:cor_der_Mq_iii}):}
Under the assumption that there exists a bounded $u$ such that $U(\D x) = u(x)\D x$, we will
obtain a stronger bound in \eqref{eq:bound_I}, namely polynomial in polynomial in $\ln|z|$, and not $|z|$. 

Since the integral criterion from Theorem \ref{thm:convo}
holds, we know that the densities of $U^{*k}(\D x)$, that is $u^{*k}$, exist and are locally bounded for $1\leq k\leq n+1$. 
Next, by \eqref{eq:conv} of Proposition \ref{prop:conv},
we conclude that for $1\leq k\leq n+1$, $U^{(k-1)}(\D x)$ has  also a density $u_{k-1}(x)=k!(-1)^{k-1}u^{*k}(x)$,
which is therefore locally bounded.
Note that
$|U^{(k-1)}\lbrb{\curly{0}}| = \int_{0}^{\infty} t^{k-1}\Pbb{\xi_t=0}\D t = 0$.

Moreover, since $\Re(-z) = - \Re(z)\in(-1,0)$, by Lemma \ref{lem:uz}
there exist constants $
C_\Re(z) > 0, c_{3, \Re(z)},\epsilon_{3, \Re(z)}$
such that for $y > C_{\Re(z)}$, $|v_{-z}(y)| \leq c_{3, \Re(z)} e^{-y\epsilon_{3, \Re(z)}}$, so if $1/|z| < C_{\Re(z)}$

\textbf{Estimate the first integral  in \eqref{eq:der_I}:} therefore,
\begin{equation}\label{eq:v_z_y_ln}
    \begin{split}
\IntOI 
&\labsrabs{ v_{-z}(y)}
\labsrabs{ U^{(k-1)}_q
\lbrb{- \D y}}\!=\!
\int_{[0, 1/|z|)\cup [1/|z|, C_{\Re(z)}) \cup
[C_{\Re(z)}, \infty)}\!
\labsrabs{ v_{-z}(y)}
\labsrabs{ U^{(k-1)}_q
\lbrb{- \D y}}\\
&\leq
\frac{1}{|z|}c_{1, \Re(z)} \labsrabs{z}
 \sup_{x
\in (-C_{\Rez},0)}\abs{ u_{k-1}(x)}
+
\int_{1/|z|}^{C_{\Re(z)}} \labsrabs{\frac{1- e^{zy}}{e^y - 1}
u_{k-1}(-y)\D y }\\
&\hphantom{{}=
\leq
\frac{1}{|z|}c_{1, \Re(z)} \labsrabs{z}
\IntOI e^{-\epsilon_{1,
\Re(z)} y}\qquad\quad
}
+c_{3,\Re(z)}\IntOI e^{-\epsilon_{3,
\Re(z)} y}\labsrabs{ U^{(k-1)}
\lbrb{-\D y}}
\\
&\leq c_{4, \Re(z)} + 
\int_{1/|z|}^{C_{\Re(z)}} \labsrabs{\frac{1- e^{zy}}{e^y - 1}
u_{k-1}(-y)}\D y ,
    \end{split} 
\end{equation}
for some $c_{4, \Re(z)}>0$ by \eqref{eq:U_k
finite}, and in the first bound on the region $0\leq y< 1/|z|$, we have used the first bound of \eqref{eq: vz bounds in proof}.
We note that if $1/|z| \geq C_{\Re(z)}$, we can bound the quantity of interest as above, but without having the integral 
over $[1/|z|,C_{\Re(z)})$, resulting in a constant bound by $c_{4, \Re(z)}$. 

For the last integral in \eqref{eq:v_z_y_ln},
because $|e^{zy} - 1| \leq e^{y\Re(z)} + 1 \leq 2e^{y\Re(z)}$ for the considered $z\in\Cb_{(0,1)}$, and $u_{k-1}$ is locally bounded, we get that
\begin{equation}\label{eq:bound integral u}
    \begin{split}
\int_{1/|z|}^{C_\Re(z)} \labsrabs{\frac{1- e^{zy}}{e^y - 1}
u_{k-1}(-y)}\D y 
&\leq \sup_{x\in\lbbrbb{-C_{\Re(z)},0} } \lbcurlyrbcurly{|u_{k-1}(x)|}\int_{1/|z|}^{C_{\Re(z)}}\frac{2e^{y\Re(z)}}{y}\D y
\\
&\leq c_1 \int_{1/|z|}^{C_{\Re(z)}}\frac{1}{y}\D y= c_1\lbrb{\ln\lbrb{C_{\Re(z)}} + \ln|z|},
    \end{split} 
\end{equation}
for $c_1:=c_1(\Re(z))  >  0$, so \eqref{eq:v_z_y_ln} gives a linear bound in $\ln|z|$ for the first summand on the first line of \eqref{eq:bound_I},
with coefficients that depend only on $\Re(z)$.

\textbf{Estimate the second integral in \eqref{eq:der_I}}: for the second one, as $\Re(z-1) \in (-1,0)$, with the same approach we can obtain a linear bound in $\ln(|1-z|)$ and because
$|1-z| \leq |z| + 1 \leq |z|(1 + 1/\Re(z))$, this gives a linear bound in $\ln|z|$, with coefficients depending only on $\Re(z)$,
and not $|z|$.

\textbf{Transfer from $I(q,z)$ to $\MBG{q}{z}$}: repeating the arguments which prove $(\ref{it:cor_der_Mq_ii})$ from \eqref{eq:bound_I}, we get the
desired polynomial bound in $\ln|z|$.

We proceed to the proofs of the last four items, 
\textit{\ref{it:cor_der_Mq_i'}}-\textit{\ref{it:cor_der_Mq_iv'}}, under the
assumptions that $\Ebb{\xi_1}$ is finite and,
for any $\beta \in (0,n)$,
\[   R(\beta):= \int_{(1,\infty)} x^{\beta+1}\Pi(\D x)<\infty=\int_{(1,\infty)} x^{n+1}\Pi(\D x) = R(n).\]
From \eqref{eq:PandInt2}, the last is equivalent to
\begin{equation}
    \label{eq: t^beta fnite}
\int_1^{\infty}t^{\beta-1}\Pbb{\xi_t\geq 0} \D t<\infty=\int_1^{\infty}t^{n-1}\Pbb{\xi_t\geq 0} \D t.
\end{equation}

\textit{Proof of item
\ref{it:cor_der_Mq_i'}:} 
note that if $\Ebb{\xi_1}$ is finite, then $A(\infty) = \Ebb{\xi_1}$ is finite, as discussed
above $\eqref{eq:A}$.
Therefore, because by assumption $R(n-1)$
is finite, we can apply \textit{(\ref{it:cor_der_Mq_i})},
\textit{(\ref{it:cor_der_Mq_ii})},
and \textit{(\ref{it:cor_der_Mq_iii})}
with $n-1$ instead of $n$,
which gives the desired result.

We continue further. Note that, from \eqref{eq:PandInt1}, 
in the considered case
$$\limo{q}U_q^{*n}(\R_+)=\limo{q}\IntOI t^{n-1}e^{-qt}\Pbb{\xi_t\geq 0} \D t=\infty,$$
which is item
\textit{\ref{it:cor_der_Mq_iv'}}. This is in contrast to \eqref{eq:U^*} in the
previous case.

\textit{Proof of item \ref{it:cor_der_Mq_ii'}:} We try be economical due to the similarities with the proofs of
(\ref{it:cor_der_Mq_ii}) and (\ref{it:cor_der_Mq_iii}}). 

\textbf{Linking passage times and $U_q^{*n}$:} using \cite[(1.25)]{DonMal04}, applied to $-\xi$, and because $|A(\infty)|\in\lbrb{0,\infty}$ since $\Ebb{\xi_1}\in\lbrb{-\infty,0}$, 
for each $\beta \in(0,n)$
\begin{equation}\label{eq:Tbound}
    \Ebb{T^{\beta}_{-x}}\leq C_{\beta, 1}\max\curly{ x^{\beta}, 1}
\end{equation}
for some finite constant $C_{\beta,1}\geq 0$ and with $T_{-x}$ the first passage time below $-x$. Then, a slight modification of \eqref{eq:Ub} yields,  for $x>0$,
\begin{equation}\label{eq:bound Uq n+1}
\begin{split}
     U_q^{*n}&\lbrb{[-x,0)} 
     = 
     \IntOI t^{n-1}e^{-qt}\Pbb{\xi_t\in\lbbrb{-x,0}}
     \D t\\
     &\leq \Ebb{\int_{0}^{T_{-x}}t^{n-1}e^{-qt}\D t}+\Ebb{e^{-qT_{-x}}\IntOI 
     e^{-qt}(t+T_{-x})^{n-1}\ind{\tilde{\xi}_t\geq 0}\D t},
\end{split}
\end{equation}
where $T_{-x}<\infty$ a.s., and by the strong Markov property $\tilde{\xi}:=\lbrb{\xi_{t+T_{-x}}-\xi_{T_{-x}}}_{t\geq 0}$ is independent of $T_{-x}$ and has the law of $\xi$.
Next, using the obvious inequality that for any $\delta>0$ small enough. there is $D_{\delta,1}<\infty$ such that for any $x>0$,
\[\int_{0}^x t^{n-1}e^{-t}\D t\leq  D_{\delta,1} x^{n-\delta},\]
we conclude that
\[\Ebb{\int_{0}^{T_{-x}}t^{n-1}e^{-qt}\D t}=q^{-n}\Ebb{\int_{0}^{qT_{-x}}t^{n-1}e^{-t}\D t}\leq D_{\delta,1} q^{-\delta}\Ebb{T^{n-\delta}_{-x}}.\]
Applying the latter and \eqref{eq:Tbound} in
\eqref{eq:bound Uq n+1}, we arrive with some $D_{\delta,2}<\infty$ at
\[U_q^{*n}\lbrb{[-x,0)}\leq D_{\delta,2} q^{-\delta} \max\curly{x^{n-\delta},1}+\Ebb{\IntOI 
     e^{-qt}(t+T_{-x})^{n-1}\ind{\tilde{\xi}_t\geq 0}\D t}.\]
    Expanding the last brackets, and treating terms like in \eqref{eq:Ub1}, we deduct 
    \begin{equation*}
\begin{split}
     U_q^{*n}\lbrb{
     [-x,0)}&
     \leq   D_{\delta,2} q^{-\delta}\max\curly{x^{n-\delta}, 1}\\
     &\qquad\qquad+ \sum_{j=0}^{n-1}\binom{n-1}{j}\Ebb{T^{n-1-j}_{-x}}\IntOI e^{-qt}t^j\Pbb{\xi_t\geq 0}\D t.
\end{split}
\end{equation*}
From \eqref{def:pot} and \eqref{eq:PandInt1}, and like in \eqref{eq:Ub1}, applying H\"{o}lder's inequality, we get 
\begin{equation}\label{eq:Ub2}
    \begin{split}
     &U_q^{*n}\lbrb{
     [-x,0)}\\
    & \quad
     \leq   D_{\delta,2} q^{-\delta}\max\curly{x^{n-\delta},1}+\IntOI e^{-qt}t^{n-1} \Pbb{\xi_t\geq 0}
     \D t+C_n
     \max\lbcurlyrbcurly{1,\Ebb{T^{n-\delta}_{-x}}}\\
     & \quad\leq C_{\delta, n}q^{-\delta}\max\curly{x^{n-\delta}, 1}+U^{*n}_q(\Oiclosed),
\end{split}
\end{equation}
where in the first inequality we have taken out from the sum the summand with index $n-1$, we have estimated the rest with $q=0$, and we have applied that by \eqref{eq:PandInt1}, for $j<n$, $U^{*j}(\Oiclosed)<\infty$ since $R(j)<\infty$.
For the second inequality, we have employed once again \eqref{eq:Tbound} for the expectation term,
and from \eqref{eq:conv} we have related the free integral to $U_q^{*n}(\Oiclosed)$.

\textbf{Estimating the derivatives of $I$}: since $R(k)<\infty$ for $k<n,$  all estimates for $\partialder{k}{q}{I}
(q,z)$ for $k\leq n$ can be reused from the previous part.
Therefore, we need to estimate only
the $(n+1)$-st derivative, which from
\eqref{eq:der_I} is equal to
\[\partialder{n+1}{q}{I}
(q,z)
=\IntOIc v_{-z}(y) U^{(n)}_q
\lbrb{- \D y}
- \IntOIc v_{z-1}(y)
U^{(n)}_q\lbrb{\D y}
-
z
U_q^{(n)}(\{0\}).\]
  Substituting \eqref{eq:conv} to relate $U^{(n)}_q$ to $U^{*(n+1)}_q$, we get
  \begin{equation*}
      \begin{split}
          \abs{\partialder{n+1}{q}{I}
(q,z)}\leq n!&\Bigg(\IntOIc \abs{v_{-z}(y)} U^{*(n+1)}_q
\lbrb{- \D y}\\
&\qquad\qquad+ \IntOIc \abs{v_{z-1}(y)}
U^{*(n+1)}_q\lbrb{\D y}+ |z|U^{*(n+1)}_q(\{0\})\Bigg).
      \end{split}
  \end{equation*}
  From Lemma \ref{lem:uz}, there exist constants $c_{1, \Re(z)}, c_{2, \Re(z)}, \epsilon_{1,
\Re(z)}, \epsilon_{2, \Re(z)} > 0$ such that,
for $z \in \Cb_{(0,1)}$,
\[
\labsrabs{v_{-z}(y)} \leq c_{1, \Re(z)} \labsrabs{z}e^{-\epsilon_{1,
\Re(z)} y},
\quad \text{and} \quad \labsrabs{v_{z-1}(y)} \leq c_{2, \Re(z)} \labsrabs{1-z}
e^{-\epsilon_{2, \Re(z)} y}.
\]
Therefore, for some constants $c_1 := c_1(\Re(z))>0$ and  $\epsilon:=\min\curly{\epsilon_{1, \Re(z)},\epsilon_{2, \Re(z)}}>0$, we arrive at
\begin{equation}\label{eq:In}
      \begin{split}
          &\abs{\partialder{n+1}{q}{I}
(q,z)}\\&\qquad\leq c_1\lbrb{2|z|\IntOIc e^{-\epsilon y} U^{*(n+1)}_q
\lbrb{- \D y}
+ |1-z|\IntOIc  e^{-\epsilon y}
U^{*(n+1)}_q\lbrb{\D y}
|z| }\\
&\qquad\leq 
c_1\lbrb{2|z|\lbrb{I^-_1(q) +I_2^-(q)}+|1-z|\lbrb{I^+_1(q) + I^+_2(q)}},
      \end{split}
  \end{equation}
  where we have borrowed the notation from \eqref{eq:repW5}, and the multiplier term $2$ comes from bounding the atom at zero by an integral
  over $[0,\infty)$. 
  
  \textbf{Bounding $I_1^\pm$ using the renewal theorem}: $I^{\pm}_1(q)$ can be estimated precisely as  below \eqref{eq:repW5} to get
  \begin{equation}
      \label{eq: I1_pm}\begin{split}
I^{\pm}_{1}(q)&\leq \sup_{v\in\Rb}U_q\lbrb{[v,v+1)}U_q^{*n}\lbrb{\Oiclosed}\frac{1}{1-e^{-\epsilon}}\\&\leq \sup_{v\in\Rb}U\lbrb{[v,v+1)}U_q^{*n}\lbrb{\Oiclosed}\frac{1}{1-e^{-\epsilon}}\leq
  c_2 U_q^{*n}\lbrb{\Oiclosed},
        \end{split}
    \end{equation}
  for $c_2:=\sup_{v\in\Rb}U\lbrb{[v,v+1)}/(1-e^{-\epsilon})<\infty$ since the supremum is finite
  by the renewal
Th  eorem \ref{thm:Revuz}.
 
  \textbf{Bounding $I_2^\pm$ using the specific for the case bounds}: next, we get as in \eqref{eq:I_2+}
  \begin{equation*}
	\begin{split}
I^+_2(q)&=\int_{(-\infty,0)}\IntOIc e^{-\epsilon y}U_q( \D y-x)U_q^{*n}(\D x)\\
&=\int_{(0,\infty)}\IntOIc e^{-\epsilon y}U_q( \D y+x)U_q^{*n}(-\D x)\\
&\leq \sum_{m\geq 0}e^{-\epsilon
		m}\IntOIo U([m+ x, m + x + 1)
  ) U^{*n}_q(-\D x)\\
  &\leq 
		 \frac{c_3}{1-e^{-\epsilon}}
		\IntOIo  U^{*n}\lbrb{[-
  x,0)} \P( \Hc_\tau \in \D x)
		 \end{split}
\end{equation*} 
for some $c_3>0$, where in the first inequality we have used that in terms of measures $U_q\leq U$. Substituting the bound from \eqref{eq:Ub2}, we find that
\begin{equation}\label{eq:I2_1}
	\begin{split}
I^+_2(q)&\leq 
		 \frac{c_3}{1-e^{-\epsilon}}
		\IntOIo\lbrb{ C_{\delta, n}q^{-\delta}\max\curly{x^{n-\delta}, 1}+U^{*n}_q(\Oiclosed) }\P( \Hc_\tau \in \D x)\\
  &\leq c_4\lbrb{q^{-\delta}+U^{*n}_q(\Oiclosed)},
		 \end{split}
\end{equation} 
for
$c_4:=c_4(\Re(z), n, 
\delta)>0$, since $R(n-\delta)<\infty$ implies via \cite[1.33]{DonMal04} applied to $-\xi$ that $\Ebb{\lbrb{\sup_{t\geq 0}\xi_t}^{n-\delta}}<\infty$, and, as already used, $\Hc_\tau=\sup_{t\geq 0}\xi_t$ a.s. Finally, as in \eqref{eq:I2-},
\begin{equation}\label{eq:I2-q}
\begin{split}
    I^-_2(q)&=
    \IntOIo e^{-\epsilon x}
    \IntOIc e^{-\epsilon z}U_q(- \D z)
    U_q^{*n}(-\D x)\\
    &\hspace{4em}+\IntOIo
    e^{-\epsilon x}\int_{[-x,0)} 
    e^{-\epsilon z}U_q(- \D z) U_q^{*n}(-\D x):=J_1(q)+J_2(q).
\end{split}
\end{equation}
For $J_1$, following \eqref{eq:J1}, we get that
\begin{equation}\label{eq:J1q}
    \begin{split}
       J_1(q)&= \IntOIc
       e^{-\epsilon z}U_q(- \D z)\IntOIo
       e^{-\epsilon x} U_q^{*n}
       (-\D x)\\
       &=
       \epsilon\IntOIc
       e^{-\epsilon z}U_q(- \D z)\IntOI  
       U_q^{*n}\lbrb{-y,0} e^{-\epsilon y}
       \D y\\
       &\leq \epsilon\IntOIo e^{-\epsilon z}U(- \D z)\IntOI   \lbrb{C_{\delta,n}q^{-\delta}\max\curly{y^{n-\delta}, 1}+U^{*n}_q(\Oiclosed) }e^{-\epsilon y}\D y\\
       &\leq c_5\lbrb{q^{-\delta}+U^{*n}_q(\Oiclosed)},
    \end{split}
\end{equation}
where in the first inequality we have used the bound in \eqref{eq:Ub2} and $U_q\leq U$. The finite constant $c_5:=c_5(\Re(z), n ,\delta)$ incorporates 
the finiteness of
the first integral
by the renewal Theorem \ref{thm:Revuz} as
well as the others involved explicit constants. Next,
for $J_2$, as below \eqref{eq:J1} we have 
\begin{equation*}
    \begin{split}
        &J_2(q)=\IntOIo e^{-\epsilon x}\int_{(0,x]} e^{\epsilon z}U_q( \D z) U_q^{*n}(-\D x)\\
        &\quad\leq
       \IntOIo e^{-\epsilon x/2}U((0,x/2])U_q^{*n}(-\D x)
        +
        \IntOIo e^{-\epsilon x}\int_{(x/2,x)
        }e^{\epsilon z}U_q( \D z) U_q^{*n}(-\D x)\\
                &\quad\leq c_6
       \IntOIo e^{-\epsilon x/3}U_q^{*n}(-\D x)
        +
        \IntOIo e^{-\epsilon x}\int_{(x/2,x)
        }e^{\epsilon z}U_q( \D z) U_q^{*n}(-\D x)\\
          &\quad\leq c_7\lbrb{q^{-\delta}+U^{*n}_q(\Oiclosed)}+
        \IntOIo e^{-\epsilon x}\int_{(x/2,x)
        }e^{\epsilon z}U_q( \D z) U_q^{*n}(-\D x),
    \end{split}
\end{equation*}
for $c_6>0, c_7 := c_7(\Re(z), n ,\d) < \infty$, 
where using the arguments below \eqref{eq:J1}, we have bounded $U_q\lbrb{\lbrbb{0,x/2}}
\leq U((0,x/2]) \leq c_8 \max\{x,1\} \leq c_9e^{\epsilon x/6}$ for some finite $c_8, c_9$ and  we have treated $\IntOIo e^{-\epsilon x/3}U_q^{*n}(-\D x)$ precisely as in \eqref{eq:J1q}, possibly at the cost of a different constant.
Third, again precisely as for $J_2$ below \eqref{eq:J1},  we can obtain 
\begin{equation*}
    \begin{split}
       \IntOIo e^{-\epsilon x}\int_{(x/2,x)
        }&e^{\epsilon z}U_q( \D z) U_q^{*n}(-\D x)\\
        &\leq c_{10}\lbrb{U_q^{*n}\lbrb{[-6,0)}+ \IntOIo U_q^{*n} (-3x,0) \P\lbrb{\Hc_\tau \in \D x}}
    \end{split}
\end{equation*}
   for $c_{10}:=c_{10}(\Re(z))\geq 0$, so using the bound \eqref{eq:Ub2}, and repeating the estimate as in \eqref{eq:I2_1}, we get by substitution above  in $J_2(q)$ that
    \[
    J_2(q) \leq c_{11}\lbrb{q^{-\delta}+U^{*n}_q(\Oiclosed)}, 
    \]for
    $c_{11}:=c_{11}(\Re(z),n ,\d)\geq 0$. Feeding the latter together with 
    the bound for $J_1(q)$, see \eqref{eq:J1q}, into 
    $I_2^-(q)$,
 \eqref{eq:I2-q}, we get with some $c_{12} := c_{12}(\Re(z), n , \delta) <\infty$,
\[I^-_2(q)\leq c_{12}\lbrb{q^{-\delta}+U^{*n}_q(\Oiclosed)}.\]

\textbf{Combining estimates for $I^\pm$:}
using the latter with \eqref{eq:I2_1}
and \eqref{eq: I1_pm}, 
and adding them in \eqref{eq:In}, we obtain
\begin{equation*}
      \begin{split}
          \abs{\partialder{n+1}{q}{I}
(q,z)}&\leq 
c_{13}\lbrb{q^{-\delta}+U^{*n}_q(\Oiclosed)}\lbrb{|1-z|+|z|}
      \end{split}
  \end{equation*}
  for $c_{13} := c_{13}(\Re(z), n, \delta) \geq 0$.
  Combining the latter with the bounds for the lower
  order derivatives of $I(q,z)$, \eqref{eq:bound_I},
  and Faa di Bruno's formula, 
  \eqref{eq:FdB exponent}, we obtain
\ref{it:cor_der_Mq_ii'}. 

\textit{Proof of item \ref{it:cor_der_Mq_iii'}:} again, we only need to estimate the $(n+1)$-st 
$q$-derivative of $I$, as for lower orders we can use the estimates
  from the proof of (\ref{it:cor_der_Mq_iii}), see
  \eqref{eq:v_z_y_ln} and below.   Adapting suitably  \eqref{eq:v_z_y_ln}
  and
  \eqref{eq:bound integral u}, we get that,
  for some
  finite
  $\epsilon, c_{14}, c_{15}>0$,
  which
  depend on $\Re(z)$,
\begin{equation}
\label{eq: der iii'}
    \begin{split}
\IntOI 
\labsrabs{ v_{-z}(y)}
\labsrabs{ U^{(n)}_q
\lbrb{- \D y}}
\leq
c_{14} 
&\IntOI e^{-\epsilon y}\labsrabs{ U_q^{*(n+1)}
\lbrb{-\D y}} \\
&\quad+
\sup_{x\in\lbbrbb{-c_{15},0} } \lbcurlyrbcurly{|u_q^{*(n+1)}(x)|}\lbrb{\ln\lbrb{c_{15}} + \ln|z|}.
    \end{split} 
\end{equation}
We have already seen in the proof of \ref{it:cor_der_Mq_ii'},
see \eqref{eq:In} and below, that
\begin{equation}
    \label{eq:der iii' part 1}
\IntOI e^{-\epsilon y}  U^{*(n+1)}_q
\lbrb{- \D y}\leq
c_{16}\lbrb{q^{-\delta}+U^{*n}_q(\Oiclosed)}
\end{equation}
with $c_{16}:=c_{16}(\Re(z), n , \delta)$,
and for the supremum,
by \eqref{eq:convo1}
from Theorem \ref{thm:convo}, we have that, for
$C:=\sup_{x \in \R} |u(x)| < \infty$,
\begin{equation*}
    \begin{split}
\sup_{x\in\lbrb{[-c_{15},0)} } \lbcurlyrbcurly{|u_q^{*(n+1)}(x)|}
&\leq (n+1)C U_q^{*n} \lbrb{-c_{15}, \infty}.
    \end{split}
\end{equation*}
Substituting
\eqref{eq:Ub2} in the former, we get that
\begin{equation*}
    \begin{split}
\sup_{x\in\lbrb{[-c_{15},0)} } \lbcurlyrbcurly{|u_q^{*(n+1)}(x)|}\leq 
(n+1)C\lbrb{C_{\delta, n}q^{-\delta} 
\max\lbcurlyrbcurly{c_{15}^\delta, 1} + 2U_q^{*n}\lbrb{\R_+}},
    \end{split}
\end{equation*}
so, for some finite
$c_{17} :=c_{17}\lbrb{\Re(z) ,n, \delta}$,
\[
\sup_{x\in\lbrb{[-c_{15},0)} } \lbcurlyrbcurly{|u_q^{*(n+1)}(x)|}\leq
c_{17}\lbrb{q^{-\delta} + U_q(\R_+)}.
\]
Substituting the last inequality and
\eqref{eq:der iii' part 1} into \eqref{eq: der iii'},
 the desired result,
\ref{it:cor_der_Mq_iii'},
follows as in the proof of
(\ref{it:cor_der_Mq_iii}), see below \eqref{eq:bound integral u}.
\end{proof}
\begin{proof}[Proof of Theorem \ref{thm:convo}]
\label{proof 4.6}
We prove by induction that for $k = 1, 2, \dots, n+1$,
\begin{equation}
    \label{eq:induc11}
u_q^{*k}(x) \leq kCU^{*(k-1)}_q(\min\{0, x\}, \infty).
\end{equation}
The result in the base case $k=1$ holds because
$U_q(\D y) \leq U(\D y)$ implies that
$\sup_{x\in\Rb}u_q(x)\leq \sup_{x\in\Rb}u(x)=:C$ a.e. which leads to inequality everywhere when working with the canonical $u_q$, see \cite[Proposition I.11.(ii)]{Bertoin96}.
Let \eqref{eq:induc11} hold for some $l\geq 1$. Then, 
\begin{equation*}
    \begin{split}
        u^{*(l+1)}_q(x)&=\int_{-\infty}^\infty u_q(y)u^{*l}_q(x-y) \D y\\
        &\leq C\int_{-\infty}^x u_q^{*l}(x-y)\D y
        +
        lC\int_{x}^\infty u_q(y)U^{*(l-1)}_q\lbrb{\min\lbcurlyrbcurly{0, x-y},\infty}\D y\\
&=CU_q^{*l}\lbrb{0,\infty}+lC\int_{x}^\infty U_q^{*(l-1)}\lbrb{x-y,\infty}U_q\lbrb{\D y}\\
&\leq CU_q^{*l}\lbrb{0,\infty}+lC\int_{-\infty}^\infty U_q^{*(l-1)}\lbrb{x-y,\infty}U_q\lbrb{\D y}\\
        &= CU_q^{*l}\lbrb{0,\infty}+lCU_q^{*l}\lbrb{x,\infty}
        \leq(l+1)CU^{*l}_q\lbrb{\min\lbcurlyrbcurly{0,x},\infty},
    \end{split}
\end{equation*}
as needed,
which settles \eqref{eq:convo1}. Next, under the integral condition in the theorem's statement, by taking a limit $q\to 0$ in \eqref{eq:induc11} and using the monotone convergence theorem, we conclude the proof for the case $q=0$ since from \eqref{eq:U^*} and \eqref{eq:U*bound} we have that $U^{*k}(x,\infty)<\infty$ for any $x\in\Rb$ and any $1\leq k\leq n$. 

\end{proof}

\section{Proofs on exponential functionals} \label{sec: proof thm main}
We state three auxiliary results before we commence
the proof of Theorem \ref{thm:main}. Their proofs can be found in Section \ref{sec: auxiliary}.
\begin{proposition} 
\label{prop:LaplaceT}
Let $f:[0, \infty) \to [0, \infty)$, and define $f_0 := f$. For $x\geq 0$  and  $n \geq 1$, define
\begin{equation*}
    f_{n}(x) :=
  \int_x^\infty \int_{s_1}^\infty \dots \int_{s_{n-1}}^ \infty 
  f(s_n) \D s_n \dots \D s_1.
\end{equation*}
Then, for $q \geq 0$,
\begin{equation} \label{eq:LaplaceT}
    \widehat{f_n}(q) = \frac{(-1)^n}{(n-1)!}\int_0^1 \widehat{f}^{(n)}(qv)(1-v)^{n-1} \D v
\end{equation}
in the extended sense, meaning that if one side is infinite, the other is infinite as well.
\end{proposition}
\begin{proposition}\label{prop:RV}
    Let $f$ be a regularly varying function at zero with index $\beta\in\lbrbb{-1,0}$ and $n\geq 1$ fixed. Then
    \begin{equation}\label{eq:RV}
    \int_{0}^1 f(qv)(1-v)^{n-1} \D v
\simo {f(q)}\int_{0}^1 v^{\beta}(1-v)^{n-1} \D v=f(q)\frac{\Gamma(n)\Gamma(\beta+1)}{\Gamma(n+\beta+1)}.
    \end{equation}
\end{proposition}
\begin{lemma} \label{lemma:bound_E}
    Let $\xi$ be a \LL process with a finite negative mean and $\P(\xi_1 > t) \simi \ell(t)/t^\alpha$ for 
    some $\alpha > 1$.
    Then
    \begin{equation} \label{eq: lemma upper bound}
        \Ebb{I_\xi^{-a}(t)} = \bo{\P(\xi_1 > t)} =  \bo{ \frac{\ell(t)}{t^\alpha}}.
    \end{equation}
\end{lemma}
\fakesubsection{Proof of Theorem \ref{thm:main}}\begin{proof}[Proof of Theorem \ref{thm:main}]\label{proof thm 3.1}
We start with the remark that the equivalence in 
\eqref{eq:reg_var_asymp} is due to
\cite[Theorem 1.1]{Watanabe2008}.
In particular,
\begin{equation}
    \label{eq: Pibar}
\PiPlus(t) := \Pi(t, \infty)\simi \frac{\ell(t)}{t^{\alpha}},
\end{equation}
and by \cite[Theorem 4]{Doney-Jones-2012}, 
applied to the  zero mean process $\xi_t + ct$ with $c=-\Ebb{ \xi_1}>0$ therein, this assumption also
leads to the second asymptotic relation in
\begin{equation}
\label{eq:asymp_xi_t_0}
     \P(\xi_t \geq 0) \simi \P(\xi_t > 0)=\P(\xi_t +ct>ct)\simi t\PiPlus(ct)
    \simi
    c_1\frac{\ell(t)}{t^{\alpha-1}},
\end{equation}
with $c_1:= \lbrb{-\Ebb{\xi_1}}^{-\alpha}$. The first is obvious when the \LLP has infinite activity as then $\Pbb{\xi_t=0}=0$. Otherwise, we can apply \cite[Theorem 4]{Doney-Jones-2012} with $\widetilde{c}_\delta:=-\Ebb{ \xi_1}-\delta$, $-\Ebb{ \xi_1}> \delta >0,$ in which case we arrive at
\[\Pbb{\xi_t>0}\leq \Pbb{\xi_t\geq0}\leq \Pbb{\xi_t>-\delta t}\simi (\widetilde{c}_\delta)^{-\alpha}\frac{\ell(t)}{t^{\alpha-1}}\simi(\widetilde{c}_\delta)^{-\alpha}c^{-1}_{1}\Pbb{\xi_t>0}.\]
Setting $\delta\to 0$, we verify the first asymptotic relation in \eqref{eq:asymp_xi_t_0}.

Let us define for $x \geq 0$ the function
$\{x\} := 
\sup\lbcurlyrbcurly{ y \in (0,1]: x - y \in \Zb}$.
This is a slight
variation of the usual fractional part
function with the difference that
it is 1 in the integers and not 0.

Next, let
$\phiplus{q} := \phi_+(q, 0)$ for $q 
\geq 0$. By choosing $z = 0$ in
\eqref{eq:phi_+} and differentiating $n := \alpha - \{\alpha\} \geq 1$ times
w.r.t. $q$,
we get that, for $q>0$,
\begin{equation} \label{eq:ln_kappa}
    \lbrb{\ln \phi_+}^{(n)}(q)
    =(-1)^{n+1}\IntOI e^{-qt}t^{n-1} 
    \P(\xi_t \geq 0) \D t.
\end{equation}
We aim to obtain the asymptotic
behaviour of the
integral on the right 
as $q \to 0$ in order to
extract information about $\phi_+$ via
a Tauberian theorem. For this purpose,
let us observe that 
by 
\eqref{eq:asymp_xi_t_0}
and the choice of $n$, we get that,
for $x \to \infty$,
\begin{equation}
    \label{eq:Karamata_cases}
    \int_0^x t^{n-1} \P(\xi_t \geq 0) \D t
    \simi c_1\int_1^x 
    t^{n-1} t^{-\alpha + 1} \ell(t) \D t 
    =
    c_1\int_1^x t^{-\{\alpha\}}
\ell(t) \D t.
\end{equation}
Now we can differentiate naturally two
cases: if $\alpha$ is non-integer, we
can apply Karamata's theorem, Theorem \ref{thm:Karamata_1.5.8.}, 
in
\eqref{eq:Karamata_cases}, and continue with a standard Tauberian theorem, Theorem \ref{thm:Tauberian_1.7.1}. In
the integer case, we will use a more subtle
method, based on the fact that
\[
x \mapsto \int_1^x t^{-1}\ell(t) \D t
\]
is in the de Haan class $\Pi_\ell$, see the
\hyperref[appn]{Appendix}
and \cite[Chapter 3]{BinGolTeu1989} for
more information on this class.

\textbf{Case 1: $\alpha$ is non-integer, i.e. $\alpha \in (1,\infty) \setminus \N$.}
Take $\alpha\in\lbrb{n,n+1}$ for some $n\geq 1$.
Let us now apply the
aforementioned Karamata's theorem, \ref{thm:Karamata_1.5.8.},
in \eqref{eq:Karamata_cases},
which implies that
\[
\int_0^x t^{n-1} \P(\xi_t \geq 0) \D t
\simi
c_1\frac{x^{1- \{\alpha\}}}{
1- \{\alpha\}} \ell(x),
\]
and so, by a standard Tauberian theorem \ref{thm:Tauberian_1.7.1},
we get from \eqref{eq:ln_kappa} that
\begin{equation} \label{eq:ln_kappa_case1}
\lbrb{\ln \phi_+}^{(n)}(q) \simo
(-1)^{n+1} c_1 \Gamma\lbrb{1 - \{\alpha\}}
q^{\{\alpha\}-1}
\ell(1/q).
\end{equation}
The latter is infinite at $0$ and we can also see
that $\lbrb{\ln \phi_+}^{(k)}(0+)$ are finite for $k<n$ which follows from the integral representation \eqref{eq:ln_kappa} and the finiteness of the non-vanishing integral in \eqref{eq:PandInt1}. 
Thus, as by Faà di Bruno's formula, with non-negative integers $m_i$, with the sum over $n$-tuples
with $m_1+2m_2+\dots+nm_n=n$,
\begin{equation}
    \label{eq:Faa_di_Bruno_ln}
    \begin{split}
&\lbrb{\ln \phi_+}^{(n)}=\\
&\qquad\sum \frac{n!}{m_1!\,m_2!\,\,\cdots\,m_n!}
\frac{ (-1)^{m_1+\cdots+m_n-1} (m_1+\cdots+m_n-1)! }{\phi_+^{m_1+\cdots+m_n}} \prod_{j=1}^n\left(\frac{\phi_+^{(j)}}{j!}\right)^{m_j},
\end{split}
\end{equation}
and the fact that, since the process drifts to
$-\infty$,
$\phiplus{0}$ is finite, we can see
that 
\begin{equation}
    \label{eq:deriv_phi_inf}
\labsrabs{\phiplusderiv{0+}{k}} <
\infty \text{ for }
0\leq k \leq n-1, \quad\text{ and } \quad
\labsrabs{\phiplusderiv{0+}{n}} = \infty.
\end{equation}
From the last three results, as the only term of $\lbrb{\ln \phi_+}^{(n)}$ containing
$\phi_+^{(n)}$ is $\phi_+^{(n)}/\phi_+$, i.e. when $m_n=1$ and $m_j=0$ for $j\leq n-1$ in the sum,
we get that
\begin{equation}\label{eq:asymp_kappa_n}
\phiplusderiv{q}{n}
\simo   (-1)^{n+1}c_1 
\Gamma\lbrb{1 - \{\alpha\}}
\phiplus{0}
q^{\{\alpha\}-1}
\ell\lbrb{1/q}.
\end{equation}
Define for brevity $g_a(t) := \Ebb{I^{-a}_\xi(t)}$, which we note is 
monotone decreasing,
and we recall that from
\eqref{eq:LT}, for $a \in (0,1)$,
\begin{equation}
    \label{eq:der_Mq/k}
    \begin{split}
          \widehat{g_a}^{(n)}(q)& =\partialder{n}{q}\IntOI e^{-qt}g_a(t) \D t=\minusone^{n}\IntOI t^ne^{-qt}g_a(t) \D t\\
&=\lbrb{
\frac{\MBG{q}{1-a}}{\phiplus{q}}}^{(n)}
= \sum_{k = 0}^n
\binom{n}{k}
\lbrb{\frac{1}{\phi_+}}^{(k)}(q)
\partialder{n-k}{q}
{\MBG{q}{1-a}}.
    \end{split} 
\end{equation}
Since $\abs{\Ebb{\xi_1}}<\infty$ and $\alpha\in\lbrb{n,n+1}$, we know from \eqref{eq:PandInt2} that $\int_{(1, \infty)} x^{n-1+1}\Pi(\D x)$ is finite. Hence, by (\ref{it:cor_der_Mq_i}) of Corollary \ref{cor:der_Mq},
all derivatives of
$\MBG{q}{1-a}$ in \eqref{eq:der_Mq/k}
are finite even for $q=0$. Furthermore,
from \cite[(7.64)]{PatieSavov2018}, we know that for
$z \in \C_{(0,1)}$,
$\lim_{q\to 0}\MBG{q}{ z} = \MBG{0}{z}$, so
combining this with
\eqref{eq:deriv_phi_inf},
we obtain that
\begin{equation}\label{eq:gan_interm}
    \widehat{g_a}^{(n)}(q) \simo \lbrb{\frac1{\phi_+}}^{(n)}(q)
    \MBG{0}{1-a},
\end{equation}
and using again Faà di Bruno's formula,
so the sum being over all $n$-tuples of non-negative integers $(m_1, \dots, m_n)$ such that $m_1 + 2m_2 + 
\dots + nm_n = n$, and for $x>0$, 
\begin{equation}
\begin{split}
\label{eq:FdB_1/f} \lbrb{\frac{1}{\phi_+}}^{\!\!(n)}\!\!\!(q) \!=\!\sum \!\frac{n!}{m_1!\,m_2!\,\,\cdots\,m_n!}
\frac{ (-1)^{m_1+\cdots+m_n} (m_1+\cdots+m_n)! }{(\phi_+(q))^{m_1+\cdots+m_n+1}} \!\prod_{j=1}^n\!\left(\frac{\phi_+^{(j)}(q)}{j!}\right)^{m_j}
\end{split}
\end{equation}
we can see that if we expand the derivative in \eqref{eq:gan_interm},
there is only one infinite at zero term, which is exactly the one that contains $\phiplusderiv{q}{n}$.
It is obtained for $m_n=1$ in \eqref{eq:FdB_1/f}, so
substituting \eqref{eq:asymp_kappa_n}, we get that
\begin{equation}\label{eq:g_a}
	\begin{split}
\widehat{g_a}^{(n)}(q) \simo 
-\frac{\phiplusderiv{q}{n}}{\phi_+^2(q)}
    \MBG{0}{1-a}
\simo
 q^{\{\alpha\}-1} \ell_1(q)
	\end{split} 
\end{equation} 
with \begin{equation}
    \label{eq:def l_1}
\ell_1(q) :=
   (-1)^{n}c_1 
\Gamma\lbrb{1 - \{\alpha\}}
\phi^{-1}_+(0)
\MBG{0}{1-a}\ell(1/q)
  \in SV_0,
  \end{equation}
  where $SV_0$ denotes the class of slowly varying at zero functions, see the \hyperref[appn]{Appendix}.
  Now, let us define $g_{a, n}$ in the spirit of Proposition \ref{prop:LaplaceT}:
  \begin{equation}
      \label{eq:def_g_an}
      g_{a,n}(x) :=
  \int_x^\infty \int_{s_1}^\infty \dots \int_{s_{n-1}}^ \infty 
  g_a(s_n)
\D s_{n}   \dots \D s_2 \D s_1,
  \end{equation}
which is well-defined since by Lemma
\ref{lemma:bound_E} and Potter's bounds, see Proposition \ref{prop:Potter's bounds}, we have $g_a(t) = \bo{\ell(t)/t^{\alpha}} = \bo{
1/t^{n + \curly{\alpha}/2}}$. Therefore, from Proposition \ref{prop:LaplaceT} itself,
\[
\widehat{g_{a,n}}(q) = \frac{(-1)^{n}}{(n-1)!}\int_0^1 \widehat{g_a}^{(n)}(qv)(1-v)^{n-1} \D v.
\]
Next from \eqref{eq:g_a} we can invoke Proposition \ref{prop:RV} with $\beta=\{\alpha\}-1$ to obtain the asymptotic behaviour of $\widehat{g_{a,n}}$:
\begin{align*}
\widehat{g_{a,n}}(q) &\simo 
\f{(-1)^{n}}{(n-1)!}
\lbrb{\int_0^1 v^{\{\alpha\}-1}(1-v)^{n-1} \D v}
q^{\{\alpha\}-1} \ell_1(q) \\
&\simo 
\f{(-1)^{n}\Gamma(\{\alpha\})}{\Gamma(\alpha)}
q^{\{\alpha\}-1}
\ell_1(q).
\end{align*}
Therefore, from the standard Tauberian theorem \ref{thm:Tauberian_1.7.1},
\[
\int_0^t g_{a,n}(y) \D y \simi
\f{(-1)^{n}\Gamma(\{\alpha\})}{\Gamma(\alpha) \Gamma(2 - \{\alpha\})}
t^{1- \{\alpha\}}
\ell_1\lbrb{\f{1}{t}}.
\]
Next, we can apply the monotone density theorem \ref{thm:Monotone density, 1.7.2} to get
\[
g_{a,n}(t) \simi \f{(-1)^{n}\Gamma(\{\alpha\})}{\Gamma(\alpha)\Gamma(1-\{\alpha\})} t^{-\{\alpha\}} \ell_1\lbrb{\f{1}{t}} = \frac{c_1 \MBG{0}{1-a}\Gamma(\{\alpha\})}{\phiplus{0}\Gamma(\alpha)}t^{-\{\alpha\}} \ell(t), 
\]
where we have substituted the definition of $\ell_1$, see \eqref{eq:def l_1}. Finally, taking into account the monotonicity of the inner integrals defining $g_{a,n}$ in \eqref{eq:def_g_an}, after $n$ consecutive applications
of the second part of the monotone density theorem, \ref{thm:Monotone density, 1.7.2}, in the last asymptotic relation, we determine, recalling  $\alpha=n+\{\alpha\}$, that
\begin{equation} \label{eq:asymp_g}
g_a(t) = \Ebb{I^{-a}_\xi(t)} 
\simi 
\frac{c_1 \MBG{0}{1-a}}{\phiplus{0}}t^{-\alpha} \ell(t).
\end{equation}
To prove the weak convergence of the normalised quantities $\Ebb{I^{-a}_\xi(t)}$ as $t \to \infty$, let us now consider the distribution functions, for $x > 0$,
\begin{equation}\label{def:gx}
     g_a^x(t):=\Ebb{I_\xi^{-a}(t)\ind{I_\xi(t) \leq x}}.
\end{equation}
 Define the Laplace transform of this function via
$f^q_a(x) := \widehat{g_a^x}(q)$.
From \cite[(7.46)]{PatieSavov2018}, it can be seen that
the Mellin transform of $f_a^q$, for $z \in \Cb_{(a-1, 0)}$, is
\begin{equation}\label{eq:Mellin f}
    \begin{split}
        \mathcal{M}_{\widehat{f_a^q}}(z) &= \IntOI x^{z-1}
\IntOI e^{-qt} \Ebb{I_\xi^{-a}(t)\ind{I_\xi(t) \leq x}}\D t \D x\\
&=
\frac{1}{q} \IntOI x^{z-1} \Ebb{I_{\Psi_q}^{-a}\ind{I_{\Psi_q} \leq x}}
\D x
= -\frac{1}{\phiplus{q}z}\MBG{q}{ z+1-a},
    \end{split}
\end{equation}
where we recall that $I_{\Psi_q}:=I_{\xi}(\mathbf{e}_q)=\int_0^{e_q} e^{-\xi_s}\D s$.
We can try to get a similar result to \eqref{eq:asymp_g} by Mellin inversion and Tauberian theorems. We first argue that the
Mellin transform is well-defined: from \cite[Theorem 2.3]{PatieSavov2018}, the
decay of $M_\Psi$ is at least polynomial
along complex lines in $\C_{(0,1)}$. To be more
precise, writing $z = b + iy$ with
$b \in (a-1,0)$, the aforementioned theorem
implies that if $\xi$ is not a CPP
with a positive drift, for any
$\beta \geq 0$ and $q \geq 0$,
\begin{equation}\label{eq:decay_M}
\lim_{|y| \to \infty} |y|^{\beta} |\MBG{q}{b+1-a+iy}| = 0,
\end{equation}
and otherwise, there exists $\epsilon > 0$
such that
\begin{equation}
\label{eq:decay_M_CPP}
\lim_{|y| \to \infty} |y|^{\epsilon} |\MBG{q}{b+1-a+iy}| = 0.
\end{equation}
Therefore, we can invert the Mellin
transform from \eqref{eq:Mellin f},
for $b \in (a-1, 0)$,
\begin{equation}
\begin{split} \label{eq:g_ax_Laplace}
f_a^q(x)  =
\widehat{g_a^x}(q)&=
-\frac{1}{2\pi i}\int_{\Re(z) = b}\frac{x^{-z}}{\phiplus{q}z}\MBG{q}{ z+1-a} \D z\\
&=-\frac{x^{-b}}{2\pi \phiplus{q}}\int_{-\infty}^\infty \frac{x^{-iy}}{b+iy}\MBG{q}{b+1-a+iy} \D y \\
&=: -\frac{x^{-b}}{2\pi \phiplus{q}}\int_{-\infty}^\infty h(q, y) \D y,
\end{split}
\end{equation}
as $\eqref{eq:decay_M}$ and $\eqref{eq:decay_M_CPP}$ ensure
that for
some $\epsilon>0$, $
\lim_{|y| \to \infty} |y|^{1+\epsilon} |h(q, y)| = 0
$, and thus the last integral is finite. 
Next, 
if we wish to apply again Proposition \ref{prop:LaplaceT} for the $n$th derivative of $\widehat{g_a^x}(q)$, it would be enough to check whether the integral of the
respective
derivatives (right-derivatives for $q=0)$ w.r.t. $q$ of  $h(q, y)$ is absolutely integrable. 

\textbf{Case 1.1: $\xi$ is not a CPP with positive drift.}
Let us
write for simplicity 
$z = b + 1 - a + iy$. Note
that because $\Re(z) \in (0, 1)$ and the considerations below \eqref{eq:der_Mq/k} are valid, then by (\ref{it:cor_der_Mq_ii}) of Corollary \ref{cor:der_Mq},
for each $0 \leq k \leq n+1$,
    there exist polynomials $P_{\Re(z), k}$ of degree $k$ such that
\begin{equation}
\label{eq: finite integral}
\labsrabs{\partialder{k}{q}
{\MBG{q}{z}}} \leq P_{ \Re(z), k}(|z|)| \MBG{q}{ z})|,
\end{equation}
and as we noted earlier, see \eqref{eq:decay_M}, for $q\geq 0$, the decay of $M_\Psi$ is super-polynomial 
along complex lines, so the last 
result proves that $\int_{-\infty}^\infty
|\partialder{k}{q} h(q,y)| \D y$ is finite.

\textbf{Case 1.2: $\xi$ is a CPP with positive drift.}
The difference in this case is that we have only a polynomial decrease of $M_\Psi$, 
see \eqref{eq:decay_M_CPP}. However, now
every point is possible for $\xi$, so \cite[Proposition II.11, Theorem II.16]{Bertoin96}
entails that the potential of the process has a bounded density w.r.t. the Lebesgue measure, and
therefore, by item (\ref{it:cor_der_Mq_iii}) of Corollary \ref{cor:der_Mq},
for each $0 \leq k \leq n+1 $,
there exist polynomials $P_{\Re(z), k}$ of degree $k$ such that
\[\labsrabs{\partialder{k}{q}
{\MBG{q}{z}}} \leq P_{ \Re(z), k}(\ln|z|)| \MBG{q}{ z})|, 
\]
and similarly to the first case, we  obtain that
 $\int_{-\infty}^\infty
|\partialder{k}{q} h(q,y)| \D y$ is finite for $q\geq 0$.

Summing up, in both subcases, 1.1 and 1.2, we can differentiate $n$ times
the equation \eqref{eq:g_ax_Laplace}, to obtain, for $q>0$,
\begin{equation}
    \label{eq:gaxn}
\widehat{g_a^x} ^{(n)}(q)
  = 
  \lbrb{\frac{h_a^x(q)}{\phiplus{q}}}^{(n)}
   \text{\quad with \quad} h_a^x(q) := -\frac{1}{2\pi i}\int_{\Re(z) = b}\frac{x^{-z}}z \MBG{q}{ z+1-a} \D z,
\end{equation}
which is in the form of \eqref{eq:der_Mq/k},
but with $h_a^x(q)$ in  place of $M_\Psi(q,1-a)$.
Next, using the logic below \cite[(7.48)]{PatieSavov2018}, we get from \cite[(7.66)]{PatieSavov2018} that, for any $\epsilon>0$ in subcase 1.1. and for some $\epsilon>0$ in subcase 1.2,
\begin{equation}
    \label{eq:M uniform convergence}
\limsup_{|b| \to \infty} |b|^\epsilon
\sup_{0\leq q \leq 1} |\MBG{q}{ z+1-a}-\MBG{0}{z+1-a}| = 0,
\end{equation}
and from the other claims of \cite[Lemma 7.3]{PatieSavov2018}
that 
\begin{equation*}
\lim_{q \to 0+}h_a^x(q) = h_a^x(0).
\end{equation*}Moreover, using again \cite[(7.66)]{PatieSavov2018}, for any $\epsilon>0$ in subcase 1.1. and for some $\epsilon>0$ for subcase 1.2, we additionally deduce from the bounds in items (\ref{it:cor_der_Mq_ii})-(\ref{it:cor_der_Mq_iii}) of Corollary \ref{cor:der_Mq} that 
\begin{equation}
    \label{eq:h_a^x limsup}
\limsup_{q \to 0+}\abs{\lbrb{h^{x}_a(q)}^{(k)}}<\infty
\end{equation}
for any $k\leq n+1$.
Therefore, as the only infinite term at zero in
\eqref{eq:gaxn} is $\phiplusderiv{0}{n}$, we can repeat the
arguments from \eqref{eq:der_Mq/k} to \eqref{eq:g_a} to get 
\[
\widehat{g_a^x} ^{(n)}(q)
  \simo -\frac{\phiplusderiv{q}{n}}
  {\phi^2_+(q)} h_a^x(0)
  \simo
   \frac{h_a^x(0)}{\MBG{0}{1-a}}q^{\{\alpha\}-1} \ell_1(q).
\]
Using the same chain of arguments as between \eqref{eq:g_a} and \eqref{eq:asymp_g}, we
obtain also an analogue of \eqref{eq:asymp_g} for $g_{a}^x$ simply by replacing $M_\Psi$
with $h_a^x$:
\begin{equation}
    \label{eq:asymp_gax}
g_a^x(t) = \Ebb{I_\xi^{-a}(t) \ind{I_\xi(t) \leq x}}
\simi 
\frac{c_1 h_a^x(0)}{\phiplus{0}}t^{-\alpha} \ell(t).
\end{equation}
Therefore, combining with the similar
\eqref{eq:asymp_g}, we get that, for any $x>0$,
\begin{equation}
    \label{eq:distribution functions}
\limi{t}\frac{\Ebb{I_\xi^{-a}(t) \ind{I_\xi(t) \leq x}}}{t^{\alpha} \ell(t)}=\frac{c_1 h_a^x(0)}{\phiplus{0}},
\quad
\text{and}
\quad
\limi{t}\frac{\Ebb{I_\xi^{-a}(t)}}{t^{\alpha} \ell(t)}=\frac{c_1\MBG{0}{1-a}}{\phiplus{0}}.
\end{equation}
Let us define on
$(0,\infty)$ the measures
$\nu_{a,t}(\D y):= y^{-a}\P(I_\xi^{-a}(t)\in \D y)t^{\alpha}/\ell(t)$. As we see from
\eqref{eq:distribution functions}, these measures
have finite masses, so as we seek to
obtain weak convergence, we will normalise
them and use the classic Portmanteau
theorem, see \cite[Corollary 8.2.10]
{Bogachev2007}. Let
$\widetilde{\nu}_{a,t}(\D y):=\nu_{a,t}(\D y)/\nu_{a,t}(0,\infty)$. From 
\eqref{eq:distribution functions},
\[
\lim_{t \to \infty}
\widetilde{\nu}_{a,t}(\D y)\lbrb{(0,x]}
= \frac{h_a^x(0)}{\MBG{0}{1-a}},
\]
hence we only need to verify that the last
defines a distribution function of a 
probability measure, for which it
would be enough to check, substituting the definition of
$h_a^x$, \eqref{eq:gaxn}, that
\begin{equation}
    \label{eq: lim h_a^x}
\lim_{x \to \infty}h_a^x(0) = 
\lim_{x \to \infty} \frac{-1}{2\pi i}\int_{\Re(z) = b}\frac{x^{-z}}z \MBG{0}{z+1-a} \D z=
\MBG{0}{1-a}.
\end{equation}
We will show that the last holds by using the residue theorem: by \cite[(2.10)]{PatieSavov2018},
the function $z \mapsto x^{-z}\MBG{0}{z+1-a}/z$ is meromorphic on $\Cb_{(a-1,a)}$, with a single simple pole at
0. Its residue there is exactly $\MBG{0}{1-a}$, so integrating by the
negatively oriented contour of the rectangle $\gamma_{R}:=\{z\in \C|-b_1 \leq \Re(z) \leq b_1,-R \leq \Im z \leq R\}$
with the choice $-b = b_1 \in (0, \min\{a,1-a\})$ and $R >0$,
gives us
\begin{equation}\label{eq:contour_int}
-2 \pi i \MBG{0}{1-a} = \lim_{R \to \infty} \oint_{\gamma_{R}}\frac{x^{-z}}z \MBG{0}{z+1-a}\D z.
\end{equation}
Along the horizontal
strips of the contour,
i.e. $\Im z = \pm R$, we use the bound 
\[
|\MBG{q}{ z}| \leq \MBG{q}{\Re(z)},
\]
which holds when the quantities
are well-defined, and is based on the results 
\cite[Theorem 6.1 (3), Proposition 6.8]
{PatieSavov2021},
which state,
respectively,
that
$W_{\phi_{q,-}}(z)$
and
$\Gamma(z)/W_{\phi_{q,+}}(z)$
are Mellin transforms of
 positive random variables.
Indeed, if we call these
random variables
$V_{\phi_{q,-}}$ and
$I_{\phi_{q,+}}$, we can 
see that
\begin{equation}
    	\begin{split}
\labsrabs{\MBG{q}{ z}} &=\labsrabs{\frac{\Gamma(z)}{W_{\phi_{q,+}}(z)}W_{\phi_{q,-}}(1-z)
}
=
\labsrabs{\Ebb{I_{\phi_{q,+}}^{z-1}}\Ebb
{V_{\phi_{q,-}}^{-z}}}
\\&\leq
\Ebb{\labsrabs{I_{\phi_{q,+}}^{z-1}}}\Ebb
{\labsrabs{V_{\phi_{q,-}}
^{-z}}}= 
\Ebb{I_{\phi_{q,+}}^{\Re(z)-1}}\Ebb
{V_{\phi_{q,-}}^{-\Re(z)}}
= \MBG{q}{\Re(z)}.
\end{split}
\end{equation}
This gives us 
\begin{equation*}
    \begin{split}
 &\labsrabs{\int_{-b_1}^{b_1}
\frac{x^{-y - iR}}{y + iR} 
\MBG{0}{y +1 -a  + iR}
-
\frac{x^{-y + iR}}{y - iR} 
\MBG{0}{y +1 -a  - iR} \D y}
\\
&\hspace{10em}
\leq
\frac{2}{R}\int_{-b_1}^{b_1}
x^{-y} \MBG{0}{y +1 -a  }\D y
=
\bo{\frac{1}{R}},      
    \end{split}
\end{equation*}
because
on $[-b_1,b_1]$,
$y \mapsto x^{-y}$ and $y \mapsto$ $\MBG{0}{y+1-a}$ are continuous which is clear for the first and by \cite[(2.10)]{PatieSavov2018}  for the second.
Therefore, by dominated convergence along the horizontal strips of $\gamma_R$,
by \eqref{eq:contour_int},
\begin{equation*}
    \begin{split}
       -2 \pi i &\MBG{0}{1-a} = 
       \lim_{R \to \infty} \oint_{\gamma_{R}}\frac{x^{-z}}z \MBG{0}{z+ 1-a}\D z \\
       &= \int_{\Re(z) = -b_1}
       \frac{x^{-z}}z \MBG{0}{
       z+ 1-a} \D z 
       - 
       \int_{\Re(z) = b_1}\frac{x^{-z}}z \MBG{0}{z+1-a} \D z. 
    \end{split}
\end{equation*}
It remains to prove
that the integral along $\Re(z) = b_1$ goes to zero as $x\to \infty$. 
Using once again the polynomial decay of $\MBG{0}{\cdot}$ along complex lines, see \eqref{eq:decay_M} and \eqref{eq:decay_M_CPP},
and the fact that $b_1>0$, we have that
\begin{equation*}
    	\begin{split}
&\labsrabs {\int_{\Re(z) = b_1}\frac{x^{-z}}z \MBG{0}{z+1-a} \D z}
= \labsrabs {\int_{-\infty}^\infty \frac{x^{-b_1 -iy}}{-b_1 + iy} \MBG{0}{-b_1 +1-a +  iy} \D y}\\
&\hspace{8.7em}\leq 
x^{-b_1}
{\int_{-\infty}^\infty 
\frac{1}{\max\lbcurlyrbcurly{b_1, |y|}}\labsrabs{\MBG{0}{-b_1 +1-a +  iy} }\D y} < \infty,
\end{split}
\end{equation*}
so by dominated convergence, the last tends to $0$ 
when $x \to \infty$ as desired, i.e.
\eqref{eq: lim h_a^x} holds. Therefore, by the Portmanteau
theorem, we have the weak convergence 
\[
\widetilde{\nu}_{a,t} (\D y)
= \frac{\nu_{a,t}(\D y)}{\nu_{a,t}\lbrb{(0, \infty]}}
\xrightarrow[t \to \infty]{w}
\widetilde{\nu}_a(\D y)
\]
with $\widetilde{\nu}_a(\D y)$ a
probability measure with distribution
function $h_a^x(0)/\MBG{0}{1-a}$. Therefore,
as,  by \eqref{eq:distribution functions}, $\limi{t}\nu_{a,t}\lbrb{(0, \infty]}
=c_1\MBG{0}{1-a}/\phi_+(0)$,
we obtain the statement
of the theorem, 
\eqref{eq:result_thm},
\begin{equation}
    \label{eq:result_proof}
\nu_{a,t}(\D y) =
\frac{\P(I_\xi^{-a}(t)\in \D y)t^{\alpha}}{y^a\ell(t)}
\xrightarrow[t \to \infty]{w}
\nu_a(\D y)
\end{equation}
with $\nu_a$ defined by $\nu_a((0,x]) = c_1 h_a^x(0)/\phi_+(0)$.

\textbf{Case 2: $\alpha$ is integer, i.e. $\alpha \in \{2, 3, \dots\}$.}

We note that now $n = \alpha - 1 \geq 1$ and $\P(\xi_t \geq 0) \simi 
c_1\ell(t)/t^n$.
The analogue of
\eqref{eq:deriv_phi_inf} in this
case is that
\begin{equation}  \label{eq:deriv_phi_inf_case2}
\labsrabs{\phiplusderiv{0+}{k}} < \infty, \quad
k = 0, 1, \dots, n-1,
\end{equation} 
and the difference is that
we do not have for sure that
the $n$th derivative of $\phi_+$
is infinite at zero. However,
by \eqref{eq:ln_kappa},
since $t^{n-1} \P(\xi_t\geq 0)
\simi \ell(t)/t$, by
Proposition
\ref{prop: de Haan 1.5.9a},
\begin{equation}
    \label{eq: t^n-1 slowly var}
\lbrb{\ln \phi_+}^{(n)}(q)
=\int_0^\infty e^{-qt}t^{n-1}\P(\xi_t \geq 0) \D t \in SV_0
\end{equation}
and, as before, using  
Faà di Bruno's formula, \eqref{eq:Faa_di_Bruno_ln},
we get that
$\phiplusderiv{q}{n} \in SV_0$, so
for small $q$, by Potter's bounds,
see Proposition \ref{prop:Potter's bounds}, $\labsrabs{\phiplusderiv{q}{n}} = \bo{q^{-1/4}}$. Now, from Faà di Bruno's formula for $1/\phi_+$,
see \eqref{eq:FdB_1/f}, it follows that
\begin{equation}
    \label{eq: 1/phi+ big O}
    \labsrabs{\lbrb{\frac{1}{\phi+}}^{(n)}(q)} = \bo{q^{-1/4}}.
\end{equation}
Next, let us define, for $q \geq 0$, $x \in (0,1)$, and $\lambda\geq1$,
\begin{equation}
    \label{eq: def Mtilde}
    \begin{split}
\widetilde{M}_{n}(q, x) &:= \partialder{n}{q}
{\MBG{q}{x}} - 
\partialder{n}{q}
{\MBG{0}{x}}, \quad\text{and}\\
K_{\lambda, n}(q) &:= 
\lbrb{1/\phi_+}^{(n)}\lbrb{q/\lambda} -
\lbrb{1/\phi_+}^{(n)}(q).
    \end{split}
\end{equation}
By (\ref{it:cor_der_Mq_i}) from Corollary \ref{cor:der_Mq},
$\partialder{k}{q}
{\MBG{0}{x}}$ are finite
for $0 \leq k \leq n$, so
$\widetilde{M}_{k}(q, x)$
are well-defined for  these $k$. Moreover, by the
mean value theorem, for $0 \leq k \leq n-1$, there
exist $\widetilde{\eta}_k \in (0,q)$ such that
\[
\labsrabs{\widetilde{M}_{k}(q, x)} =\labsrabs{ q\partialder{k+1}{q}
{\MBG{\widetilde{\eta}_k}{x}} }= \bo{q} = \bo{q^{3/4}}.
\]
In a similar spirit, by
\eqref{eq:deriv_phi_inf_case2},
$K_{\lambda, k}(q)$ are
well defined for $0 \leq k  \leq n-1$,
and by the mean value theorem, for these $k$
and $q > 0$,
there exist $\eta_k \in (q/\lambda, q)$ such that
\[
\labsrabs{K_{\lambda, k}(q)} = \labsrabs{q\lbrb{\frac{1}{\lambda} -  1}\lbrb{\frac{1}{\phi_+}}^{(k+1)}(\eta_k)}
= \bo{q^{3/4}}
\]
by \eqref{eq:deriv_phi_inf_case2} and
 \eqref{eq: 1/phi+ big O}. To summarise:
\begin{equation}
\label{eq:M_kappa}
    \labsrabs{\widetilde{M}_{k}(q, x) }\!=\! \bo{q^{3/4}},
    \labsrabs{K_{\lambda, k}(q)} \!= \! \bo{q^{3/4}}
    \,\,\text{for}\,\,
    0 \leq k \leq n-1;
    \,\,\,\text{and }
    \widetilde{M}_{n}(q, x) = \so{1}\!.
\end{equation}
Let 
\begin{equation}
    \label{eq:def G_lambda}G_{\lambda,n}(q):=\widehat{g_a}^{(n)}(q/\lambda)
- \widehat{g_a}^{(n)}(q) .
\end{equation}
We can 
use once again
\eqref{eq:der_Mq/k} 
and then
\eqref{eq:deriv_phi_inf_case2} and the terms of order $\bo{q^{3/4}}$ from \eqref{eq:M_kappa} to obtain that, for $q \to 0$,
\begin{equation}
\label{eq:G_lambda calculation}
\begin{split}
 &G_{\lambda,n}(q)=
\sum_{k = 0}^n
\binom{n}{k}
\bigg(
K_{\lambda, k}(q)
\partialder{n-k}{q}
{\MBG{0}{1-a}}
\\
&\hphantom{{}=
\sum_{k = 0}^n
\binom{n}{k}
}
+
\lbrb{1/\phi_+}^{(k)}\lbrb{{q/\lambda}}
\widetilde{M}_{n-k}(q/\lambda, 1-a)
-
\lbrb{1/\phi_+
}^{(k)}(q)
\widetilde{M}_{ n-k}(q,1-a)\bigg)
\\
&=
K_{\lambda, n}(q)\MBG{0}{1-a}
+ \bo{q^{1/2}} \\
&\hphantom{{}=
\sum_{k = 0}^n
\binom{n}{k}
}
+ 
\lbrb{1/\phi_+
}^{(n)}(q/\lambda)
\widetilde{M}_{ 0}(q/\lambda,1-a)
-
\lbrb{1/\phi_+}^{(n)}(q)
\widetilde{M}_{0}(q, 1-a)\\
&\hphantom{{}=
\sum_{k = 0}^n
\binom{n}{k}
}\hspace{5.5em}+
(1/\phi_+)(q/\lambda)
\widetilde{M}_{n}(q/\lambda, 1-a)
-
(1/\phi_+)(q)
\widetilde{M}_{n}(q, 1-a).
\end{split}
\end{equation}
For the last 4 terms above, we
calculate
\begin{equation}
\label{eq:G_lambda 1st step}
\begin{split}
&\lbrb{1/\phi_+
}^{(n)}(q/\lambda)
\widetilde{M}_{ 0}(q/\lambda,1-a)
-
\lbrb{1/\phi_+}^{(n)}(q)
\widetilde{M}_{0}(q, 1-a)\\
&\qquad\qquad+
(1/\phi_+)(q/\lambda)
\widetilde{M}_{n}(q/\lambda, 1-a)
-
(1/\phi_+)(q)
\widetilde{M}_{n}(q, 1-a)\\
&\quad= K_{\lambda, n}(q)\widetilde{M}_{0}(q/\lambda, 1-a)
+ \lbrb{1/\phi_+}^{(n)}(q)\lbrb{
\widetilde{M}_{0}(q/\lambda, 1-a) - \widetilde{M}_{0}(q, 1-a)}\\
&\qquad\qquad
+ K_{\lambda, 0}(q) 
\widetilde{M}_{n}(q/\lambda, 1-a)
+ (1/\phi_+)(q)
\lbrb{\widetilde{M}_{n}(q/\lambda, 1-a) - \widetilde{M}_{n}(q, 1-a)}\\
&\quad=K_{\lambda, n}(q)\bo{q^{1/2}} + \bo{q^{-1/2}}\bo{q} + \bo{q^{1/2}}\so{1}
\\
&\qquad\qquad+\bo{1}\lbrb{\widetilde{M}_{n}(q/\lambda, 1-a) - \widetilde{M}_{n}(q, 1-a)}\\
&\quad= \so{K_{\lambda,n}(q)} + \bo{q^{1/2}}
+\bo{1}\lbrb{\widetilde{M}_{n}(q/\lambda, 1-a)
 - \widetilde{M}_{n}(q, 1-a)},
\end{split}
\end{equation}
where we have substituted
\eqref{eq: 1/phi+ big O}
and  \eqref{eq:M_kappa}, and have
 used the mean value theorem once again for $\widetilde{M}_{0}(\cdot, 1-a)$. For the last difference,
 substituting the definition of 
 $\widetilde{M_n}$, \eqref{eq: def Mtilde},
 by the mean value theorem,
 there exists an $\eta_q \in (q/\lambda, q)$ such that
 \begin{equation*}
     \begin{split}
          \widetilde{M}_{n}\lbrb{
 \frac{q}{\lambda}, 1-a}
 - \widetilde{M}_{n}(q, 1-a)
 &=
 \partialder{n}{q}
{\MBG{\frac{q}{\lambda}}{1-a}} - 
\partialder{n}{q}
{\MBG{q}{1-a}}\\
&= q\lbrb{\frac{1}{\lambda} -  1}\partialder{n+1}{q}
{\MBG{\eta_q}{x}},
     \end{split}
 \end{equation*}
which from item 
\ref{it:cor_der_Mq_ii'} of Corollary \ref{cor:der_Mq} with $\delta=1/4$
is $\bo{q(U^{*n}_q(\R_+) + q^{-1/4})}$. Moreover, 
from Proposition \ref{prop:conv}, we know that
$U^{*n}_q(\R_+)
= \int_0^\infty  e^{-qt}t^{n-1} \P(\xi_t \geq 0)\D t$, which in our case,
see \eqref{eq: t^n-1 slowly var}, is a slowly varying function as $q \to 0$, and is therefore $\bo{q^{-\delta}}$ for every
$\delta > 0$ by Proposition \ref{prop:Potter's bounds}. This leads to
\begin{equation}
\label{eq: M tilde bigO}\widetilde{M}_{n}(q/\lambda, 1-a)
 - \widetilde{M}_{n}(q, 1-a)
\!=\!
\bo{q\lbrb{U^{*n}_q(\R_+)\! +\! q^{-1/4}}} \!=\!
\bo{q\lbrb{q^{-1/4}}} \!=\! \bo{q^{3/4}}\!.
\end{equation}
Substituting the above and
\eqref{eq:G_lambda 1st step} into \eqref{eq:G_lambda calculation}, we get that
\begin{equation}
\label{eq:G_lambda asymp terms}
    G_{\lambda,n}(q) \simo
K_{\lambda, n}(q)\MBG{0}{1-a}
+\so{K_{\lambda, n}(q)} + \bo{q^{3/4}}.
\end{equation}
Now, we try to obtain the asymptotics of $K_{\lambda, n}$ via a link
with the representation \eqref{eq:ln_kappa}: first, let us note that from it, for $q>0$,
\begin{equation}
	\begin{split} 
\lbrb{\ln \phi_+}^{(n)}(q/\lambda)-
\lbrb{\ln \phi_+}^{(n)}(q)
&=
(-1)^{n+1} \int_0^\infty \lbrb{e^{-qt/\lambda} -
e^{-qt}}t^{n-1} \P(\xi_t \geq 0) \D t\\
&= (-1)^{n+1}q \int_0^\infty t^{n-1} \P(\xi_t \geq 0)
\int_{t/\lambda}^t e^{-qs} \D s \D t\\
&=
(-1)^{n+1}q \int_0^\infty  e^{-qs} 
\int_{s}^{\lambda s}
t^{n-1} \P(\xi_t \geq 0)
\D t
\D s.
\end{split} 
\end{equation} 
Moreover, 
by
\eqref{eq:asymp_xi_t_0} and
Proposition \ref{prop: de Haan log p.127},
$\int_s^{\lambda s} t^{n-1}
\P(\xi_t \geq 0) \D t\simi
c_1 \ln(\lambda) \ell(s)$, so, similar to \eqref{eq:ln_kappa_case1},
by Karamata's theorem \ref{thm:Karamata_1.5.8.}
and a standard Tauberian theorem
\ref{thm:Tauberian_1.7.1},
\begin{equation}
\label{eq:dif_ln_asymp}
\lbrb{\ln \phi_+}^{(n)}(q/\lambda)-
\lbrb{\ln \phi_+}^{(n)}(q)
 \simo(-1)^{n+1} c_1
 \ln(\lambda)\ell(1/q).   
\end{equation}
Next, note that we can represent
\begin{equation}
	\begin{split} 
\lbrb{\ln f}^{(n)} = \lbrb{-f (1/f)'}^{(n-1)}
=-
\sum_{k=0}^{n-1}
\binom{n-1}{k}
f^{(k)} (1/f)^{(n-k)},
\end{split} 
\end{equation} 
so we have that
\begin{equation}
	\begin{split} 
&\lbrb{\ln \phi_+}^{(n)}(q/\lambda)
-
\lbrb{\ln \phi_+}^{(n)}(q)
\\&\,\,= -\sum_{k=0}^{n-1}
\binom{n-1}{k}
\lbrb{K_{\lambda,n-k}(q)
\phi^{(k)}_+(q/\lambda)
+
\lbrb{
\phiplusderiv{q/\lambda}{k} - \phiplusderiv{q}{k}
}
\lbrb{1/\phiplus{q}}^{(n-k)}
}\\
&\,\,=
-K_{\lambda, n}(q)\phiplus{q/\lambda}
+ \bo{q^{3/4}},
\end{split} 
\end{equation} 
where we have used \eqref{eq:deriv_phi_inf_case2}, \eqref{eq:M_kappa},
and the mean value theorem, combined with
\eqref{eq:deriv_phi_inf_case2} and
\eqref{eq: 1/phi+ big O}. Substituting in \eqref{eq:dif_ln_asymp}, we get that
\begin{equation}
\label{eq:asymp_kappa_lambda}
K_{\lambda, n}(q) 
\simo 
(-1)^n c_1 \ln(\lambda)\phi^{-1}_+(0)
    \ell(1/q).
\end{equation}
Thus, from \eqref{eq:G_lambda asymp terms}, because slowly
varying functions grow slower
than polynomials, see
again Proposition
\ref{eq:Potter's bounds}, we get that
\begin{equation}
    \label{eq:asymp_K}
    G_{\lambda,n}(q) \simo K_{\lambda,n}(q) \MBG{0}{1-a}
\simo (-1)^n \MBG{0}{1-a} c_1 \ln(\lambda)\phi^{-1}_+(0)
\ell(1/q).
\end{equation}
Let us recall that we defined previously in
\eqref{eq:def_g_an} the function $g_{a, n}$ by
\[ g_{a,n}(x) :=
  \int_x^\infty \int_{s_1}^\infty \dots \int_{s_{n-1}}^ \infty 
  g_a(s_n)
\D s_{n}   \dots \D s_2 \D s_1.\]
We introduce now, for $\lambda \geq 1$,
\begin{equation}
    \label{eq:def R}
    R(t) := \int_0^t g_{a,n}(s) \D s,
\quad
and
\quad R_\lambda(t) :=
\int_t^ {\lambda t} g_{a,n}(s) \D s,
\end{equation}
which are well-defined by Lemma \ref{lemma:bound_E}.
We have that, for $q>0$,
\begin{equation}\label{eq:R_lambda}
	\begin{split}   
q \widehat{R_\lambda}(q)& =
q \IntOI e^{-qt} \int_t^ {\lambda t} g_{a,n}(s) \D s \D t \\
&= 
q \IntOI g_{a,n} (s) \int_{s/\lambda}^s e^{-qt} \D t \D s = 
\widehat{g_{a,n}}\lbrb{q/\lambda} - \widehat{g_{a,n}}(q),
\end{split} 
\end{equation} 
which by Proposition \ref{prop:LaplaceT} and 
the definition of $G_{\lambda,n}$,
see
\eqref{eq:def G_lambda},
is equal to
\begin{equation*}\label{eq:integer_dif}
 \frac{(-1)^{n}}{(n-1)!}\int_0^1
 G_{\lambda,n}(qv)(1-v)^{n-1} \D v,
\end{equation*}
and thus, by \eqref{eq:asymp_K} and Proposition
\ref{prop:RV},
\begin{equation}
    \label{eq:qLaplaceR}
q \widehat{R_\lambda}(q) \simo
\frac{c_1\MBG{0}{1-a}
 \ln(\lambda)}{n!
\phiplus{0}}
\ell\lbrb{
1/q}.
\end{equation}
Next, 
in the last result
we will
use a generalised
Tauberian theorem
for $\widehat{R_\lambda}$,
Theorem \ref{thm:extended Tauberian, 1.7.6},
 for
which we will need
to check the regularity
condition
\eqref{eq:slow decrease condtion}:
\begin{equation*}
    \lim_{\mu\to 
1+}
\liminfi{x}\inf_{1\leq t
\leq \mu}
\frac{R_\lambda(tx)-R_\lambda(x)}{\ell(x)}\geq 0.
\end{equation*}
Since $R$ is increasing
and $g_{a,n}$
is decreasing,
we have that, for 
$\mu \geq t \geq 1$,
\begin{equation*}
    \begin{split}
    R_\lambda(tx)-R_\lambda(x)
    &=
    R(\lambda tx)-R(\lambda x)-(R(tx)-R(x))
    \geq -(R(tx)-R(x))
\\&\geq -(R(\mu x)-R(x))
     = -\int_x^{\mu x}
     g_{a,n}(s)\D s
     \geq (1- \mu)xg_{a,n}(x),
    \end{split}
\end{equation*}
so it would be enough to check that
$xg_{a,n}(x) = \bo{\ell(x)}$
as $x \to \infty$. This is indeed true because
from Lemma \ref{lemma:bound_E},
$\Ebb{I_\xi^{-a}(t)} \leq c_2 \ell(t)/t^{n+1}$ for
some $c_2 >0$, and
using this in the definition of
$g_{a,n}$, \eqref{eq:def_g_an},
we have that, for large $x$,
     \begin{equation}
         \label{eq: xg_a, bigO}
     x g_{a,n}(x) =
  x\int_x^\infty \int_{s_1}^\infty \dots \int_{s_{n-1}}^ \infty 
  \Ebb{I_\xi^{-a}(s_n)} \D s_n \dots \D s_2 \D s_1
  \leq
  c_{3} \ell(x), 
  \end{equation}
where $c_{3} > 0$, and the calculation is
valid by Theorem \ref{thm:Karamata_1.5.10.}.
Therefore, from the extended Tauberian theorem
\ref{thm:extended Tauberian, 1.7.6},
\[R_\lambda(t)=R(\lambda t)-R(t) \simi c_1 \frac{\MBG{0}{1-a}
 \ln(\lambda)}{n!
\phiplus{0}}
\ell\lbrb{
t},\]
which means by Definition \ref{def: Pi_ell} that
$R \in \Pi_\ell$, and
hence, by 
\eqref{eq:def R} and Theorem \ref{thm:de Haan monotone density, 3.6.8},
\[g_{a,n}(t) \simi c_1 \frac{\MBG{0}{1-a}}{n!
\phiplus{0}}
t^{-1}\ell(t). \]
After $n$ applications of the monotone density theorem \ref{thm:Monotone density, 1.7.2} in the last equivalence, we get that
\begin{equation}   \label{eq:g_a_integer}
    g_a(t) = \Ebb{I^{-a}_\xi(t)} 
\simi 
\frac{c_1 \MBG{0}{1-a}}{\phiplus{0}}
t^{-(n+1)}\ell(t),
\end{equation}
which, as $\alpha = n+1$, is the same result as in the first case,
see \eqref{eq:asymp_g}. The remaining arguments are also similar to past ones: first, note that
\eqref{eq:gaxn} holds:
 for $q>0$ and $b \in (a-1, 0)$,
\begin{equation}
\label{eq: gaxn bis}
\widehat{g_a^x} ^{(n)}(q)
  = 
  \lbrb{\frac{h_a^x(q)}{\phiplus{q}}}^{(n)}
   \text{\quad with \quad} h_a^x(q) := -\frac{1}{2\pi i}\int_{\Re(z) = b}\frac{x^{-z}}z \MBG{q}{ z+1-a} \D z,
\end{equation}
and let us define, similar to
$G_{\lambda,n}$, see \eqref{eq:def G_lambda},
\begin{equation*}
G^x_{\lambda, n}(q):=\widehat{g_a^x}^{(n)}(q/\lambda)
- \widehat{g_a^x}^{(n)}(q) 
\end{equation*}
with $g_a^x$ defined in \eqref{def:gx}.
We
will use similar arguments to the ones that starting with
\eqref{eq:der_Mq/k}, lead from
\eqref{eq:def G_lambda} to \eqref{eq:asymp_K}, i.e.
from
\[
\widehat{g_a}^{(n)}(q) =
\lbrb{
\frac{\MBG{q}{1-a}}{\phiplus{q}}
}^{(n)}\quad \text{to}
\quad
G_{\lambda,n}(q) \simo K_{\lambda, n}(q) \MBG{0}{1-a},
\]
to obtain that,
using $h_a^x(q)$ instead of
$\MBG{q}{1-a}$, that  \eqref{eq: gaxn bis} leads to
\[G^x_{\lambda, n}(q) \simo
K_{\lambda, n}(q) h_a^x(0).
\]
We will follow in detail the mentioned chain of arguments: define first an analogue of the defined in \eqref{eq: def Mtilde} $\widetilde{M_k}$: for $k \leq n$,
\[
\widetilde{H}_{n}(q, x) :=
\lbrb{h_a^x}^{(n)}(q) - \lbrb{h_a^x}^{(n)}(0), 
\]
which is well defined by \eqref{eq:h_a^x limsup}. Therefore,
as in \eqref{eq:M_kappa}, by the mean value theorem
\begin{equation*}
    \labsrabs{\widetilde{H}_{k}(q, x) }= \bo{q^{3/4}},
    \labsrabs{K_{\lambda, k}(q)} =  \bo{q^{3/4}}
    \,\,\text{for}\,\,
    0 \leq k \leq n-1;
    \,\text{and}\,
    \widetilde{H}_{n}(q, x)\! =\! \so{1}
\end{equation*}
so \eqref{eq:G_lambda calculation} and
\eqref{eq:G_lambda 1st step} hold with $\widetilde{H}$ instead of $\widetilde{M}$, and we are left again to treat the term
\[\widetilde{H}_{n}(q/\lambda, 1-a)
 - \widetilde{H}_{n}(q, 1-a),\]
 which is equal by definition to
\[
-\frac{1}{2\pi i}\int_{\Rez = b}\frac{(1-a)^{-z}}z\lbrb{\partialder{n}{q}
{\MBG{\frac{q}{\lambda}}{ z+1-a}}-
\partialder{n}{q}
{\MBG{q}{ z+1-a}}}
     \D z.\]
     Before in
\eqref{eq: M tilde bigO}, we 
used the mean value theorem for the $n$th $q$-derivative of $M_\Psi$ to show that
\[
\widetilde{M}_{n}(q/\lambda, 1-a)
 - \widetilde{M}_{n}(q, 1-a)
=
 \bo{q^{3/4}}.
\]
As 
$M_\Psi$ is complex valued, 
we cannot use
the mean value theorem, however we make a similar argument under the sign of the integral. We sketch the proof in the case where
$\xi$ is not a CPP with positive drift, as the arguments in both cases have similar nature. Full details are available in \cite[p.137]{Minchev-2024}. Write, for $x_a := 1-a\in(0,1)$,
\begin{equation}\label{eq:In sup}
    \begin{split}
       \labsrabs{ I_n}&:=\labsrabs{\frac{1}{2\pi i}\int_{\Rez = b}\frac{x_a^{-z}}z\lbrb{\partialder{n}{q}
{\MBG{\frac{q}{\lambda}}{ z+1-a}}-
\partialder{n}{q}
{\MBG{q}{ z+1-a}}}
     \D z}\\
     &\,=\frac{1}{2\pi }\labsrabs{\int_{\Rez = b}\frac{x_a^{-z}}z\int_{q}^{q/\lambda}\partialder{n+1}{q}
{\MBG{w}{z+1-a}}\D w
     \D z}\\
     &\,\leq\frac{1}{2\pi }\int_{-\infty}^\infty\frac{x_a^{-b}}{|b+iy|}q\lbrb{\lambda-1}\sup_{y \in (q/\lambda,q)}\labsrabs{\partialder{n+1}{q}
{\MBG{y}{b+1-a + iy}}}
     \D y.
    \end{split}
\end{equation} 
Using logic similar to the one
above \eqref{eq: M tilde bigO}, we will show that,
for sufficiently small $q$, there exists a 
polynomial $Q$ of 
degree $n+1$ such that
\[
\sup_{y \in (q,q/\lambda)}\labsrabs{\partialder{n+1}{q}
{ \MBG{q}{z}}} 
\leq
q^{-1/4}Q_{ \Rez, n}\lbrb{\labsrabs{z}}\labsrabs{ \MBG{q}{z}},
\]
First, by \textit{\ref{it:cor_der_Mq_ii'}} of Corollary \ref{cor:der_Mq} with $\delta=1/4$, we have that
there exist a polynomial $P_{\Rez, n}$ of
    degree $n+1$ such that,  for all $q>0$
    and $\Rez \in (0,1)$,
    $$\labsrabs{\partialder{n+1}{q}
{ \MBG{q}{z}}} \leq \lbrb{U^{*n}_q(\Oiclosed) + q^{-1/4}}P_{ \Rez, n}\lbrb{\labsrabs{z}}\labsrabs{ \MBG{q}{z}}.$$
Next, 
as before $
U^{*n}_q(\R_+)
=
\bo{q^{-1/4}}
$,
so for sufficiently small $q$,
    $$\labsrabs{\partialder{n+1}{q}
{ \MBG{q}{z}}} \leq 2q^{-1/4}P_{ \Rez, n}\lbrb{\labsrabs{z}}\labsrabs{ \MBG{q}{z}}.$$
Using this bound for over
$( q/\lambda, q)$, we get that, for small $q$,
\begin{align*}
    \sup_{y \in (q,q/\lambda)}\labsrabs{\partialder{n+1}{q}
{ \MBG{q}{z}}} &
\leq
2\lambda^{1/4}
 q^{-1/4}P_{ \Rez, n}\lbrb{\labsrabs{z}}\labsrabs{ \MBG{q}{z}}\\
&=: q^{-1/4}Q_{ \Rez, n}\lbrb{\labsrabs{z}}\labsrabs{ \MBG{q}{z}},
\end{align*}
with $Q$ a polynomial of
degree $n+1$ as desired.
Substituting in \eqref{eq:In sup}, we get for small $q$ that
\begin{equation} \label{eq: In case}
    \begin{split}
       \labsrabs{ I_n}&\leq\frac{q^{3/4}(\lambda-1)}{2\pi }\int_{-\infty}^\infty\frac{x_a^{-b}}{|b+iy|}Q_{ \Rez, n}\lbrb{\labsrabs{b+1-a+iy}}\labsrabs{ \MBG{q}{b+1-a+iy}}
     \D y.
    \end{split}
\end{equation} 
We have already proven that integrals of the last type are finite,
see below \eqref{eq: finite integral}, giving the desired
\[
  \widetilde{H}_{n}\lbrb{
 \frac{q}{\lambda}, 1-a}
 - \widetilde{H}_{n}(q, 1-a)
=\bo{q^{3/4}},
\] and as a conseqience the modified
\eqref{eq:G_lambda asymp terms}:
\begin{equation*}
    G_{\lambda,n}^x(q) \simo
K_{\lambda, n}(q)h_a^x(0)
+\so{K_{\lambda, n}(q)} + \bo{q^{3/4}},
\end{equation*}
which, by \eqref{eq:asymp_kappa_lambda}, 
gives 
\[G^x_{\lambda, n}(q) \simo
K_{\lambda, n}(q) h_a^x(0).
\]
Continuing the analogy, following
   \eqref{eq:def R},
let us define
\[R^x(t) := \int_0^t g^x_{a,n}(s) \D s, \quad
\text{and}
\quad R^x_\lambda(t) :=
\int_t^ {\lambda t} g^x_{a,n}(s) \D s,\] 
which are well-defined because
$g^x_{a,n} \leq g_{a,n}$. The same inequality also gives that $y g^x_{a,n}(y)\leq y g_{a,n}(y) = \bo{\ell(y)}$ as $y \to \infty$, see \eqref{eq: xg_a, bigO}, so repeating the
arguments from \eqref{eq:R_lambda} to \eqref{eq:g_a_integer}, we obtain the similar to \eqref{eq:g_a_integer} result,
\begin{equation*}
    g^x_a(t) = \Ebb{I_\xi^{-a}(t) \ind{I_\xi(t) \leq x}}
\simi 
\frac{c_1 h_a^x(0)}{\phiplus{0}}t^{-(n+1)} \ell(t).
\end{equation*}
Then, in the same way as from \eqref{eq:asymp_gax}
to \eqref{eq:result_proof}, we obtain the same result:
\begin{equation*}
\nu_{a,t}(\D y) =
\frac{\P(I_\xi^{-a}(t)\in \D y)t^{\alpha}}{y^a\ell(t)}
\xrightarrow[t \to \infty]{w}
\nu_a(\D y).
\end{equation*}
\end{proof}     
\fakesubsection{Proof of Theorem \ref{thm:mainPr}}
\begin{proof}[Proof of Theorem \ref{thm:mainPr}]\label{proof 3.5} Set in the notation of Theorem \ref{thm:main} and the proof above $\tilde{\nu}_{a}((0,x]):=\nu_a((0,x])/\nu_a(0,\infty)$, and let $W$ be the positive random variable with law $\tilde{\nu}_{a}$.
Then, from \eqref{eq: nu_a} of Theorem \ref{thm:main} and \eqref{eq:moments} of Corollary \ref{eq:moments}, we see that, for any $x\in\lbrb{0,\infty}$,
\[F_W(x):=\tilde{\nu}_{a}((0,x])=-\frac{1}{2\pi i}\int_{\Re z = b}\frac{x^{-z}}z \frac{\MBG{0}{ z+1-a}}{\MBG{0}{ 1-a}} \D z\]
for $b \in \Cb_{(a-1,0)}$. Since this formula is precisely the formula for Mellin inversion, we conclude that
\begin{equation*}\label{eq:M_W2}
 \Mcc_{F_W}(z) =-\frac{1}{z} \frac{\MBG{0}{ z+1-a}}{\MBG{0}{ 1-a}}=-\frac{1}{z}\frac{\Gamma(z+1-a)W_{\phi_{0,+}}(1-a)}{W_{\phi_{0,+}}(z+1-a)\Gamma(1-a)}\frac{W_{\phi_{0,-}}(a-z)}{W_{\phi_{0,-}}(a)},  
\end{equation*}
where we have substituted the definition of $M_\Psi$, \eqref{eq:M_I_psi}. Now, setting $\phi_{\pm}:=\phi_{0,\pm}$, we further derive from above
    \begin{equation}\label{eq:M_W}
        \Mcc_{F_W}(z) =-\frac{1}{z}\frac{\Ebb{I^{z}_{\phi_+}I^{-a}_{\phi_+}}}{\Ebb{I^{-a}_{\phi_+}}}\frac{\Ebb{(Y^{-1}_{\phi_-})^{z}(Y^{-1}_{\phi_-})^{1-a}}}{\Ebb{(Y^{-1}_{\phi_-})^{1-a}}},
        \end{equation}
        where $Y_{\phi_-}$ is the random variable for $q=0$ in the definition of the bivariate \BG function, see \eqref{eq:rec}, or just below \cite[(4.2)]{PatieSavov2018}, and $I_{\phi_+}$ is the exponential functional of the killed subordinator (because $\phi_+(0)>0$) pertaining to $\phi_+$, see \cite[Theorem 2.22]{PatieSavov2018}. The latter arguments are common knowledge in the area of exponential functionals. Clearly, then \eqref{eq:M_W} can be further rewritten as 
         \begin{equation}\label{eq:M_W1}
        \Mcc_{F_W}(z) =-\frac{1}{z}\Mcc_{\Bc_{-a}I_{\phi_+}}(z+1)\Mcc_{\Bc_{1-a}Y^{-1}_{\phi_-}}(z+1)=-\frac{1}{z}\Mcc_{\Bc_{-a}I_{\phi_+}\times \Bc_{1-a}Y^{-1}_{\phi_-}}(z+1),
        \end{equation}
        with the operator $\Bc_{\cdot}$ standing for the size-biased versions of the respective random variables, see \eqref{eq:sizeBiased}. Recalling that the relation between the Mellin transform of the cumulative distribution function and the law of random variable, see \cite[(7.12)]{PatieSavov2018}, matches the expression at the right-hand side in \eqref{eq:M_W1}, we conclude that $W\stackrel{d}{=}\Bc_{-a}I_{\phi_+}\times \Bc_{1-a}Y^{-1}_{\phi_-}$ and \eqref{eq:mainPr} is established. 
        Given that \eqref{eq:M_W1} holds, then the infinite differentiability of $F$ follows immediately from \cite[item (3) from Theorem 2.4]{PatieSavov2018} since it applies with $N_\Psi=\infty$ therein, see \cite[(2.18) of Theorem 2.3]{PatieSavov2018}. This concludes the proof of the theorem.
\end{proof}
\fakesubsection{Proof of Theorem \ref{thm:upperB}}
\begin{proof}[Proof of Theorem \ref{thm:upperB}]
\label{proof 3.8}
    Obviously, it suffices to prove the theorem for $F(x)=x^{-a}$, $a\in\lbrb{0,1}$. Then, \eqref{eq:LT} reads off
    \begin{equation*}
    \begin{string}
     \IntOI e^{-qt}\Ebb{I_\xi^{-a}(t)}\D t=\frac{\MBG{q}{1-a}}{\phi_+(q,0)}.
    \end{string}
\end{equation*}
    If \eqref{condi:upperB} holds for some $n\geq 1$, then, from item (\ref{it:cor_der_Mq_i}) of Corollary \ref{cor:der_Mq}, we have that, for $0 \leq k \leq n+1$,  $\partialder{k}{q}
{ \MBG{q}{z}}$ are finite and right-continuous at $q=0$, and from \eqref{eq:PandInt1} we know that $\abs{\phi^{(n)}_+(0)}<\infty$. From the dominated convergence theorem, we immediately deduce that, for $q>0$,
\[ (-1)^n\IntOI t^ne^{-qt}\Ebb{I_\xi^{-a}(t)}\D t=\partialder{n}{q}{\frac{\MBG{q}{1-a}}{\phi_+(q,0)}},\]
which augmented by the monotone convergence theorem and the finiteness and the right-continuity of the derivatives yields
\[(-1)^n\IntOI t^n\Ebb{I_\xi^{-a}(t)}\D t=\left.\partialder{n}{q}{\frac{\MBG{q}{1-a}}{\phi_+(q,0)}}\right|_{q=0}.\]
The finiteness of the right-hand side and the fact that $\Ebb{I_\xi^{-a}(t)}$ is monotone decreasing leads to
\[\sum_{m\geq 1}\Ebb{I_\xi^{-a}(m+1)}m^n\leq\IntOI t^n\Ebb{I_\xi^{-a}(t)}\D t<\infty ,\]
which settles the claim in \eqref{res:upperB}. The final claim follows from the expression for the derivatives above with $n=0$, the finiteness of $\phi_+(0,0)$, and the monotone convergence theorem.
\end{proof}

\section{Proofs of auxiliary results } \label{sec: auxiliary}

 \begin{proof}[Proof of Proposition \ref{prop:conv}]
We proceed by induction for the case $q>0$. The base case $n=1$ is valid from the definition of potential measures. For $n \in \N$, we calculate
\begin{equation*}
	\begin{split}
U_q^{*(n+1)} &(\D x) = \int_{-\infty}^{\infty}U_q^{*n}(\D y) U_q(\D x - y) \\
&= \frac{1}{(n-1)!}\int_{-\infty}^{\infty} \IntOI \IntOI e^{-q(t_1 + t_2)}t_1^{n-1} \P(\xi_{t_1} \in \D y)
\P(\xi_{t_2} \in \D x -y) \D t_1 \D t_2\\
&= \frac{1}{(n-1)!} \IntOI \int_0^t e^{-qt}t_1^{n-1}\P(\xi_t \in \D x) \D t_1 \D t
\\
&= \frac{1}{n!} \IntOI e^{-qt} t^n \P(\xi_t \in \D x)\D t,
	\end{split} 
\end{equation*} 
where we have used the induction hypotheses 
for $U_q^{*n}$, and afterwards,
we have employed that \[\int_{-\infty}^{\infty} 
\P(\xi_{t_2} \in \D x -y)\Pbb{\xi_{t_1}\in \D y}=\Pbb{\xi_{t_1+t_2}\in \D x},\]
and changed variables as follows $t_1+t_2=t,t_1=t_1$,
and finally we have exchanged the order of integration. The case $q=0$ is valid by taking a limit and using the monotone convergence theorem.
\end{proof}

\begin{proof}[Proof of Lemma \ref{lem:uz}]
For $\Re(z)>-1$, we can see that $u$ is bounded and continuous on $[0,\infty)$ by the Taylor estimates
\begin{equation}\label{eq:uzAsymp}
    \begin{split}
        u_z(y)\simo \frac{z^2-z}{2}y
         \quad\text{and}
  \quad
  u_z(y)\simi \frac{(z-1)\ind{\Re(z)\geq 
  0}
  +\ind{\Re(z)<0}}
  {e^{y+\Re(z)y\ind{\Re(z)<0}}}.
    \end{split}
\end{equation}
As $v_z(y) := ze^{-y} - u_z(y)$, the same
holds for $v_z$.
For the precise bound, we start by using
the estimation lemma
for complex integrals:
\begin{equation}\label{eq: vz integral bound}
    \begin{split}
        |v_z(y)| = \labsrabs{\frac{1 -e^{-zy}}{e^y - 1}} 
        =
        \labsrabs{\frac{
        \int_0^{zy}
        e^{-x} \D x}{e^y - 1}}
        \leq \frac{|zy|
        \max\curly{ 1,  e^{-y\Re(z)}}
        }{e^y - 1} .
    \end{split} 
\end{equation}
If $\Re(z)\geq 0$, as $y/(e^y-1) \leq e^{-y/2}$, we immediately obtain
the first part of \eqref{eq:uz_bound} with $C_{1, \Re(z)} = 1$, and  
$\epsilon_{1, \Re(z)} = 1/2$. For the second part, bounding, for $y>1=:C_\Re(z)$,
\[
 |v_z(y)| = \labsrabs{\frac{1 -e^{-zy}}{e^y - 1}} 
        \leq
\frac{
        1 + e^{-\Re(z) y}}{e^y - 1}
        \leq \frac{2
        }{e^y - 1}
        \leq
        \frac{2
        }{e^{y/2}},
\]
checking directly the last inequality, we obtain the result with  $C_{2, \Re(z)} = 2$, and  
$\epsilon_{2, \Re(z)} = 1/2$.

Next, if $\Re(z)\in (-1,0)$, \eqref{eq: vz integral bound} gives
\[
 |v_z(y)|\leq |z|\frac{
         ye^{-y\Re(z)}
        }{e^y - 1}.
\]We see that there exists $C'_{\Re(z)} > 0$, such that for $y > C'_{\Re(z)}$, $e^y - 1$ is at least $ye^{(-\Re(z)+ (1+\Re(z))/2)y}$,
and
for $0\leq y \leq C'_{\Re(z)}$, we can just bound from above $y\mapsto  ye^{-y\Re(z)}/(e^y - 1)$ with its
supremum, which is finite since the limit at 0 is finite and
the function is continuous. We obtain that for
some $C''_{\Re(z)}>0$, and
$\epsilon_{\Re(z)}=
(1 + \Re(z))/2 > 0$,
we have that
\[\f{ye^{-y\Re(z)}}{e^y - 1}\leq\begin{cases}
			C''_{\Re(z)}, & \text{if $y < C'_{\Re(z)}$};\\
            e^{-\epsilon_{\Re(z)}y}, & \text{if $y \geq C'_{\Re(z)}$}.
		 \end{cases}
   \]
This leads to the desired result with $C_{1, \Re(z)} =\max\curly{1, C''_{\Re(z)}e^{\epsilon_{\Re(z)} C'_{\Re(z)}}}$ and
$\epsilon_{1,\Re(z)} = \epsilon_\Re(z)$. For the second
part of \eqref{eq:uz_bound},
\[
 |v_z(y)| = \labsrabs{\frac{1 -e^{-zy}}{e^y - 1}} 
        \leq
\frac{
        1 + e^{-\Re(z) y}}{e^y - 1}
        \leq \frac{2e^{-\Re(z) y}
        }{e^y - 1}
        \leq
        \frac{2e^{-\Re(z) y}
        }{2e^{(-\Re(z) + \epsilon_\Re(z))y}},
\]
for sufficiently large $y > C_\Re(z)$ for some $C_\Re(z)$,
as $-\Re(z) + \epsilon_\Re(z) = (1-\Re(z))/2<1$, so the second part of \eqref{eq:uz_bound} 
is true with $C_{2, \Re(z)}
= 1$ and $\epsilon_{2,\Re(z)} = \epsilon_\Re(z)$.
\end{proof}

\begin{proof}[Proof of Proposition \ref{prop:LaplaceT}]
We proceed by induction. For $k=1$, by Tonelli's
theorem, we have that
\begin{equation*}
	\begin{split}
\widehat{f_1}(q)  &= \IntOI\! e^{-qs} \!\int_{s}^\infty f(y) \D y \D s =  \IntOI \!f(y) \!\int_{0}^y e^{-qs} \D s \D y = \frac{1}{q}  \IntOI\! f(y)  \lbrb{1 - e^{-qy}} \D y \\&=  \frac{1}{q}\lbrb{\widehat{f}(0) - \widehat{f}(q)  }= -\frac{1}{q}\int_0^q \widehat{f}'(u) \D u = -\int_0^1 \widehat{f}'(qv) \D v .
	\end{split} 
\end{equation*} 

For the induction step, we use similar logic: assume that the result is true for $n = 1, 2, \dots, k$. Then, using it for $n = 1$ and the function $f_k$, we get
\begin{equation*}
	\begin{split}
\widehat{f_{k+1}}(q)  &= \IntOI e^{-qs} \int_{s}^\infty f_k(y) \D y \D s =  -\int_0^1 \widehat{f_k}'(qv) \D v.
	\end{split} 
\end{equation*} 
Now differentiating \eqref{eq:LaplaceT} for $n=k$ in $q$, we substitute in the last equality 
\begin{equation*}
	\begin{split}
-\int_0^1 \widehat{f_k}'(qv) \D v &= \frac{(-1)^{k+1}}{(k-1)!}\int_0^1 \int_0^1  y\widehat{f}^{(k+1)}(qvy)(1-y)^{k-1} \D v \\
&=  \frac{(-1)^{k+1}}{(k-1)!}\int_0^1 \int_0^1  w\widehat{f}^{(k+1)}(qvw)(1-w)^{k-1} \D w \D v\\
&= \frac{(-1)^{k+1}}{(k-1)!}\int_0^1  \widehat{f}^{(k+1)}(qy) \int_y^1 (1-w)^{k-1} \D w \D y\\
&= \frac{(-1)^{k+1}}{k!}\int_0^1  \widehat{f}^{(k+1)}(qy) (1-y)^{k}\D y,
	\end{split} 
\end{equation*} 
where we used the substitution $y = vw$. This concludes the induction step and, respectively, establishes the
claim of the
proposition.
\end{proof}
\begin{proof}[Proof of Proposition \ref{prop:RV}]
    Fix $\varepsilon\in\lbrb{0,1}$ and from the assumptions let $|f(q)|\simo q^{\beta}\ell(q)$ with the slowly varying function $\ell$ being positive in a neighbourhood of zero. Then, from Theorem \ref{thm:UCT_1.2.1} applied to $\ell$ for the compact interval $\lbbrbb{\varepsilon,1}$, we have that
    \[\limo{q}\frac{\int_{\varepsilon}^1f(qv)(1-v)^{n-1}\D v}{f(q)}=\int_{\varepsilon}^1v^{\beta}(1-v)^{n-1} \D v.\]
     With $g(w):=f(w^{-1})\simi w^{-\beta}\ell(w^{-1})$ and $\ell'(w):=\ell(w^{-1})$, we compute   that 
    \begin{equation*}
        \begin{split}
            &\limsupo{q}\abs{\frac{\int_0^{\varepsilon}f(qv)(1-v)^{n-1} \D v}{f(q)}}
            \leq \limsupo{q}\frac{\int_0^{\varepsilon q}\abs{f(w)} \D w}{q\abs{f(q)}}
            \\
            &\quad=\limsupo{q}\frac{\int_{\varepsilon^{-1}  q^{-1}}^\infty\abs{f(w^{-1})}w^{-2}\D w}{q\abs{f(q)}}
            =
            \limsupo{q}\frac{\int_{\varepsilon^{-1} q^{-1}}^\infty\abs{g(w)}w^{-2}\D w}{q\abs{g(q^{-1})}}
            \\&\quad=\limsupi{y}\frac{y\int_{\varepsilon^{-1}  y}^\infty\abs{g(w)}w^{-2}\D w}{\abs{g(\varepsilon^{-1}  y)}}\frac{\abs{g(\varepsilon^{-1}  y)}}{\abs{g(y)}}\\
            &\quad=\limsupi{y}\frac{\int_{\varepsilon^{-1}  y}^\infty\abs{\ell'(w)}w^{-2-\beta}\D w}{\abs{\varepsilon^{\beta} y^{-\beta-1}\ell'(\varepsilon^{-1}y )}}\frac{\abs{g(\varepsilon^{-1}  y)}}{\abs{g(y)}}
            \\&\quad=\varepsilon\limsupi{y}\frac{\int_{\varepsilon^{-1}  y}^\infty\abs{\ell'(w)}w^{-2-\beta}\D w}{\abs{\varepsilon^{\beta+1} y^{-\beta-1}\ell'(\varepsilon^{-1}y )}}\frac{\abs{g(\varepsilon^{-1}  y)}}{\abs{g(y)}}=\frac{\varepsilon^{1+\beta}}{1+\beta},
        \end{split}
    \end{equation*}
    where in the last relation we have employed Theorem \ref{thm:Karamata_1.5.10.} with $\alpha=-\beta-2<-1$, since $\beta\in\lbrbb{-1,0}$, and the fact that $g$ is regularly varying with index $-\beta$.
    From the two relations, we get that, for any $\varepsilon\in\lbrb{0,1}$, 
    \[\limsupo{q}\abs{\frac{\int_{0}^1f(qv)(1-v)^{n-1}\D v}{f(q)}-\int_{0}^1v^{\beta}(1-v)^{n-1}\D v}\leq \frac{\varepsilon^{1+\beta}}{1+\beta}+\int_0^{\varepsilon}v^{\beta}(1-v)^{n-1} \D v.\]
    Since $\varepsilon$ can be set to tend to $0$ and $\beta\in\lbrbb{-1,0}$, we conclude the proof. 
\end{proof}
\begin{proof}[Proof of Lemma \ref{lemma:bound_E}]
Let us note with
 $\Delta \xi_t$ the jump in time $t$ of $\xi$ and define
$\jump{J_1}{\epsilon t}$ $:=$ $\inf\{s\geq 0: \Delta
\xi_s
\geq \epsilon t\}$ for $\epsilon > 0$, i.e. the time of the first jump of size at least $\epsilon t$. We split the considered expectation according to the time of this jump:
\begin{equation}
\label{eq:E_split_Jet}
\Ebb{I^{-a}_\xi(t)}
=
\Ebb{I_\xi^{-a}(t)\ind{\jump{J_1}{\epsilon t}\leq t}}
+
\Ebb{I_\xi^{-a}(t)\ind{\jump{J_1}{\epsilon t}>t}}.
\end{equation}
For the first term, we have that
\begin{equation*}
    \begin{split}
        \Ebb{I_\xi^{-a}(t)\ind{J^{(\epsilon t)}_1\leq t}}&=\int_0^t \Ebb{I_\xi^{-a}(t)\middle|\jump{J_1}{\epsilon t}=s}
        \PiPlus(\epsilon t)
        e^{-\PiPlus(\epsilon t)s} \D s \\
        &\leq \PiPlus(\epsilon t)\IntOI
        \Ebb{I_\xi^{-a}(t)\middle|\jump{J_1}{\epsilon t}=s} \D s.
    \end{split}
\end{equation*}
In the following, we will use the notation $\jump{\xi}{A}$ for the process $\xi$ from
which we have removed jumps larger than $A > 0$. It is immediate that
$\jump{\xi}{A} \leq \xi$, and thus, for each $t > 0$,
\begin{equation}
\label{eq:bounds_small_jumps}
    I_{\jump{\xi}{A}} (t) \geq I_{\xi}(t) \quad 
    \text{and} \quad
I^{-a}_{\jump{\xi}{A}}(t) \leq I_{\xi}^{-a}(t).
\end{equation}
Therefore, for $0<s\leq t$,
we can see that 
\[\Ebb{I_\xi^{-a}(t)\middle| \jump{J_1}{\epsilon t}=s} \leq
\Ebb{I_{\xi}^{-a}(s)\middle| \jump{J_1}{\epsilon t}=s}
= 
\Ebb{I_{\jump{\xi}{\epsilon t}}^{-a}(s)}\leq\Ebb{I_\xi^{-a}(s)}. \]
Moreover, as in \eqref{eq: Pibar}, by \cite[Theorem 1.1]{Watanabe2008} and the assumption of
regularly varying tail, for large $t$, $\overline{\Pi}(\epsilon t)
\leq c_1 \P(\xi_1 > t)$ for some $c_1 > 0$, so we obtain that for large $t$,
\begin{equation}
\label{eq:E_large_jump}
\Ebb{I_\xi^{-a}(t)
\ind{\jump{J_1}
{\epsilon t}\leq t}}
\leq c_1 \P(\xi_1 >  t) \IntOI 
\Ebb{I_\xi^{-a}(s)} \D s = \bo{\P(\xi_1 > t)}, 
\end{equation}
as the integral is finite from \eqref{eq:LT}.

We continue with bounding $\Ebb{I_\xi^{-a}(t)\ind{\jump{J_1}{\epsilon t}>t}}$.
Once again, writing
\begin{align*}
\Ebb{I_\xi^{-a}(t)\ind{\jump{J_1}{\epsilon t}>t}}
&= \P\lbrb{\jump{J_1}{\epsilon t}>t} \Ebb{I_\xi^{-a}(t)\middle|\jump{J_1}{\epsilon t}>t} \\
&\leq \Ebb{I_\xi^{-a}(t)\middle|\jump{J_1}{\epsilon t}>t}
= \Ebb{I_{\jump{\xi}{\epsilon t}}^{-a}(t)},
\end{align*}
we are left with estimating $I_{\jump{\xi}{\epsilon t}}$.
 Let $\jump{J_1}{1} := \inf\{ s \geq 0: \Delta \xi_s \geq 1\} \sim Exp(\PiPlus(1))$,
 and split, for some $c_2 > 0$,
 \begin{equation}
\label{eq:split_J_11}
\begin{split}
      \Ebb{I^{-a}_{\jump{\xi}{\epsilon t}}(t)} 
  =
  \P\lbrb{\jump{J_1}{1}\geq c_2 \ln t}
  &\Ebb{I^{-a}_{\jump{\xi}{\epsilon t}}(t)\middle|\jump{J_1}{1}
  \geq  c_2 \ln t} \\
  &\qquad
  +
  \P\lbrb{\jump{J_1}{1}< c_2 \ln t}
  \Ebb{I^{-a}_{\jump{\xi}{\epsilon t}}(t)\middle|\jump{J_1}{1}< c_2 \ln t}.   
  \end{split}
 \end{equation} 
 As $\jump{J_1}{1} \sim  Exp(\PiPlus(1))$, let us choose $c_2 =(
 \alpha +1)/\PiPlus(1)$, so we have 
\[\P\lbrb{\jump{J_1}{1}\geq c_2 \ln t} = 1/t^{\alpha + 1}.\] Therefore, by \eqref{eq:bounds_small_jumps}, we have for the first term in \eqref{eq:split_J_11} that, for large $t$,
 \[
 \begin{split}
 &\P\lbrb{\jump{J_1}{1}\geq c_2 \ln t}
  \Ebb{I^{-a}_{\jump{\xi}{\epsilon t}}(t)\middle|\jump{J_1}{1} \geq   c_2 \ln t}
    \leq
    \frac{1}{t^{\alpha + 1}}   \Ebb{I^{-a}_{\jump{\xi}{\epsilon t}}(c_2 \ln t)\middle|\jump{J_1}{1} \geq   c_2 \ln t}
    \\
    &\qquad =  \frac{1}{t^{\alpha + 1}} \Ebb{I^{-a}_{\jump{\xi}{1}}(c_2 \ln t)} \leq  \frac{1}{t^{\alpha + 1}} \Ebb{I_\xi^{-a}(1)} = \bo{\frac{\ell(t)}{t^\alpha}}=\bo{\P(\xi_1 \geq t)},
 \end{split}
 \]
 where we used that $1/t = \bo{\ell(t)}$ by Proposition 
\ref{prop:Potter's bounds},
so we have obtained a bound of the needed order.
 
 To work with the second term of \eqref{eq:split_J_11}, we disintegrate
  $\jump{J_1}{1} \sim Exp(\PiPlus(1))$ and use the Markov property at $\jump{J_1}{1}$: for $\eta$
 an independent copy of $\xi$,
\begin{equation}\label{eq:earlyJ_11} 
    \begin{split}
       &\P\lbrb{\jump{J_1}{1}< c_2 \ln t}
  \Ebb{I^{-a}_{\jump{\xi}{\epsilon t}}(t)\middle|\jump{J_1}{1} <   c_2 \ln t}\\
  &\qquad= \P\lbrb{\jump{J_1}{1}< c_2 \ln t}
  \int_0^{c_2 \ln t}
       \Ebb{I^{-a}_{\jump{\xi}{\epsilon t}}(t)\middle|\jump{J_1}{1}=s}  \PiPlus(1)
        e^{-\PiPlus(1)s}
        \D s\\
  &\qquad\leq 
  \PiPlus(1)
  \int_0^{c_2 \ln t}
        \Ebb{\lbrb{
        I_{\jump{\xi}{1}}(s) + e^{-\jump{\xi_s}{1} - 
        \jump{\Delta_s}{\epsilon t}} I_{\jump{\eta}{\epsilon t}}
        (t-s)}^{-a}}
        \D s,
    \end{split}
\end{equation}
where we have used that, conditionally on $\{\jump{J_1}{1}=s\}$,
\[I^{-a}_{\jump{\xi}{\epsilon t}}(t)=I_{\jump{\xi}{1}}(s)+e^{-\jump{\xi_s}{1} - 
        \jump{\Delta_s}{\epsilon t}} \int_{0}^{t-s}e^{-\lbrb{\xi^{(\epsilon t)}_{s+v}-\xi^{(\epsilon t)}_{s}}}\D v\]
        and $\lbrb{\xi^{(\epsilon t)}_{s+v}-\xi^{(\epsilon t)}_{s}}_{v\geq 0}$ is an independent copy of $\jump{\xi}{\epsilon t}$.
Now, let us consider two cases: $\eta^{(\epsilon t)}_{t/4} \leq t\Ebb{ 
\jump{\eta_1}{\epsilon t}} / 8$ and
the opposite. Note that $\Ebb{ 
\jump{\eta_1}{\epsilon t}} \leq  \Ebb{ \eta_1} = \Ebb{ \xi_1} < 0$, so $t\Ebb{ 
\jump{\eta_1}{\epsilon t}} / 8
\to -\infty$
as
$t \to \infty$. We will show that in the first case, the value of $\eta^{(\epsilon t)}_{t/4}$ is small enough to make $I_{\jump{\eta}{\epsilon t}}$
sufficiently large, and afterwards that the probability of the second case is sufficiently
small.

\textbf{Case 1:} 
Under $A_t:= \left\{\eta^{(\epsilon t)}_{t/4} \leq t\Ebb{ 
\jump{\eta_1}{\epsilon t}} / 8\right\}
$. 

In this case, and for $s \leq c_2 \ln t$ and large $t$, we have
by the Markov property at time $t/4$
and \eqref{eq:bounds_small_jumps} that
\begin{equation*}
    \begin{split}
\indset{A_t}I_{\jump{\eta}{\epsilon t}}(t-s) &= \indset{A_t}\int_0^{t-s}
e^{-\jump{\eta}{\epsilon t}_v} \D v
\geq \indset{A_t}\int_{t/4}^{t/2} e^{-\jump{\eta}{\epsilon t}_v} \D v\\
&= \indset{A_t}e^{-\jump{\eta}{\epsilon t}_{t/4}}
\int_{0}^{t/4} e^{-{\widetilde{\eta}}^{(\epsilon t)}_v} \D v
\geq 
e^{-t\Ebb{ 
\jump{\eta_1}{\epsilon t}}/8}
I_{\jump{\widetilde{\eta}}{\epsilon t}}(t/4)\\
&\geq
e^{-t\Ebb{ 
\eta_{1}}/8}
I_{\widetilde{\eta}}(t/4)
\geq e^{-t\Ebb{ 
\eta_{1}}/8}
I_{\widetilde{\eta}}(1),
    \end{split}
\end{equation*}
so
\[
\Ebb{\indset{A_t}I^{-a}_{\jump{\eta}{\epsilon t}}(t-s)}
\leq e^{at\Ebb{ 
\eta_{1}}/8}
\Ebb{I^{-a}_{\widetilde{\eta}}(1)}
=:e^{-c_3  at} c_4,
\]
where $\widetilde{\eta}$ is an independent
copy of $\eta$ and $c_3,c_4 > 0$. We note that the $c_4$ is well-defined, i.e. the last expectation
is finite, 
by \cite[(2.36)]{PatieSavov2018}. Therefore,
\begin{equation}\label{eq:case1int}
    \begin{split}
        \int_0^{c_2 \ln t}&
        \Ebb{\lbrb{
        I_{\jump{\xi}{1}}(s) + e^{-\jump{\xi_s}{1} - 
        \jump{\Delta_s}{\epsilon t}}
        I_{\jump{\eta}{\epsilon t}}
        (t-s)}^{-a}\indset{A_t}}
        \D s \\&\hspace{10em}
        \leq
        \int_0^{c_2 \ln t}
        \Ebb{\lbrb{
        e^{-\jump{\xi_s}{1} - 
        \jump{\Delta_s}{\epsilon t}}
        I_{\jump{\eta}{\epsilon t}}
        (t-s)}^{-a}\indset{A_t}}
        \D s 
        \\&\hspace{10em}
        \leq
        e^{-c_3 at}c_4 \int_0^{c_2 \ln t}
        \Ebb{\lbrb{e^{-\xi^{(1)}_s  - 
        \epsilon t}}^{-a}}
        \D s \\
        &\hspace{10em}:=
        c_4 e^{-c_5 t}\int_0^{c_2 \ln t}
        \Ebb{e^{a\xi^{(1)}_s}} \D s
    \end{split}
\end{equation}
for $c_5 := a(c_3 - \epsilon)$, which is positive
for sufficiently small $\epsilon$ and where we have estimated $\jump{\Delta}{\epsilon t}_s\leq \epsilon t$. As for the last integral
above,
because the \LL measure of  $\jump{\xi}{1}$ is supported on $(-\infty, 1]$, by \cite[Theorem 25.3]{Sato1999}, one can see that
the characteristic function of $\xi^{(1)}$ in any $a\geq0$,
$\Psi^{(1)}(a)$, is finite, so for large $t$,
\[
\int_0^{c_2 \ln t}
        \Ebb{e^{a\xi^{(1)}_s}} \D s
        =
        \int_0^{c_2 \ln t}
        e^{s\Psi^{(1)}(a)} \D s
        \leq
        c_2 \ln t
        e^{c_2 \ln t \labsrabs{\Psi^{(1)}(a)}},
\]
so the behaviour of the last expression in \eqref{eq:case1int}
is dominated by
$e^{-c_5 t}$, and thus is $\bo{\ell(t)/t^\alpha}$.

\textbf{Case 2:} 
Under $A_t^c:= \left\{\eta^{(\epsilon t)}_{t/4} > t\Ebb{ 
\jump{\eta_1}{\epsilon t}} /8\right\}
$. 

First, let us bound the expectation from \eqref{eq:earlyJ_11}, conditioned
on this case:
using once again \eqref{eq:bounds_small_jumps}, we calculate that
 \begin{equation}\label{eq:case2int}
    \begin{split}
    \int_0^{c_2 \ln t}
&\Ebb{\indset{A_t^c}\lbrb{
        I_{\jump{\xi}{1}}(s) + e^{-\jump{\xi_s}{1} - 
        \jump{\Delta_s}{\epsilon t}}
        I_{\jump{\eta}{\epsilon t}}
        (t-s)}^{-a}}
        \D s \\
&= \int_0^{c_2 \ln t}
        \P(A_t^c)\Ebb{\lbrb{
        I_{\jump{\xi}{1}}(s) + e^{-\jump{\xi_s}{1} - 
        \jump{\Delta_s}{\epsilon t}}
        I_{\jump{\eta}{\epsilon t}}
        (t-s)}^{-a}\middle| A_t^c} 
        \D s \\
&\leq
        \P(A_t^c)
        \int_0^{\infty}
        \Ebb{I^{-a}_{\jump{\xi}{1}}(s)}
        \D s
        \leq
        \P(A_t^c)
        \int_0^{\infty}
        \Ebb{I_\xi^{-a}(s)} \D s.
    \end{split}
\end{equation}
Therefore, because, 
as already noted, the last integral is finite from \eqref{eq:LT}, 
a bound of 
$\P(A_t^c)$ of order $\ell(t)/t^\alpha$ would be enough to conclude. We will obtain such a bound via a Chernoff type argument. For $\lambda > 0$, by Markov's inequality for $x\mapsto e^{\lambda x}$ and using
the \LLK formula, we have
\begin{equation}
    \label{eq:Chernoff}
\P(A_t^c) 
=
\P\lbrb{\eta^{(\epsilon t)}_{t/4} > t\Ebb{ 
\jump{\eta_1}{\epsilon t}} / 8}
\leq \exp\lbrb{t\jump{\Psi}{\epsilon t}(\lambda)/4 - \lambda t\Ebb{ 
\jump{\eta_1}{\epsilon t}}/ 8},
\end{equation}
where we recall, see \eqref{eq:LLK}, that
 \begin{equation}\label{eq:Psi_et}
    \begin{split}
        \jump{\Psi}{\epsilon t}(\lambda)=
        \lambda \Ebb{ 
\jump{\eta_1}{\epsilon t}}
+\frac{1}{2}\lambda^2 \sigma^2 + \int_{(-\infty,\epsilon t]} \lbrb{e^{\lambda x} - 1 - \lambda x} \Pi(\D x).
    \end{split}
\end{equation}
We would choose $\lambda$ to depend on $t$ in $\eqref{eq:Chernoff}$ as
\[\lambda_t := \frac{\ln t^{(\alpha-1-\epsilon)/4 }}{\epsilon t},
\]
and we will work with $0<\epsilon<\alpha-1$. For better readability, we will not use the explicit 
expression of $\lambda_t$ if not needed. What would be
of vital importance is that
$\lambda_t \to 0$ when $t \to \infty$ and that the highest order terms in
\eqref{eq:Psi_et} would be of order $\lambda_t$. 

Let us start with bounding the integral in \eqref{eq:Psi_et}: since 
$\Ebb{ \xi_1} = \Ebb{ \eta_1}$ is finite, we know that $\int_{-\infty}^{-1} x \Pi(x) 
\D x$ is finite, see e.g. \cite[Corollary 25.8]{Sato1999},
so we can 
choose $K_1<-1$ 
such 
that 
\[-\int_{(-\infty, K_1]} x\Pi(\D x)
\leq -\Ebb{ 
\eta_1}/4\leq -\Ebb{ 
\jump{\eta_1}{\epsilon t}}/4.\] Next, split the integral from \eqref{eq:Psi_et} into three regions:
\[
\int_{(-\infty,\epsilon t]} \lbrb{e^{\lambda_t x} - 1 - \lambda_t x} \Pi(\D x)= 
\int_{(-\infty, K_1]\cup (K_1, 1] \cup
(1, \epsilon t]}
\lbrb{e^{\lambda_t x} - 1 - \lambda_t x}
\Pi(\D x).
\]
Over the first region, by the choice of $K_1<-1$ and since $e^{\lambda_t x} - 1 \leq 0$
for $x < 0$ and large $t$,
\begin{equation}
    \label{eq:Int_1}
\int_{(-\infty,K_1]} \lbrb{e^{\lambda_t x} - 1 - \lambda_t x }\Pi(\D x) \leq-
\lambda_t\Ebb{ 
\jump{\eta_1}{\epsilon t}}/{4}.
\end{equation}
For bounding over the second region,
note that because $e^x - 1 - x \simo x^2/2$,
it is true that for sufficiently
small $x$, $e^x - 1 - x \leq x^2$. Therefore, as $\lambda_t \to 0$
and $K_1$ is fixed, for large $t$,
\begin{equation}
    \label{eq:Int_2}\begin{split}
\int_{(K_1,1]} \lbrb{e^{\lambda_t x} - 1 - \lambda_t x } \Pi(\D x)
&\leq \lambda_t^2 \int_{(K_1, -1] \cup
(-1,1]}x^2 \Pi(\D x)\\
&\leq \lambda_t^2 K_1^2\int_{(K_1,1]}\min\lbcurlyrbcurly{1,
 x^2 }\Pi (\D x)= \bo{ \lambda_t^2 }= \so{1/t}
    \end{split}
\end{equation}
as by construction $\lambda_t = \bo{\ln t / t}$.
Over the last region, after integration by parts,
\begin{equation}\label{eq:Int_by_parts}
    \begin{split}
&\int_{(1,\epsilon t]} \lbrb{e^{\lambda_t x} - 1 - \lambda_t x } \Pi(\D x)
=
\lambda_t\int_{(1,\epsilon t]} \int_0^x \lbrb{e^{\lambda_t y} - 1} \D  y \Pi(\D x)
\\
&\,= \lambda_t \int_0^1\int_{(1,\epsilon t]}\lbrb{e^{\lambda_t y} - 1} \Pi(\D x) \D  y 
+\lambda_t \int_1^{\epsilon t} \int_{(y,\epsilon t]}\lbrb{e^{\lambda_t y} - 1} \Pi(\D x) \D  y 
\\
&\,=
\lambda_t\int_{0}^{1}\lbrb{e^{\lambda_t y} - 1}\lbrb{
\PiPlus(1) - \PiPlus(\epsilon t) 
} \D  y +
\lambda_t\int_{1}^{\epsilon t}\lbrb{e^{\lambda_t y} - 1}\lbrb{
\PiPlus(y) - \PiPlus(\epsilon t)  
} \D  y.
    \end{split}
\end{equation}
Similarly to a previous argument, as $e^{y} - 1 \simo y$, for large $t$ we have
\begin{equation}
    \label{eq:Int_3_1}
    \begin{split}\lambda_t\int_{0}^{1}\lbrb{e^{\lambda_t y} - 1}\lbrb{
\PiPlus(1) - \PiPlus(\epsilon t) 
} \D  y
&\leq
\lambda_t\PiPlus(1) \int_0^1 2\lambda_t  y\D y\\
&
= \lambda_t^2 \PiPlus(1) = \bo{ \lambda_t^2 }= \so{1/t}
\end{split}
\end{equation}
so our final step would be to bound the last integral from \eqref{eq:Int_by_parts}.
As previously noted,
from
\eqref{eq: Pibar},
there exists $C^\ast >0$
such that,
for large $t$, $\PiPlus(t) \leq
C^\ast\P(\xi_1 > t) \leq 2C^\ast \ell(t)/t^\alpha
\leq 1/t^{\alpha - \epsilon}$ and we
recall that $\alpha - \epsilon > 1$. Let us fix
 $K_2 > 1$ such that for
$t > K_2$, the last inequality holds, i.e. $\PiPlus(t) \leq 1/t^{\alpha - \epsilon}$. Therefore, for large $t$, because
$\sqrt{t}\lambda_t = \so{1}$,
 \begin{equation}\label{eq:Int_3_2}
    \begin{split}
\lambda_t& \int_1^{\epsilon t}\lbrb{e^{\lambda_t y} - 1}\lbrb{
\PiPlus(y) - \PiPlus(\epsilon t)  
} \D  y  \\
&\quad
=
\lambda_t\int_{(1,K_2]
\cup (K_2, 
\sqrt{t}]
\cup (\sqrt{t}, \epsilon t]}
\lbrb{e^{\lambda_t y} - 1}\lbrb{
\PiPlus(y) - \PiPlus(\epsilon t)  
} \D  y \\
&\quad
\leq
\lambda_t \lbrb{\PiPlus(1)  \int_{1}^{K_2}
2 \lambda_t y \D y
+  \int_{K_2}^{\sqrt{t}}2 \lambda_t y \frac{1}{y^{\alpha
-\epsilon}} \D y +  e^{\lambda_t \epsilon t} \int_
{\sqrt{t}} ^ \infty \frac{1}{y^{\alpha -
\epsilon}}\D y }\\
&\quad
= \lambda_t\lbrb{c_6 \lambda_t + c_7 \lambda_t t^{(\epsilon - \alpha + 2)/2}
+ c_8 e^{\lambda_t \epsilon t} t^{(\epsilon -\alpha + 1)/2}}
\\
&\quad
=\lambda_t\lbrb{c_6 \lambda_t + c_{9} t^{(\epsilon - \alpha)/2} \ln t
+ c_8  t^{(\epsilon -\alpha + 1)/4}} = \so{1/t}
    \end{split}
\end{equation}
where we have substituted that $\lambda_t = \ln t^{(\alpha-1-\epsilon)/4}/(\epsilon t)$ and
implicitly have defined $c_6, c_7, c_8, c_{9} \geq 0$.
Using \eqref{eq:Int_3_1} and \eqref{eq:Int_3_2} in \eqref{eq:Int_by_parts},
we get
\begin{equation}
    \label{eq:Int_3}
\int_{(1,\epsilon t]} \lbrb{e^{\lambda_t x} - 1 - \lambda_t x }\Pi(\D x)
=\so{1/t}
\end{equation}
Substituting the bounds \eqref{eq:Int_1},
\eqref{eq:Int_2} and \eqref{eq:Int_3} into
\eqref{eq:Psi_et}, we obtain
\[
\Psi^{(et)}(\lambda_t) \leq 
3\lambda_t\Ebb{ \jump{\eta_1}{\epsilon t}}/{4} +
\so{1/t}
\]
and plugging this into \eqref{eq:Chernoff},
along with $\lambda_t = \ln t^{(\alpha-1-\epsilon)/4}/(\epsilon t)$, we can see that for large $t$
\begin{equation*}
    \begin{split}
\P(A_t^c) &\leq 
\exp\lbrb{
\frac{1}{16}\lambda_t t \Ebb{ 
\jump{\eta_1}{\epsilon t}}+ \so{1}} \\
&\leq
\exp\lbrb{
\frac{1}{32}\lambda_t t \Ebb{ 
\eta_1}} 
=
\exp\lbrb{\frac{(\alpha-1 - \epsilon)\Ebb{ 
\eta_{1}}\ln t}{32\epsilon}}.  
    \end{split}
\end{equation*}
Choosing $\epsilon$ sufficiently small,
explicitly, $\epsilon < (1-\alpha)\Ebb{\eta_1}/(32(\alpha+1)-\Ebb{\eta_1}$), we get that, for large $t$,
\[
\P(A_t^c) \leq 
\exp\lbrb{-(\alpha+1) \ln(t)} = 1/t^{\alpha + 1}
= \bo{\ell(t)/t^\alpha} = \bo{\P\lbrb{\xi_1 > t}},
\]
which concludes the proof that
$\Ebb{I_\xi^{-a}(t)\ind{\jump{J_1}{\epsilon t}>t}}
= \Ebb{I_{\jump{\xi}{\epsilon t}}^{-a}(t)}$ is $\bo{\ell(t)/t^{\alpha}}= \bo{\P\lbrb{\xi_1 > t}}$. Using this in \eqref{eq:E_split_Jet} along with
\eqref{eq:E_large_jump}, we obtain the statement of the lemma, see \eqref{eq: lemma upper bound}.
\end{proof}


\appendix
\section*{Appendix A. Some facts from regular variation theory}\label{appn} 
\label{sec:append}
\renewcommand{\thesection}{A} 

In this appendix, we provide
basic facts about the classical
Karamata theory of
slowly varying functions and
an extension of it over
a class of functions, 
named after de Haan. The presentation
is based on, respectively, Chapter I and Chapter III of the extensive
work \cite{BinGolTeu1989}.
\subsection{Classical Karamata theory}
\begin{definition}Let $\ell$ be a positive measurable
function, defined on some neighbourhood
$[X, \infty)$
of infinity, and satisfying for all
$\lambda > 0$,
\[
\lim_{x\to \infty}\frac{\ell(\lambda x)}{\ell(x)} = 1,
\]
then $\ell$ is said to be
slowly varying at infinity. We note
the class of these functions as $SV_\infty$. In a similar manner, if
$\ell$ is defined in a
positive/negative neighbourhood of
$0$ and satisfying
\[
\lim_{x\to 0+/0-}\frac{\ell(\lambda x)}{\ell(x)} = 1,
\]
$\ell$ is said to be slowly varying
at $0+$/$0-$. We note the class of these
functions as $SV_0$.
\end{definition}
Basic properties, such as
comparison with polynomials and
linear properties, can be found in
\cite[Proposition 1.3.6]{BinGolTeu1989}.

\begin{theorem}[Uniform Convergence Theorem,
{\cite[Theorem 1.2.1]{BinGolTeu1989}}]
\label{thm:UCT_1.2.1}
If $\ell \in SV_\infty$, then for all $\lambda > 0$,
\begin{equation}
    \label{eq:UCT}
\lim_{x\to \infty}\frac{\ell(\lambda x)}{\ell(x)} = 1,
\text{ uniformly on each compact set in $(0,\infty)$}.
\end{equation}
\end{theorem}
As we will often use the fact that slowly varying functions grow slower than polynomials, we state the following easy
consequence of
Potter's theorem, see
{\cite[Theorem 1.5.6]{BinGolTeu1989}}.

\begin{proposition}
    \label{prop:Potter's bounds}
    If $\ell \in SV_\infty$, then for 
    any chosen constants
    $A > 1$, $\delta >0$,
    there exist $X=X(A, \delta)$ such that
    for all $x \geq X$,
    \begin{equation}
    \label{eq:Potter's bounds}
    \ell(x) \leq A x^\delta \quad
    \text{and}
    \quad
    \ell(x) \geq A x^{-\delta}.
    \end{equation}
\end{proposition}

\begin{remark}
    If $\ell \in SV_0$, applying Theorem
    \ref{thm:UCT_1.2.1} for $\ell'(x) = \ell(1/x) \in SV_\infty$, we can see 
    that \eqref{eq:UCT} holds after the modification $x \to 0$.
    Similarly, \eqref{eq:Potter's bounds} leads to
    $\ell(x) = \bo{x^{-\delta}}$ for any $\d >0$ as $x \to 0$.
\end{remark}

\begin{theorem}[Karamata's Theorem,
{\cite[Proposition 1.5.8]{BinGolTeu1989}}]
\label{thm:Karamata_1.5.8.}
If $\ell \in SV_\infty$, $X$ is so large that
$\ell(x)$ is locally bounded in $[X, \infty)$, and
$\alpha > -1$, then
\[
\int_X^x t^\alpha \ell(t) \D t \simi x^{\alpha + 1}\ell(x)/(\alpha +1).
\]
\end{theorem}
\begin{theorem}[Karamata's Theorem,
{\cite[Proposition 1.5.10]{BinGolTeu1989}}]\label{thm:Karamata_1.5.10.}
  If $\ell$ is a slowly varying function  and $\alpha<-1$ then $\int^\infty t^\alpha \ell(t)\D t$ converges and 
  \begin{equation*}
      \limi{x}\frac{x^{\alpha+1}\ell(x)}{\int_x^\infty t^\alpha \ell(t)\D t}=-\alpha-1.
  \end{equation*}
\end{theorem}
\begin{theorem}[Karamata Tauberian Theorem,
{\cite[Theorem 1.7.1]{BinGolTeu1989}}]
\label{thm:Tauberian_1.7.1}
Let $U$ be a non-decreasing right-continuous
function on $\R$ with $U(x) = 0$ for
all $x < 0$. If $\ell \in SV_\infty$
and $c \geq 0, \rho \geq 0$,
the following are equivalent:
\begin{align}
    \label{eq:Taub_Measure}
    U(x) &\simi c x^\rho \ell(x)/\Gamma(1 + \rho),\\
    \label{eq:Taub_Laplace}
    \IntOIc e^{-sx} \D U(x) =:
    \widehat{U}(s) &\simo cs^{-\rho}\ell(1/s).
\end{align}
When $c=0$, \eqref{eq:Taub_Measure} is to be
interpreted as $U(x) = \so{x^\rho\ell(x)}$; similarly
for \eqref{eq:Taub_Laplace}.
\end{theorem}

\begin{theorem}[extended Karamata Tauberian Theorem,
{\cite[Theorem 1.7.6]{BinGolTeu1989}}]
\label{thm:extended Tauberian, 1.7.6}
Assume $U(\cdot) \geq 0, c\geq0, \rho>-1, s\widehat{U}(s):=
s\IntOI e^{-sx}U(x) \D x$ convergent for $s>0$, and $\ell \in SV_\infty$.
Then, under the condition
\begin{equation}
    \label{eq:slow decrease condtion}
     \lim_{\lambda \to 1+} \liminfi{x} \inf_{1\leq t \leq \lambda}
 \frac{U(tx) - U(x)} {x^\rho \ell(x)} \geq 0,
    \end{equation}
it is true that
\[
 s\widehat{U}(s) \simo cs^{-\rho}\ell(1/s)\quad
 \text{and}
 \quad
\]
implies
\[
U(x) \simi cx^\rho \ell(x)/\Gamma(1+\rho).
\]    
\end{theorem}
\begin{remark}
    The reverse implication in the last theorem is true without 
    the assumption \eqref{eq:slow decrease condtion}.
\end{remark}

\begin{theorem}[Monotone Denstity theorem,
{\cite[Theorems 1.7.2 - 2']{BinGolTeu1989}}]
\label{thm:Monotone density, 1.7.2}
Let $U(x) = \int_0^x u(y) \D y$. If $U(x) \simi c x^\rho \ell(x)$,
where $c \in \R$, $\rho \in \R$, $\ell \in SV_\infty$,
and if $u$ is ultimately monotone, then
\[
u(x) \simi c\rho x^{\rho - 1} \ell(x).
\]
Let $V(x)=\int_{x}^\infty v(x)dx$ with $v$ integrable at infinity. Let $V(x) \simi c x^\rho \ell(x)$ for $\rho < 0$ and let $v$ be ultimately non-increasing. Then
\[v(x)\simi -c\rho x^{\rho - 1} \ell(x).\]
\end{theorem}

\subsection{de Haan theory}

We begin with a motivating example:
let $\ell \in SV_\infty$ and choose $X$ so that 
$\ell \in L^1_\text{loc}[X, \infty)$.
In Karamata's theorem \ref{thm:Karamata_1.5.8.}, the situation where $\alpha = -1$, i.e. when
considering the
asymptotic behaviour of
\[
m(x) := \int_X^x\ell(t) \D t/t,
\]represents a special boundary case.
We have the following result.
\begin{proposition}[{\cite[Proposition 1.5.9a]{BinGolTeu1989}}]
\label{prop: de Haan 1.5.9a}
The function $m \in SV_\infty$ and
\[
\lim_{x \to \infty} \frac{m(x)}{\ell(x)}=
\infty.\]
\end{proposition}
However, it turns out that
there exist more precise links
between $m$ and $\ell$.
\begin{proposition}[{\cite[p. 127]{BinGolTeu1989}}]
\label{prop: de Haan log p.127}
For any $\lambda > 0$,
\[
\frac{1}{\ell(x)}\int_x^{\lambda x}\ell(t) \D t/t = \lim_{x \to \infty} \frac{m(\lambda x) - m(x)}{\ell(x)}
= \ln \lambda.\]
\end{proposition}
The class of all functions which
exhibit similar behaviour to $m$,
constitutes the de Haan class $\Pi_\ell$.
\begin{definition}
\label{def: Pi_ell}
    We say that a measaruble
    function $f$ is in the
    de Haan class $\Pi_\ell$ if for
    any $\lambda >0$,
    \[
 \lim_{x \to \infty} \frac{f(\lambda x) - f(x)}{\ell(x)}
= c\ln \lambda,\]
with some non-zero constant $c$,
called the $\ell$-index of $f$.
\end{definition}
The following theorem from \cite{BinGolTeu1989} refers the work of de Haan, \cite{deHaan1976}, as its backbone. It characterises the class $\Pi_\ell$ via
a monotone density representation.
\begin{theorem}[{\cite[Theorem 3.6.8]{BinGolTeu1989}}]
    \label{thm:de Haan monotone density, 3.6.8}
    Let $U(x) = U(X) + \int_X^x u(t) \D t$, where $u$ is
    monotone on some neighbourhood of infinity,
    and let $\ell \in SV_\infty$. Then $U \in \Pi_\ell$
    with $\ell$-index $c$ if and only if
    \[
    u(x) \simi cx^{-1}\ell(x).
    \]
    
\end{theorem}
        
    \addcontentsline{toc}{section}{Acknowledgments}
    \section*{Acknowledgments}
This study is financed by the European Union's NextGenerationEU, through the National Recovery and Resilience Plan of the Republic of Bulgaria, project No BG-RRP-2.004-0008. The authors are thankful to Pierre Patie for his usual unusual resourcefulness in spotting probabilistic relations as in Theorem \ref{thm:mainPr}.  We also thank the anonymous reviewers for their thoughtful questions, comments, and suggestions, which led to improvements throughout the paper, in particular motivating Remarks \ref{rem: factor} and \ref{rem: DoneyMaller} and improving the presentation.

    \addcontentsline{toc}{section}{References}

    \bibliographystyle{plain}
	\bibliography{bibliography.bib}

\begin{thebibliography}{10}

\bibitem{AliJedRiv14}
L.~Alili, W.~Jedidi, and V.~Rivero.
\newblock On exponential functionals, harmonic potential measures and
  undershoots of subordinators.
\newblock {\em ALEA Lat. Am. J. Probab. Math. Stat.}, 11(1):711--735, 2014.

\bibitem{AlMatSh01}
L.~Alili, H.~Matsumoto, and T.~Shiraishi.
\newblock On a triplet of exponential {B}rownian functionals.
\newblock In {\em S\'{e}minaire de {P}robabilit\'{e}s, {XXXV}}, volume 1755 of
  {\em Lecture Notes in Math.}, pages 396--415. Springer, Berlin, 2001.

\bibitem{ArRiv23}
J.~Arista and V.~Rivero.
\newblock Implicit renewal theory for exponential functionals of {L}\'{e}vy
  processes.
\newblock {\em Stochastic Process. Appl.}, 163:262--287, 2023.

\bibitem{BanParSm21}
V.~Bansaye, J.C. Pardo, and C.~Smadi.
\newblock Extinction rate of continuous state branching processes in critical
  {L}{\'{e}}vy environments.
\newblock {\em {ESAIM}: Probability and Statistics}, 25:346--375, 2021.

\bibitem{BarkerSavov2021}
A.~Barker and M.~Savov.
\newblock Bivariate {B}ernstein-gamma functions and moments of exponential
  functionals of subordinators.
\newblock {\em Stochastic Process. Appl.}, 131:454--497, 2021.

\bibitem{Bertoin96}
J.~Bertoin.
\newblock {\em L\'{e}vy processes}, volume 121 of {\em Cambridge Tracts in
  Mathematics}.
\newblock Cambridge University Press, Cambridge, 1996.

\bibitem{BerYor05}
J.~Bertoin and M.~Yor.
\newblock Exponential functionals of {L}\'{e}vy processes.
\newblock {\em Probab. Surv.}, 2:191--212, 2005.

\bibitem{BinGolTeu1989}
N.~H. Bingham, C.~M. Goldie, and J.~L. Teugels.
\newblock {\em Regular variation}, volume~27 of {\em Encyclopedia of
  Mathematics and its Applications}.
\newblock Cambridge University Press, Cambridge, 1989.

\bibitem{Bogachev2007}
V.~I. Bogachev.
\newblock {\em Measure theory. {V}ol. {II}}.
\newblock Springer-Verlag, Berlin, 2007.

\bibitem{deHaan1976}
L.~de~Haan.
\newblock An {A}bel-{T}auber theorem for {L}aplace transforms.
\newblock {\em J. London Math. Soc. (2)}, 13(3):537--542, 1976.

\bibitem{Doney2007}
R.~A. Doney.
\newblock {\em Fluctuation theory for {L}\'{e}vy processes}, volume 1897 of
  {\em Lecture Notes in Mathematics}.
\newblock Springer, Berlin, 2007.

\bibitem{Doney-Jones-2012}
R.~A. Doney and E.~M. Jones.
\newblock Conditioned random walks and {L}\'{e}vy processes.
\newblock {\em Bull. Lond. Math. Soc.}, 44(1):139--150, 2012.

\bibitem{DonMal04}
R.~A. Doney and R.~A. Maller.
\newblock Moments of passage times for {L}\'{e}vy processes.
\newblock {\em Ann. Inst. H. Poincar\'{e} Probab. Statist.}, 40(3):279--297,
  2004.

\bibitem{Doring-Savov-Trottner-Watson-24}
L.~D\"{o}ring, M.~Savov, L.~Trottner, and A.~Watson.
\newblock The uniqueness of the {W}iener-{H}opf factorisation of {L}\'{e}vy
  processes and random walks.
\newblock {\em Bull. Lond. Math. Soc.}, 56(9):2951--2968, 2024.

\bibitem{Doring-Trottner-Watson-2024}
L.~D\"{o}ring, L.~Trottner, and A.~Watson.
\newblock Markov additive friendships.
\newblock {\em Trans. Amer. Math. Soc.}, 377(11):7699--7752, 2024.

\bibitem{Kallenberg-17}
O.~Kallenberg.
\newblock {\em Random measures, theory and applications}, volume~77 of {\em
  Probability Theory and Stochastic Modelling}.
\newblock Springer, Cham, 2017.

\bibitem{Kluppelberg-Kyprianou-Maller-2004}
C.~Kl\"{u}ppelberg, A.~E. Kyprianou, and R.~A. Maller.
\newblock Ruin probabilities and overshoots for general {L}\'{e}vy insurance
  risk processes.
\newblock {\em Ann. Appl. Probab.}, 14(4):1766--1801, 2004.

\bibitem{Knopova-Schilling-13}
V.~Knopova and R.~Schilling.
\newblock A note on the existence of transition probability densities of
  {L}\'{e}vy processes.
\newblock {\em Forum Math.}, 25(1):125--149, 2013.

\bibitem{KuzPar12}
A.~Kuznetsov and J.~C. Pardo.
\newblock Fluctuations of stable processes and exponential functionals of
  hypergeometric {L}\'{e}vy processes.
\newblock {\em Acta Appl. Math.}, 123:113--139, 2013.

\bibitem{Kyprianou-Rivero-Sengul-Yang-2020}
A.~Kyprianou, V.~Rivero, B.~Seng\"{u}l, and T.~Yang.
\newblock Entrance laws at the origin of self-similar {M}arkov processes in
  high dimensions.
\newblock {\em Trans. Amer. Math. Soc.}, 373(9):6227--6299, 2020.

\bibitem{LiXu18}
Z.~Li and W.~Xu.
\newblock Asymptotic results for exponential functionals of {L}\'{e}vy
  processes.
\newblock {\em Stochastic Process. Appl.}, 128(1):108--131, 2018.

\bibitem{LoePatSav19}
R.~Loeffen, P.~Patie, and M.~Savov.
\newblock Extinction time of non-{M}arkovian self-similar processes,
  persistence, annihilation of jumps and the {F}r\'{e}chet distribution.
\newblock {\em J. Stat. Phys.}, 175(5):1022--1041, 2019.

\bibitem{MatYor05}
H.~Matsumoto and M.~Yor.
\newblock Exponential functionals of {B}rownian motion. {I}. {P}robability laws
  at fixed time.
\newblock {\em Probab. Surv.}, 2:312--347, 2005.

\bibitem{Maulik-Zwart-06}
K.~Maulik and B.~Zwart.
\newblock Tail asymptotics for exponential functionals of {L}\'{e}vy processes.
\newblock {\em Stochastic Process. Appl.}, 116(2):156--177, 2006.

\bibitem{Minchev-2024}
M.~Minchev.
\newblock {\em Functionals of L{\'e}vy Processes and Their Applications}.
\newblock Phd thesis, Sofia University “St.\ Kliment Ohridski”, Sofia,
  Bulgaria, 2024.

\bibitem{MinSav23}
M.~Minchev and M.~Savov.
\newblock Asymptotics for densities of exponential functionals of
  subordinators.
\newblock {\em Bernoulli}, 29(4):3307--3333, 2023.

\bibitem{PalauPardoSmadi2016}
S.~Palau, J.C. Pardo, and C.~Smadi.
\newblock Asymptotic behaviour of exponential functionals of {L}\'{e}vy
  processes with applications to random processes in random environment.
\newblock {\em ALEA Lat. Am. J. Probab. Math. Stat.}, 13(2):1235--1258, 2016.

\bibitem{PalSarSav23}
Z.~Palmowski, H.~Sariev, and M.~Savov.
\newblock Moments of exponential functionals of {L}\'{e}vy processes on a
  deterministic horizon---identities and explicit expressions.
\newblock {\em Bernoulli}, 30(4):2547--2571, 2024.

\bibitem{PardoPatieSavov2012}
J.~C. Pardo, P.~Patie, and M.~Savov.
\newblock A {W}iener-{H}opf type factorization for the exponential functional
  of {L}\'{e}vy processes.
\newblock {\em J. Lond. Math. Soc. (2)}, 86(3):930--956, 2012.

\bibitem{Patie-2011}
P.~Patie.
\newblock A refined factorization of the exponential law.
\newblock {\em Bernoulli}, 17(2):814--826, 2011.

\bibitem{Pat_2012}
P.~Patie.
\newblock Law of the absorption time of some positive self-similar {M}arkov
  processes.
\newblock {\em Ann. Probab.}, 40(2):765--787, 2012.

\bibitem{PatieSavov2012}
P.~Patie and M.~Savov.
\newblock Extended factorizations of exponential functionals of {L}\'{e}vy
  processes.
\newblock {\em Electron. J. Probab.}, 17:1--22, 2012.

\bibitem{PatieSavov2013}
P.~Patie and M.~Savov.
\newblock Exponential functional of {L}\'{e}vy processes: generalized
  {W}eierstrass products and {W}iener-{H}opf factorization.
\newblock {\em C. R. Math. Acad. Sci. Paris}, 351(9-10):393--396, 2013.

\bibitem{PatSav17}
P.~Patie and M.~Savov.
\newblock Cauchy problem of the non-self-adjoint {G}auss-{L}aguerre semigroups
  and uniform bounds for generalized {L}aguerre polynomials.
\newblock {\em J. Spectr. Theory}, 7(3):797--846, 2017.

\bibitem{PatieSavov2018}
P.~Patie and M.~Savov.
\newblock Bernstein-gamma functions and exponential functionals of {L}\'{e}vy
  processes.
\newblock {\em Electron. J. Probab.}, 23:1--101, 2018.

\bibitem{PatieSavov2021}
P.~Patie and M.~Savov.
\newblock Spectral expansions of non-self-adjoint generalized {L}aguerre
  semigroups.
\newblock {\em Mem. Amer. Math. Soc.}, 272(1336):vii+182, 2021.

\bibitem{PatSav19}
P.~Patie, M.~Savov, and Y.~Zhao.
\newblock Intertwining, excursion theory and {K}rein theory of strings for
  non-self-adjoint {M}arkov semigroups.
\newblock {\em Ann. Probab.}, 47(5):3231--3277, 2019.

\bibitem{Revuz84}
D.~Revuz.
\newblock {\em Markov chains}, volume~11 of {\em North-Holland Mathematical
  Library}.
\newblock North-Holland Publishing Co., Amsterdam, second edition, 1984.

\bibitem{SalVos18}
P.~Salminen and L.~Vostrikova.
\newblock On exponential functionals of processes with independent increments.
\newblock {\em Teor. Veroyatn. Primen.}, 63(2):330--357, 2018.

\bibitem{Sato1999}
K.~Sato.
\newblock {\em L\'{e}vy processes and infinitely divisible distributions},
  volume~68 of {\em Cambridge Studies in Advanced Mathematics}.
\newblock Cambridge University Press, Cambridge, 1999.

\bibitem{Urban95}
K.~Urbanik.
\newblock Infinite divisibility of some functionals on stochastic processes.
\newblock {\em Probab. Math. Statist.}, 15:493--513, 1995.

\bibitem{Watanabe2008}
T.~Watanabe.
\newblock Convolution equivalence and distributions of random sums.
\newblock {\em Probab. Theory Related Fields}, 142(3-4):367--397, 2008.

\bibitem{Xu2021}
W.~Xu.
\newblock Asymptotic results for heavy-tailed {L}\'{e}vy processes and their
  exponential functionals.
\newblock {\em Bernoulli}, 27(4):2766--2803, 2021.

\bibitem{Xu_23}
W.~Xu.
\newblock Asymptotics for exponential functionals of random walks.
\newblock {\em Stochastic Process. Appl.}, 165:1--42, 2023.

\bibitem{Yamamuro1998}
K.~Yamamuro.
\newblock On transient {M}arkov processes of {O}rnstein-{U}hlenbeck type.
\newblock {\em Nagoya Math. J.}, 149:19--32, 1998.

\bibitem{Yor92}
M.~Yor.
\newblock On some exponential functionals of {B}rownian motion.
\newblock {\em Adv. in Appl. Probab.}, 24(3):509--531, 1992.

\bibitem{Yor01}
M.~Yor.
\newblock {\em Exponential functionals of {B}rownian motion and related
  processes}.
\newblock Springer Finance. Springer-Verlag, Berlin, 2001.

\end{thebibliography}
 
\end{document}